\documentclass[12pt,reqno]{amsart}
\usepackage{amsmath, amsthm, amssymb}
\usepackage{enumerate}
\usepackage{mathrsfs}
\usepackage{color}
\usepackage[english]{babel}
\usepackage{graphicx}
\usepackage{hyperref}
\usepackage{float}

\usepackage{multirow}
\usepackage{psfrag}
\usepackage[ansinew]{inputenc}

\def\({\left(}
\def\){\right)}

\def\N{{\mathbb{N}}}

\def\CC{{\mathbb{C}}}

\newcommand{\be}{\begin{equation}}
\newcommand{\ee}{\end{equation}}
\newcommand{\bea}{\begin{eqnarray}}
\newcommand{\eea}{\end{eqnarray}}
\newcommand{\beann}{\begin{eqnarray*}}
\newcommand{\eeann}{\end{eqnarray*}}

 \newtheorem{theorem}{Theorem}
\newtheorem{corollary}{Corollary}[section]
\newtheorem{remark}{Remark}[section]

\newtheorem{lemma}{Lemma}[section]

\newcommand{\bn}{\begin{eqnarray*}}       
\newcommand{\en}{\end{eqnarray*}}
\newcommand{\beq}{\begin{equation}}
\newcommand{\eeq}{\end{equation}}

\newcommand{\R}{\mathbb{R}}

\numberwithin{equation}{section}

\begin{document}

\title[Stability of black solitons under intensity-dependent dispersion]{\bf Stability of black solitons in optical systems with intensity-dependent dispersion}

\author{Dmitry E. Pelinovsky}
\address[D.E. Pelinovsky]{Department of Mathematics and Statistics, McMaster University,	Hamilton, Ontario, Canada, L8S 4K1}
\email{dmpeli@math.mcmaster.ca}

\author{Michael Plum}
\address[M. Plum]{Institut f\"{u}r Analysis, Karlsruher Institut f\"{u}r Technologie, Karlsruhe, Germany, 76131}
\email{michael.plum@kit.edu}

\begin{abstract}
Black solitons are identical in the nonlinear Schr\"{o}dinger (NLS) equation 
with intensity-dependent dispersion and the cubic defocusing NLS equation. We prove that the intensity-dependent dispersion introduces new properties in the stability analysis of the black soliton. First, the spectral stability problem possesses only isolated eigenvalues on the imaginary axis. Second, the energetic stability argument holds in Sobolev spaces with exponential weights. Third, the black soliton persists with respect to addition of a small decaying potential and remains spectrally stable when it is pinned to the minimum points of the effective potential. The same model exhibits a family of traveling dark solitons for every wave speed and we incorporate properties of these dark solitons for small wave speeds in the analysis of orbital stability of the black soliton. 
\end{abstract}

\date{\today}
\maketitle

\section{Introduction}

The canonical model of nonlinear optics for dark and black solitons
is the cubic defocusing NLS equation \cite{Kivshar,Kev-Dark-2015}, which can be written in the form 
\begin{equation}
\label{nls-cubic}
i u_t + u_{xx} + 2 (1-|u|^2) u = 0,
\end{equation}
where $u(t,x) : \mathbb{R} \times \mathbb{R}\mapsto \mathbb{C}$ is normalized by the boundary conditions
\begin{equation}
\label{bc-nonzero}
|u(t,x)| \to 1 \quad \mbox{\rm as} \;\; |x| \to \infty.
\end{equation}
Dark solitons are traveling wave solutions of the form 
\begin{equation}
\label{dark-cubic}
u(t,x) = \gamma \tanh(\gamma(x-2ct)) + i c, \quad \gamma := \sqrt{1-c^2},
\end{equation}
where $c \in (-1,1)$ is the parameter for the half wave speed. The limiting 
solution for $c = 0$, 
\begin{equation}
\label{black-cubic}
u(t,x) = \tanh(x), 
\end{equation} 
is referred to as the black soliton as the intensity $I := |u|^2$ reaches zero at a point of odd symmetry in $x$. The dark solitons (\ref{dark-cubic}) can be extended with two 
translational parameters due to the basic symmetries of the NLS equation (\ref{nls-cubic}): 
\begin{equation}
\label{symmetry}
\mbox{\rm If } u(t,x) \;\; \mbox{\rm is a solution, so is } 
e^{i \theta_0} u(t,x+\zeta_0), \quad \theta_0, \zeta_0 \in \mathbb{R}.
\end{equation}

The cubic defocusing NLS equation (\ref{nls-cubic}) is integrable with inverse scattering such that the dark solitons can be associated with eigenvalues of a self-adjoint linear (Dirac) operator \cite{ZakhShabat}. Orbital and asymptotic 
stabilities of the dark and black solitons were studied in many mathematical papers, e.g. with functional-analytic methods \cite{BGSS,BGSS14,GS}, the inverse scattering transform \cite{Cuccagna,PZ}, and higher-order conserved quantities \cite{GPII}. Lower-order conserved quantities were recently constructed in \cite{Koch,Koch-followup} to derive global solutions of the cubic NLS equation (\ref{nls-cubic}) in spaces of lower regularity.

Extended models for dark solitons were considered by using the generalized defocusing NLS equation, e.g. with the cubic--quintic and saturable nonlinearities \cite{PelKiv}. Many results were obtained in the mathematical analysis of the generalized NLS models such as the local and global existence of solutions to the initial-value problem in energy space \cite{Gallo,Gerard}, orbital stability of dark and black solitons \cite{Alejo,Chiron,Lin}, and persistence of the black soliton under a small exponentially decaying potential \cite{PelKev08}. 

\subsection{NLS models with intensity-dependent dispersion}

The objective of this work is to study orbital stability of the black soliton in a novel NLS model, which we take in the normalized form
\begin{equation}
\label{nls-idd}
i (1-|\psi|^2) \psi_t + \psi_{xx} = 0,
\end{equation}
where $\psi(t,x) : \mathbb{R} \times \mathbb{R}\mapsto \mathbb{C}$ is the wave function. We assume that $\psi(t,\cdot) \in \mathcal{F}$ for every $t \in \R$, where
\begin{equation}
\label{function-set}
\mathcal{F} := \left\{ f \in L^{\infty}(\R) : \quad |f(x)| < 1, \;\; x \in \mathbb{R} \quad \mbox{\rm and} \quad |f(x)| \to 1 \;\; {\rm as} \;\; |x| \to \infty \right\}.
\end{equation}
After the standing wave transformation 
$$
\psi(t,x) = e^{-2it} u(t,x), 
$$
the new model is equivalent 
to the modified cubic defocusing NLS equation in the form 
\begin{equation}
\label{nls}
i (1-|u|^2) u_t + u_{xx} + 2 (1-|u|^2) u = 0,
\end{equation}
which has the same solution (\ref{black-cubic}) for the black soliton. The black soliton can be extended with two translational parameters $\theta_0$ and $\zeta_0$ due to the same symmetries (\ref{symmetry}). In addition, it can be extended to the family of
traveling wave solutions with the wave speed $c$:
\begin{equation}
\label{soliton-parameters}
\psi(t,x) = e^{-2i t} U_c(x - 2c t),
\end{equation}
where $U_c(\xi)$ with $\xi := x - 2c t$ is a solution of the normalized equation 
\begin{equation}
\label{nls-trav}
U_c'' - 2i c (1-|U_c|^2) U_c' + 2 (1-|U_c|^2) U_c = 0.
\end{equation}
Dark solitons with the profile $U_c$ are solutions of the differential equation (\ref{nls-trav}) in the set $\mathcal{F}$ satisfying 
the boundary conditions 
\begin{equation}
\label{bc-dark}
U_c(\xi) \to e^{i \theta_{\pm}(c)} \quad \mbox{\rm as} \;\; \xi \to \pm \infty
\end{equation}
for some phases $\theta_{\pm}(c) \in [0,2\pi)$. If $c = 0$, then \begin{equation}
\label{bc-black}
\varphi(\xi) := U_{c=0}(\xi) = \tanh(\xi)
\end{equation} 
is the black soliton with $\theta_+(0) = 0$ and $\theta_-(0) = \pi$. In addition, equation (\ref{nls-idd}) is invariant under the scaling transformation:
\begin{equation}
\label{scaling-transform}
\mbox{\rm If } \psi(t,x) \;\; \mbox{\rm is a solution, so is } 
\psi(\omega^2 t,\omega  x), \quad \omega > 0.
\end{equation}

The starting model (\ref{nls-idd}) represents a class of NLS equations 
with intensity-dependent dispersion, which have been used for modeling of coherently prepared multistate atoms \cite{greentree}, quantum well waveguides \cite{koser}, and fiber-optics communication systems \cite{OL2020}. Bright solitons of the NLS model with normalized intensity-dependent dispersion, 
\begin{equation}
\label{nls-idd-bright}
i \psi_t +  (1-|\psi|^2)  \psi_{xx} = 0,
\end{equation}
were studied in \cite{PRK,RKP}. Bright solitons have cusped singularities at the unit intensity and the nature of singularities was incorporated in the study of existence and energetic stability of bright solitons in the energy space \cite{PRK,RKP}. Note that the energetic stability results were conditional to the existence of local solutions to the initial-value problem for the NLS model (\ref{nls-idd-bright}) in the energy space and the energy conservation, which was left as an open problem. 

\subsection{Motivations}

The NLS model (\ref{nls-idd}) with the inverse behavior of the normalized intensity-dependent dispersion compared to the NLS model (\ref{nls-idd-bright}) features both the dark and black solitons and is the subject of the present work. Similar to the scopes of \cite{PRK,RKP}, we will not address local well-posedness of the initial-value problem but will focus on the analysis of energetic stability of the black soliton and the novelty compared to the case of the cubic defocusing NLS equation (\ref{nls-cubic}).

{\em It is particularly interesting to see how the intensity-dependent dispersion introduces exponential weights in the construction of a Lyapunov functional for the black soliton. }

For the cubic defocusing NLS equation (\ref{nls-cubic}), it is well-known that the phase of complex perturbations to the black soliton cannot be controlled in the energy space due to nonzero boundary conditions. Related to this property, the second variation of the Lyapunov functional is not coercive in $H^1(\R)$ under the constraint of fixed ({\em renormalized}) momentum and the continuous spectrum of the linearized operator does not possess a spectral gap near the zero eigenvalue. To remedy the lack of coercivity, an exponentially weighted $H^1(\R)$ space was introduced in \cite{GS}, in which the orbital stability of the black soliton was obtained with the standard Lyapunov method. The phase of complex perturbations can be controlled in the weighted $H^1(\R)$ space. A generalization of the ideas of \cite{GS} was used in the weighted $H^2(\R)$ space in \cite{GPII} due to integrability of the cubic defocusing NLS equation (\ref{nls-cubic}).

The same method of weighted $H^1(\R)$ spaces was then used in \cite{CPP18} to study orbital stability of the domain walls which are minimizers of energy with nonzero boundary conditions at infinity \cite{ABCP}.  Recently, the same idea of the exponential weighted $H^1(\R)$ space was used in \cite{Alejo} to study orbital stability of the black soliton in the quintic defocusing NLS equation.

\subsection{Summary of results} 

As the main outcome of this work, we will show that the intensity-dependent dispersion gives a natural definition of the exponentially weighted $L^2(\R)$ space in which the second variation of the Lyapunov functional is coercive 
under the constraints of fixed momentum and mass. The exponentially weighted $H^1(\R)$ space appears as the form domain of the linearized operator. 

The orbital stability of the black soliton is established 
in the exponentially weighted $L^2(\R)$ space 
with the Lyapunov method that incorporates the additional scaling transformation (\ref{scaling-transform}). Furthermore, the spectrum of the linearized operator in the exponentially weighted $L^2(\R)$ space is shown to be purely discrete with a spectral gap near the zero eigenvalue. One important ingredient in the analysis 
of orbital stability of the black soliton is the asymptotic behavior of the dark 
solitons with the profile $U_c \in \mathcal{F}$ satisfying (\ref{nls-trav}) and (\ref{bc-dark}) as $c \to 0$. We establish the asymptotic behavior with arguments based on the implicit function theorem as the explicit expressions for $U_c$ are not available. 

As a consequence of the spectral gap near the zero eigenvalue, spectral stability of the black soliton can also be studied for the extended NLS model with a small decaying potential as in 
\begin{equation}
\label{nls-idd-potential}
i (1-|\psi|^2) \psi_t + \psi_{xx} = \varepsilon V(x) \psi,
\end{equation}
where $\varepsilon$ is a small parameter and $V$ is a fixed real-valued potential. By using Lyapunov--Schmidt reduction methods, 
we will show for $V \in W^{2,\infty}(\R) \cap L^2(\R)$ that the black soliton persists at the non-degenerate extremal points of the effective potential 
\begin{equation}
\label{potential-effective}
\mathcal{V}(s) := \int_{\R} V(x+s) {\rm sech}^2(x) dx.
\end{equation}
Moreover, we will show for $V \in W^{2,\infty}(\R) \cap L^1(\R)$ that the black soliton is linearly stable if the extremal point is a local minimum of $\mathcal{V}(s)$ and linearly unstable 
if it is a local maximum of $\mathcal{V}(s)$. This suggests robust pinning of the black soliton to the minima of the effective potential $\mathcal{V}(s)$. This outcome of the NLS model (\ref{nls-idd}) is a great practical improvement compared to the cubic NLS equation (\ref{nls-cubic}), 
where pinning of the black soliton is linearly unstable both at the maximum and minimum points of the effective potential $\mathcal{V}(s)$ \cite{PelKev08}.

\subsection{Organization of the manuscript} 

Conserved quantities of the NLS model (\ref{nls-idd}) are obtained in Section \ref{sec-formalism}. Main results are described in Section \ref{sec-results}. 
Linearized operators in the weighted $L^2(\R)$ space are studied in Section \ref{sec-spaces}. Existence of traveling dark solitons and their asymptotic behavior in the limit of $c \to 0$ are clarified in Section \ref{sec-existence}. 
Energetic stability of the black soliton is proven in Section \ref{sec-stability}. Persistence and spectral stability of the black soliton under the small decaying potential $V$ in the NLS model (\ref{nls-idd-potential}) are studied in Section \ref{sec-potential}. Section \ref{sec-conclusion} gives the summary and further directions of study. Appendices \ref{app-a} and \ref{app-b} contain proofs of technical results used in Sections \ref{sec-spaces} and \ref{sec-stability}, respectively.

\section{Conserved quantities}
\label{sec-formalism}

Assume that the NLS equation (\ref{nls-idd}) admits a smooth solution $\psi(t,\cdot) \in \mathcal{F} \cap C^{\infty}(\R)$ for some $t \in \R$ satisfying
\begin{equation}
\label{decay-psi}
\psi(t,x) = e^{i \theta_{\pm}} \left[ 1 + A_{\pm} e^{-\alpha_{\pm} |x|}
+ {\rm o}(e^{-\alpha_{\pm}|x|}) \right] \qquad 
\mbox{\rm as} \;\; x \to \pm \infty,
\end{equation}
for some real $\alpha_{\pm} > 0$, $A_{\pm} < 0$, and $\theta_{\pm} \in [0,2\pi)$ which may depend on time $t \in \R$. The asymptotic expansion (\ref{decay-psi}) is assumed to be differentiable term by term. Substituting (\ref{decay-psi}) into (\ref{nls-idd}) shows that $\theta_{\pm}$ changes in time as $\theta_{\pm}'(t) = -\frac{1}{2} \alpha_{\pm}(t)^2$.

The NLS model (\ref{nls-idd}) has formally conserved quantities
\begin{equation}
\label{cons}
M(\psi) := \int_{\R} (1-|\psi|^2)^2 dx \quad \mbox{\rm and} \quad  
E(\psi) := \int_{\R} |\psi_x|^2 dx,
\end{equation}
which have the meaning of mass and energy, respectively. This can be checked directly with 
\begin{align*}
\frac{d}{dt} M(\psi) &= -2 \int_{\R} (1-|\psi|^2) (\bar{\psi} \psi_t + \bar{\psi}_t \psi) dx  \\
&= -2i (\bar{\psi} \psi_{x} - \bar{\psi}_{x} \psi) |_{x \to -\infty}^{x \to +\infty} = 0,
\end{align*}
and
\begin{align*}
\frac{d}{dt} E(\psi) &= \int_{\R} (\bar{\psi}_x \psi_{xt} + \bar{\psi}_{xt} \psi_x) dx \\
&= i \frac{\bar{\psi}_x \psi_{xx} - \bar{\psi}_{xx} \psi_x}{1-|\psi|^2} 
|_{x \to -\infty}^{x \to +\infty} = 0,
\end{align*}
where in the last line we have used that $\psi(t,\cdot) \in \mathcal{F}  \cap C^{\infty}(\R)$ has 
the exponential decay (\ref{decay-psi}) which ensures that 
$\bar{\psi}_x \psi_{xx} - \bar{\psi}_{xx} \psi_x$ converges to zero  at infinity with the double exponential rate compared to $1-|\psi|^2$.

In addition, the NLS equation (\ref{nls-idd}) has formally conserved 
momentum 
\begin{equation}
\label{momentum}
P(\psi) = \frac{1}{2i} \int_{\R} \frac{(1-|\psi|^2)^2}{|\psi|^2} (\bar{\psi} \psi_x - \bar{\psi}_x \psi) dx,
\end{equation}
provided that $\psi(t,x) \neq 0$ everywhere. To confirm conservation of the momentum, we write $P = P_1 + P_2$, where 
\begin{equation}
\label{momentum-parts}
P_1(\psi) = \frac{i}{2} \int_{\R} (1-|\psi|^2) (\bar{\psi} \psi_x - \bar{\psi}_x \psi) dx, \quad P_2(\psi) = \frac{1}{2i} \int_{\R} (1-|\psi|^2) 
\left( \frac{\psi_x}{\psi} - \frac{\bar{\psi}_x}{\bar{\psi}} \right) dx,
\end{equation}
and obtain after straightforward computations,
\begin{align*}
\frac{d}{dt} P_1(\psi) &= -\frac{1}{2} \int_{\R} (\bar{\psi} \psi_{xxx} - \bar{\psi}_x \psi_{xx} - \bar{\psi}_{xx} \psi_x + \bar{\psi}_{xxx} \psi ) dx 
 -\int_{\R} \frac{|\psi|^2 (\bar{\psi}_x \psi_{xx} + \bar{\psi}_{xx} \psi_{x})}{1-|\psi|^2} dx\\
&= -\frac{1}{2} (\bar{\psi} \psi_{xx} - 2 |\psi_{x}|^2 + \bar{\psi}_{xx} \psi ) |_{x \to -\infty}^{x \to +\infty} 
-\int_{\R} \frac{|\psi|^2 (|\psi_{x}|^2)_x}{1-|\psi|^2} dx
\end{align*}
and 
\begin{align*}
\frac{d}{dt} P_2(\psi) &= \frac{1}{2} \int_{\R} \left( \frac{\psi_{xxx}}{\psi} - \frac{\psi_x \psi_{xx}}{\psi^2} - \frac{\bar{\psi}_x \bar{\psi}_{xx}}{\bar{\psi}^2} +\frac{\bar{\psi}_{xxx}}{\bar{\psi}} 
\right) dx 
+\int_{\R} \frac{(\bar{\psi}_x \psi_{xx} + \bar{\psi}_{xx} \psi_{x})}{1-|\psi|^2} dx\\
&= \frac{1}{2}  \left( \frac{\psi_{xx}}{\psi} +\frac{\bar{\psi}_{xx}}{\bar{\psi}} 
\right) |_{x \to -\infty}^{x \to +\infty} 
+\int_{\R} \frac{(|\psi_{x}|^2)_x}{1-|\psi|^2} dx,
\end{align*}
from which it follows that $P(\psi) = P_1(\psi) + P_2(\psi)$ conserves 
if $\psi(t,\cdot) \in \mathcal{F}  \cap C^{\infty}(\R)$ satisfies (\ref{decay-psi}) and $\psi(t,x) \neq 0$ everywhere.

\begin{remark}
	$P_1(\psi)$ and $P_2(\psi)$ in (\ref{momentum-parts}) generalize the conserved momentum and phase quantities of the cubic NLS equation (\ref{nls-cubic}) in the same way as $M(\psi)$ in (\ref{cons}) generalizes the conserved mass 
	of the cubic NLS equation. The renormalized momentum $P = P_1 + P_2$ is introduced in the cubic NLS equation (\ref{nls-cubic}) from the two conserved quantities as 
	the technical tool to prove orbital stability of dark solitons \cite{BGSS,BGSS14,Chiron,Lin}. In the context of the NLS model (\ref{nls-idd}), neither $P_1$ nor $P_2$ are conserved quantities but their sum is the conserved quantity bearing the same meaning as 
	the renormalized momentum of the cubic NLS equation (\ref{nls-cubic}). 
\end{remark}

\begin{remark}
Using the decomposition $P = P_1 + P_2$ given by (\ref{momentum-parts}), we can rewrite the momentum $P(\psi)$ in the form 
	\begin{align}
	\label{momentum-new}
	P(\psi) = \frac{i}{2} \int_{\R} (2 - |\psi|^2) (\bar{\psi} \psi_x - \bar{\psi}_x \psi) dx + \arg(\psi) |_{x \to -\infty}^{x \to +\infty}, 
	\end{align}
where the last term is $\theta_+ - \theta_-$ if $\psi(t,\cdot) \in \mathcal{F}\cap C^{\infty}(\R)$ satisfies (\ref{decay-psi}). The form (\ref{momentum-new}) is defined even if $\psi(t,x)$ vanishes at some $(t,x)$.
\end{remark}

The conservation of the momentum $P(\psi)$ in the form (\ref{momentum-new}) can be verified independently from the NLS model (\ref{nls-idd}) and the exponential decay (\ref{decay-psi}) which suggests $\theta_{\pm}'(t) = -\frac{1}{2} \alpha_{\pm}(t)^2$. Differentiating (\ref{momentum-new}) and using (\ref{nls-idd}) yield
\begin{align*}
\frac{d}{dt} P(\psi) &= -\frac{1}{2} \alpha_+^2 + \frac{1}{2} \alpha_-^2 + 
\frac{1}{2} \int_{\R} \frac{2 - |\psi|^2}{1 - |\psi|^2} \left( \bar{\psi}_x \psi_{xx} + \bar{\psi}_{xx} \psi_x - \bar{\psi} \psi_{xxx}- \bar{\psi}_{xxx} \psi \right) dx \\
& \quad - \frac{1}{2} \int_{\R} \frac{2 - |\psi|^2}{(1 - |\psi|^2)^2} ( \bar{\psi} \psi_{xx} + \bar{\psi}_{xx} \psi ) dx + \frac{1}{2} 
\int_{\R} \frac{1}{1 - |\psi|^2} ( \bar{\psi}  \psi_{xx} - \bar{\psi}_{xx} \psi  )  ( \bar{\psi}  \psi_{x} - \bar{\psi}_{x} \psi  ) dx \\
&= -\frac{1}{2} \alpha_+^2 + \frac{1}{2} \alpha_-^2 -
\frac{1}{2} \frac{2 - |\psi|^2}{1 - |\psi|^2} \left( \bar{\psi} \psi_{xx} + \bar{\psi}_{xx} \psi \right)  \biggr|_{x \to -\infty}^{x \to +\infty} 
+ 4 |\psi_x|^2 \biggr|_{x \to -\infty}^{x \to +\infty} = 0,
\end{align*}
where the exponential decay (\ref{decay-psi}) was used to get cancellation of terms at infinity. Hence, the momentum $P(\psi)$ in the form (\ref{momentum-new}) is also conserved in time $t$. Compared to (\ref{momentum}), the form (\ref{momentum-new}) does not require $\psi(t,x) \neq 0$ everywhere.

\section{Main results}
\label{sec-results}

We recall that $\psi(t,x) = e^{-2it} \tanh(x)$ is the standing wave 
of the NLS model (\ref{nls-idd}) corresponding to the black soliton. 
From the conserved quantities $M(\psi)$ and $E(\psi)$ in (\ref{cons}), 
we construct the Lyapunov functional for the black soliton in the form 
\begin{equation}
\label{Lyapunov}
\Lambda(\psi) := E(\psi) + M(\psi) = \int_{\mathbb{R}} \left[ |\psi_x|^2 + 
(1-|\psi|^2)^2 \right] dx. 
\end{equation}
Since it coincides with the Lyapunov functional for the cubic NLS equation 
(\ref{nls-cubic}), it is clear that the black soliton $\varphi(x) = \tanh(x)$ is a critical point of $\Lambda(\psi)$ in the energy space 
\begin{equation}
\label{energy-space}
\Sigma := \left\{ \psi \in H^1_{\rm loc}(\R) : \quad \psi_x \in L^2(\R), \;\; 1 - |\psi|^2 \in L^2(\R) \right\}, 
\end{equation}
equipped with the distance between two elements $\psi_1,\psi_2 \in \Sigma$:
\begin{equation*}
\mathcal{D}_{\Sigma}(\psi_1,\psi_2) := \sqrt{\| \psi_1' - \psi_2' \|^2_{L^2} + \| |\psi_1|^2 - |\psi_2|^2 \|^2_{L^2}}.
\end{equation*}

The critical point $\varphi \in \Sigma$ of $\Lambda(\psi)$ is degenerate only due to the translational and phase symmetries (\ref{symmetry}). We introduce the perturbation $u+iv \in H^1_{\rm loc}(\R)$ to $\varphi$ according to the expansion 
\begin{equation}
\label{perturbation}
\psi(t,x) = e^{-2it} \left[ \varphi(x) + u(t,x) + i v(t,x) \right].
\end{equation}
The expansion of the Lyapunov functional $\Lambda(\psi)$ is given by 
\begin{equation}
\label{expansion-Lambda}
\Lambda(\psi) - \Lambda(\varphi) = Q_+(u) + Q_-(v) + R(u,v), 
\end{equation}
where the quadratic forms are given by 
\begin{align}
\label{quad-plus}
Q_+(u) &= \int_{\mathbb{R}} \left[ (u_x)^2 + 2 (3 \varphi^2 - 1) u^2 \right] dx,\\
\label{quad-minus}
Q_-(v) &= \int_{\mathbb{R}} \left[ (v_x)^2 + 2 (\varphi^2 - 1) v^2 \right] dx,
\end{align}
and the remainder term is given by 
\begin{equation}
\label{remainder} 
R(u,v) = \int_{\R} \left[ (2\varphi u + u^2 + v^2)^2 - 4 \varphi^2 u^2 \right] dx.
\end{equation}
Expansion (\ref{expansion-Lambda}) is the same as in the cubic NLS equation (\ref{nls-cubic}), where the form of the remainder term (\ref{remainder}) 
uses the quantity
\begin{equation}
\label{eta-intro}
\eta := |\varphi + u + iv|^2 - \varphi^2 = 2 \varphi u + u^2 + v^2,
\end{equation}
which belongs to $L^2(\mathbb{R})$ if $\psi$ belongs to the energy space $\Sigma$. What appears to be new 
is the linearized equations which arise when the decomposition 
(\ref{perturbation}) is substituted to the NLS model (\ref{nls-idd}) and the quadratic and cubic terms with respect to perturbation $u + i v$ are crossed out. Separation of the real and imaginary parts gives the following linearized equations:
\begin{equation*}
(1-\varphi^2) u_t = L_- v, \qquad (1-\varphi^2) v_t = -L_+ u.
\end{equation*}
where the linear operators $L_{\pm}$ are defined by the differential expressions 
\begin{eqnarray}
\label{L-plus} 
L_+ &=& -\partial_x^2 + 6 \varphi^2 - 2, \\
\label{L-minus}
L_- &=& -\partial_x^2 + 2 \varphi^2 - 2.
\end{eqnarray}
Separation of variables gives the spectral stability problem in the form 
\begin{equation}
\label{lin-stab}
\left[ \begin{array}{cc} 0 & L_- \\ -L_+ & 0 \end{array} \right] 
\left[ \begin{array}{c} u \\ v \end{array} \right] = \lambda (1-\varphi^2) 
\left[ \begin{array}{c} u \\ v \end{array} \right] \qquad \Leftrightarrow \qquad 
\begin{array}{c} L_- v = \lambda (1-\varphi^2) u, \\ -L_+ u = \lambda 
(1-\varphi^2) v,\end{array}
\end{equation}
which naturally suggests to consider the linear operators $L_{\pm}$ 
in the Hilbert space 
\begin{equation}
\label{Hilbert-space}
\mathcal{H} := \left\{ f \in L^2_{\rm loc}(\R) : \quad 
\sqrt{1 - \varphi^2} f \in L^2(\R) \right\},
\end{equation}
with the associated inner product 
\begin{equation}
\label{inner-product}
(f,g)_{\mathcal{H}} := \int_{\R} (1-\varphi^2) \bar{f} g dx.
\end{equation}
and the induced norm $\| \cdot \|_{\mathcal{H}}$. The Hilbert space $\mathcal{H}$ is nothing but an exponentially weighted $L^2(\R)$ space since 
$$
\sqrt{1 - \varphi^2}(x) = {\rm sech}(x). 
$$
Related to $\mathcal{H}$, the form domain and the operator domain 
for the linear operators $L_{\pm}$ are defined as follows:
\begin{equation}
\label{form-domain}
\mathcal{H}^1_+ := H^1(\R), \quad 
\mathcal{H}^1_- := \left\{ f \in \mathcal{H} : \quad f' \in L^2(\R) \right\},
\end{equation}
and
\begin{equation}
\label{operator-domain}
\mathcal{H}^2_{\pm} := \left\{ f \in \mathcal{H}_{\pm}^1 : \quad 
\cosh^2(\cdot) \; L_{\pm} f \in \mathcal{H} \right\}.
\end{equation}
Writing $\mathcal{L}_{\pm} := \cosh^2(\cdot) L_{\pm}$ defines linear operators 
$\mathcal{L}_{\pm} : \mathcal{H}^2_{\pm} \subset \mathcal{H} \mapsto \mathcal{H}$. The linear stability  problem (\ref{lin-stab}) can be rewritten 
with the use of these operators in the form 
\begin{equation}
\label{lin-stab-H}
\mathcal{L}_- v = \lambda  u, \qquad -\mathcal{L}_+ u = \lambda v.
\end{equation}
The quadratic forms $Q_{\pm}$ in (\ref{quad-plus})--(\ref{quad-minus}) 
are defined by $\mathcal{L}_{\pm}$ in $\mathcal{H}^1_{\pm}$. 
We will next present four main results proven in this work.  \\

The first result (Section \ref{sec-spaces}) concerns with the spectrum of the linear operators $\mathcal{L}_{\pm}$ in $\mathcal{H}$ and the linear stability problem (\ref{lin-stab-H}) in $\mathcal{H}\times \mathcal{H}$.

\begin{theorem}
	\label{theorem-lin}
	The spectrum of $\mathcal{L}_{\pm}$ in $\mathcal{H}$ with the dense domain $\mathcal{H}^2_{\pm} \subset \mathcal{H}$ consists of simple isolated eigenvalues 
	as follows
\begin{equation}
	\label{spectrum-plus}
	\sigma_{\mathcal{H}}(L_+) = \{0, \mu_1, \mu_2, \dots \}, \qquad 
	0 < \mu_1 < \mu_2 < \dots
\end{equation}
and
\begin{equation}
\label{spectrum-minus}
\sigma_{\mathcal{H}}(L_-) = \{ -2, 0, \nu_1, \nu_2, \dots \}, \qquad 
	0 < \nu_1 < \nu_2 < \dots
\end{equation}
The spectrum of the linear stability problem (\ref{lin-stab-H}) in $\mathcal{H}\times \mathcal{H}$ consists of pairs of isolated eigenvalues 
\begin{equation}
\label{spectrum-stab}
\{ \pm i \omega_1, \pm i \omega_2,  \dots \}, \qquad \qquad  0 < \omega_1 \leq \omega_2 \leq \dots
\end{equation}
and a quadruple zero eigenvalue associated with the symmetries (\ref{symmetry}).
\end{theorem}

\begin{remark}
	\label{remark-L-minus}
	Eigenvalues of $\mathcal{L}_+$ and $\mathcal{L}_-$ in $\mathcal{H}$ can be found explicitly as 
	\begin{equation*}
\mu_n := n(n+5), \quad	\nu_n := (n+1) (n+2) - 2, \quad n \in \mathbb{N}.
	\end{equation*}
\end{remark}

\begin{remark}
	The linearized operator in the stability problem (\ref{lin-stab-H}) admits a two-dimensional kernel spanned 
	by $\{ \vec{w}_1, \vec{w}_2 \}$, where 
	\begin{equation}
	\label{eigenvectors}
	\vec{w}_1 = \left[ \begin{array}{c} \varphi' \\ 0 \end{array} \right], \quad
	\vec{w}_2 = \left[ \begin{array}{c} 0 \\ \varphi \end{array} \right].	
	\end{equation}
	These eigenvectors are associated with the symmetries (\ref{symmetry}). 
\end{remark}

Since the spectrum of $\mathcal{L}_{\pm}$ in $\mathcal{H}$ is purely discrete but includes negative and zero eigenvalues, we set some constraints on the perturbation 
term in the decomposition (\ref{perturbation}) in order to get coercivity 
of the quadratic forms in the expansion (\ref{expansion-Lambda}) and the energetic stability of the black soliton. The constrained space in $\mathcal{H}$ is formed by the following two orthogonality conditions:
\begin{equation*}
\mathcal{H}_c := \left\{ f \in \mathcal{H} : \quad 
(1-\varphi^2, f)_{\mathcal{H}} = 0, \quad (\varphi, f)_{\mathcal{H}} = 0 \right\},
\end{equation*}
where we have used that $\varphi' = 1 - \varphi^2$.

In terms of the perturbation $u + iv$ to the black soliton $\varphi$, the constraints 
\begin{equation}
\label{constraints-1}
(\varphi, u)_{\mathcal{H}} = 0 \quad \mbox{\rm and} \quad 
(1-\varphi^2, v)_{\mathcal{H}} = 0 
\end{equation}
are due to fixed mass $M(\psi)$ 
and momentum $P(\psi)$ and can be satisfied in the time evolution 
of the NLS model (\ref{nls-idd})  by extending the black soliton to the  travelling dark soliton with the wave speed $c$ and by using the scaling transformation (\ref{scaling-transform}). On the other hand, the constraints 
\begin{equation}
\label{constraints-2}
(1-\varphi^2, u)_{\mathcal{H}} = 0 \quad \mbox{\rm and} \quad 
(\varphi, v)_{\mathcal{H}} = 0 
\end{equation}
are due to the two symmetries (\ref{symmetry}) and can be satisfied by using modulation parameters for the orbit $\{ e^{i \theta} \varphi(\cdot + \zeta) \}_{\theta,\zeta \in \R}$ of the black soliton with the profile $\varphi$.

In applications to the spectral stability problem (\ref{lin-stab-H}) and its two-dimensional kernel (\ref{eigenvectors}) with $\varphi' = 1 - \varphi^2$, constraints (\ref{constraints-1}) represent the symplectic orthogonality conditions for $(u,v)$ with respect to $\{ \vec{w}_1, \vec{w}_2 \}$, whereas constraints (\ref{constraints-2}) represent the standard orthogonality conditions for $(u,v)$ with respect to $\{ \vec{w}_1, \vec{w}_2 \}$. \\

The second result (Section \ref{sec-existence}) is about the existence of dark solitons with the profile $U_c \in \mathcal{F}$ satisfying 
the differential equation (\ref{nls-trav}) and the boundary conditions 
(\ref{bc-dark}). We show that the dark solitons exist for every speed $c \in \mathbb{R}$ and converge to the black soliton as $c \to 0$.

\begin{theorem}
	\label{theorem-dark}
	For every $c \in \R$, there exists a  dark soliton  with the profile $U_c \in \mathcal{F} \cap C^{\infty}(\mathbb{R})$ satisfying (\ref{nls-trav}) and (\ref{bc-dark}). The mapping $c \mapsto U_c$ is $C^{\infty}$ on $\mathbb{R} \backslash \{ 0 \}$ and
	\begin{equation}
	\label{U-dark-conv}
	\| U_c - \varphi \|_{\mathcal{H}^2_-} \to 0 \quad \mbox{\rm as} \;\; c \to 0.
	\end{equation}
\end{theorem}

\begin{remark}
	For every $c \in \R$, there exists another solitary wave with the profile $U$ satisfying (\ref{nls-trav}) and (\ref{bc-dark}) with $1 < |U(\xi)| < \infty$ for every $\xi \in \mathbb{R}$. However, it is not relevant for the analysis of the black soliton with the profile $\varphi$.
\end{remark}

\begin{remark}
	It follows from the proof of Theorem \ref{theorem-dark} that the profile $U_c$ has the exponential rate of decay independently of the speed parameter $c \in \mathbb{R}$, in particular,
	$$
	|U_c(\xi)| = 1 + A_c e^{-2|\xi|} + {\rm o}(e^{-2|\xi|}) \quad \mbox{\rm as} \;\; |\xi| \to \infty,
	$$ 
	where $A_c < 0$. This is a novel feature compared to the dark solitons (\ref{dark-cubic}) of the cubic NLS equation (\ref{nls-cubic}), where 
	$|U_c(\xi)| = 1 + A_c e^{-2 \sqrt{1-c^2} |\xi|} + {\rm o}(e^{-2\sqrt{1-c^2}|\xi|})$ as $|\xi| \to \infty$ with $A_c < 0$.
\end{remark}

The third result (Section \ref{sec-stability}) establishes the energetic stability 
of the black soliton in $\Sigma \cap \mathcal{H}$, which gives 
the orbital stability under the assumption that the initial-value problem 
for the NLS model (\ref{nls-idd}) is locally well-posed in $\Sigma \cap \mathcal{H}$, where the distance in $\Sigma \cap \mathcal{H}$ is defined by 
\begin{equation*}
\mathcal{D}_{\Sigma \cap \mathcal{H}}(\psi_1,\psi_2) := \sqrt{\| \psi_1' - \psi_2' \|^2_{L^2} + \| |\psi_1|^2 - |\psi_2|^2 \|^2_{L^2} + \| \psi_1 - \psi_2 \|^2_{\mathcal{H}}}.
\end{equation*}
Since the energy method uses conserved quantities, we also assume that 
$E(\psi)$, $M(\psi)$, and $P(\psi)$ are preserved in the time evolution 
of the NLS model (\ref{nls-idd}) in $\Sigma \cap \mathcal{H}$. 
Note that we have shown conservation of these quantities in Section \ref{sec-formalism} under additional constraints (\ref{decay-psi}) and nonzero $\psi(t,x)$ but did not rigorously address their conservation 
in $\Sigma \cap \mathcal{H}$.

\begin{theorem}
	\label{theorem-nonlinear}
Assume that the initial-value problem for the NLS model (\ref{nls-idd}) 
is locally well-posed in $\Sigma \cap \mathcal{H}$ and the values of $E(\psi)$, $M(\psi)$ and $P(\psi)$ are independent of time $t$. Then the black 
soliton is orbitally stable in $\Sigma \cap \mathcal{H}$. 
\end{theorem}

\begin{remark}
	We use the standard definition of orbital stability and instability. 
	We say that $\varphi$ is orbitally stable in $\Sigma \cap \mathcal{H}$ if 
for every $\epsilon > 0$ there exists $\delta > 0$ such that 
	for every $\psi_0 \in \Sigma \cap \mathcal{H}$ satisfying
	$\mathcal{D}_{\Sigma \cap \mathcal{H}}(\psi_0,\varphi)  < \delta$, 
	the unique solution $\psi \in C^0(\mathbb{R}_+,\Sigma \cap \mathcal{H})$ satisfies 
	$$
	\inf_{\theta,\zeta \in \mathbb{R}} \mathcal{D}_{\Sigma \cap \mathcal{H}} \left( \psi(t,\cdot),e^{i \theta} 
	\varphi(\cdot - \zeta) \right) < \epsilon,
	$$
	for every $t > 0$. Otherwise, we say that $\varphi$ is orbitally unstable in $\Sigma \cap \mathcal{H}$.
\end{remark}

\begin{remark}
	In the context of the cubic NLS equation (\ref{nls-cubic}), only three constraints were used in \cite{GS} for the proof of orbital stability of black solitons: one was defined in $\mathcal{H}$ and two were defined in $L^2(\R)$. A similar approach was used in \cite{Alejo} for the quintic NLS equation. This relies on two symmetries of the NLS equation and the momentum conservation. Compared to these works, we are using four orthogonality conditions, all defined in $\mathcal{H}$, because the new NLS model (\ref{nls-idd}) has the additional scaling transformation (\ref{scaling-transform}).
\end{remark}

\begin{remark}
	We do not write down explicitly the time evolution of the modulation parameters $\theta$ and $\zeta$ of the orbit $\{ e^{i \theta} \varphi(\cdot + \zeta) \}_{\theta,\zeta \in \R}$. This can be done with the standard projection algorithm in $\mathcal{H}$, see \cite{CPP18,GPII,GS}, but is not required for the orbital stability result.
\end{remark}

\begin{remark}
	It is expected that the values of $E(\psi)$, $M(\psi)$ and $P(\psi)$ are preserved in time $t$ only if the solution $\psi(t,\cdot)$ stays in $\mathcal{F}$ given by (\ref{function-set}). Therefore, we expect that the local well-posedness analysis will show that the set $\mathcal{F}$ is invariant under the time evolution of the NLS model (\ref{nls-idd}).
\end{remark}

The final result (Section \ref{sec-potential}) is about persistence and stability of the black soliton in the perturbed NLS model (\ref{nls-idd-potential}) with the external potential $V$. We prove that  the same properties as in Theorem \ref{theorem-lin} can be extended to the black soliton pinned to the minimal point of the effective potential $\mathcal{V}(s)$ given by (\ref{potential-effective}). 

\begin{theorem}
	\label{theorem-potential}
	Assume that $V \in W^{2,\infty}(\R) \cap L^1(\R)$ and that $s \in \mathbb{R}$ is a simple root of $\mathcal{V}'(s)$, where $\mathcal{V}(s)$ is given by (\ref{potential-effective}). There exists $\varepsilon_0 > 0$ such that for every $\varepsilon \in (0,\varepsilon_0)$, there exists a black soliton of the perturbed NLS model (\ref{nls-idd-potential}) in the form 
	\begin{equation}	
\psi(t,x) = e^{-2it} \phi_{\varepsilon}(x), \qquad 
\phi_{\varepsilon}(x) = \varphi(x-s) + \tilde{\varphi}_{\varepsilon}(x),
	\end{equation}
	where $\varphi(x) = \tanh(x)$ and $\tilde{\varphi}_{\varepsilon} \in H^2(\R)$ satisfies $\| \tilde{\varphi}_{\varepsilon} \|_{H^2} \leq C \varepsilon$ for some $\varepsilon$-independent positive constant $C$. Moreover, 
	\begin{itemize}
		\item If $\mathcal{V}''(s) > 0$, the stability spectrum consists of pairs of isolated eigenvalues (\ref{spectrum-stab}) and a double zero eigenvalue. 
		\item If $\mathcal{V}''(s) < 0$, the stability spectrum consists of pairs of isolated eigenvalues (\ref{spectrum-stab}), a double zero eigenvalue, and a pair of simple real eigenvalues $\{ \pm \lambda_0 \}$ with $\lambda_0 > 0$. 
	\end{itemize}
\end{theorem}

\begin{remark}
	In the case $\mathcal{V}''(s) < 0$, the black soliton is orbitally unstable in $\Sigma \cap \mathcal{H}$ as the linear instability in the NLS models implies orbital instability \cite{GSS87} provided the initial-value problem is locally well-posed in $\Sigma \cap \mathcal{H}$. 
\end{remark}

\begin{remark}
	In the case $\mathcal{V}''(s) > 0$, the black soliton is not orbitally stable because the corresponding Lyapunov functional has two  eigendirections which correspond to two negative eigenvalues of the Hessian operator. Nevertheless, the spectrum of the linear stability problem is purely discrete in $\mathcal{H}$ and one can expect that the black soliton is orbitally stable in the NLS model (\ref{nls-idd-potential}) if $n \omega_1 \neq \omega_{k+1}$ for every $n \in \N$ and $k \in \N$, where $\omega_1$ is the smallest nonzero eigenvalue bifurcating from the zero eigenvalue if $\varepsilon \to 0$.
\end{remark}

The remainder of this paper is devoted to the proof of Theorems \ref{theorem-lin}, \ref{theorem-dark}, \ref{theorem-nonlinear}, and \ref{theorem-potential}, concluded with discussions of open problems.

\section{Linear operators $\mathcal{L}_{\pm}$ in the exponentially weighted $L^2(\R)$ space}
\label{sec-spaces}

Here we analyze the spectra of linear operators $\mathcal{L}_{\pm}$ and the linearized operator 
\begin{equation}
\label{operator-L}
\mathcal{L} = \left[ \begin{array}{cc} 0 & \mathcal{L}_- \\ - \mathcal{L}_+ & 0 \end{array} \right]
\end{equation}
in the Hilbert space $\mathcal{H}$ defined in (\ref{Hilbert-space}). We provide the proof of Theorem \ref{theorem-lin} which gives spectral stability of the black soliton and specifies that the spectra of $\mathcal{L}_-$ and  $\mathcal{L}_+$ in $\mathcal{H}$  and $\mathcal{L}$ in $\mathcal{H} \times \mathcal{H}$ are purely discrete. 

\subsection{Spectrum of $\mathcal{L}_-$ in $\mathcal{H}$}

It follows from (\ref{L-minus}) that $L_- = -\partial_x^2 - 2\,{\rm sech}^2(x)$. The spectral problem for $\mathcal{L}_-$ in $\mathcal{H}$ take the form 
\begin{equation}
\label{spectrum-L-minus}
- v''(x) - 2 \; {\rm sech}^2(x) \; v(x) = \nu \; {\rm sech}^2(x) \; v(x), \quad x \in \mathbb{R}.
\end{equation}
The first two eigenvalues of the spectral problem (\ref{spectrum-L-minus}) are available explicitly:
\begin{align*}
\nu = -2 : & \quad v(x) = 1, \\
\nu = 0 : & \quad v(x) = \tanh(x).
\end{align*}
The bounded solutions of the second-order differential equation (\ref{spectrum-L-minus}) belong to the form domain $\mathcal{H}^1_-$ given by (\ref{form-domain}), whereas the linearly growing solutions do not belong to $\mathcal{H}^1_-$. Hence, eigenvalues of the spectral problem 
(\ref{spectrum-L-minus}) coincide with the admissible values of $\nu$ 
for which the Schr\"{o}dinger operator 
$$
\mathcal{S}_{\nu} := -\partial_x^2 - (2+\nu) {\rm sech}^2(x) : H^2(\R) \subset L^2(\R) \to L^2(\R)
$$ 
admits the end-point resonance at the zero energy level with a bounded eigenfunction $v(x)$. This problem is well-known 
in mathematical physics (see \cite[Chapter 6, pp.768--769]{LL}), for which the differential equation (\ref{spectrum-L-minus}) is converted to the hypergeometric equation. It is also well-known that the power series for the hypergeometric function is truncated into a polynomial which admits a bounded eigenfunction if and only if 
$$
2+\nu = n(n+1), \quad \mbox{\rm where} \;\; n \in \mathbb{N}_0 := \{0,1,2,\dots\}.
$$ 
This gives the exact location of the simple 
eigenvalues at $\nu = n(n+1) - 2$ (Remark \ref{remark-L-minus}).

\begin{remark}
	\label{rem-Sturm}
	Each eigenvalue is simple because if one solution of the second-order differential equation (\ref{spectrum-minus}) is bounded, the second solution is linearly growing in $x$ at infinity. Moreover, Sturm's oscillation theorem states that the bounded eigenfunctions for the $n$-th eigenvalue has $(n-1)$ zeros on $\mathbb{R}$. Hence $0$ is the second eigenvalue of $\mathcal{L}_-$ in $\mathcal{H}$. 
\end{remark}

We are going to make the picture above precise and give the proof that 
the linear operator $\mathcal{L}_-$ in the Hilbert space $\mathcal{H}$ has a purely discrete spectrum consisting of simple isolated eigenvalues. In order to do so, we define a positive linear operator 
$$
M_- := -\partial_x^2 + {\rm sech}^2(x)
$$ 
or equivalently, 
\begin{equation}
\label{M-minus}
\mathcal{M}_- : \mathcal{H}^2_- \to \mathcal{H}, \quad 
\mathcal{M}_- := - \cosh^2(x) \partial_x^2 + 1,
\end{equation}
where the operator domain $\mathcal{H}_-^2$ is given by (\ref{operator-domain}). 
It follows from the triangle inequality that the operator domain $\mathcal{H}_-^2$ can be written in the equivalent form
\begin{equation}
\label{operator-domain-minus}
\mathcal{H}^2_- = \left\{ f \in \mathcal{H} : \quad f' \in L^2(\R), \;\; 
\cosh(\cdot) f'' \in L^2(\R) \right\}
\end{equation}
equipped with the norm
\begin{equation}
\label{operator-norm-minus}
\| f \|_{\mathcal{H}^2_-} := \sqrt{ \| \cosh(\cdot) f'' \|_{L^2}^2 + \| f' \|^2_{L^2} + \| f \|^2_{\mathcal{H}}}.
\end{equation}

The first result insures that $0$ does not belong to the spectrum of the operator $\mathcal{M}_-$ in $\mathcal{H}$.

\begin{lemma}
	\label{lem-M-minus}
	The linear operator $\mathcal{M}_-$ given by (\ref{M-minus}) is bijective 
	and symmetric, whereas $\mathcal{M}^{-1}_-$ is a bounded operator from $\mathcal{H}$ to $\mathcal{H}^2_-$. 
\end{lemma}

\begin{proof}
	For injectivity, we use the integration by parts formula (Lemma \ref{lem-B}) which ensures that if $u \in \mathcal{H}^2_-$, then 
	$\langle -u'',u \rangle_{L^2} = \| u' \|^2_{L^2}$. Hence, if there exists $u \in \mathcal{H}^2_-$ such that $\mathcal{M}_- u = 0$, then 
	$$
	\int_{\mathbb{R}} \left( |u'(x)|^2 + {\rm sech}^2(x) |u(x)|^2 \right) dx = 0
	$$
	which implies that $u = 0$ in $\mathcal{H}$. Hence, ${\rm Ker}(\mathcal{M}_-)$ is trivial in $\mathcal{H}$. 
	
	For surjectivity, let $f \in \mathcal{H}$ and consider the resolvent equation for $\mathcal{M}_-$ in the weak form:
	\begin{equation}
	\label{res-M-minus}
\left( u, \phi\right)_{\mathcal{H}^1_-} = \left( f, \phi\right)_{\mathcal{H}}, \quad  \forall \phi \in \mathcal{H}^1_-,
	\end{equation}
where the inner product in $\mathcal{H}^1_-$ is defined by 
	$$
	\left( u,\phi \right)_{\mathcal{H}^1_-} := 	\left( u,\phi \right)_{\mathcal{H}} + \langle u', \phi' \rangle_{L^2}, \quad \forall u, \phi \in \mathcal{H}^1_-, 
	$$
	together with the induced norm $\| \cdot \|_{\mathcal{H}^1_-}$.
Due to the Cauchy--Schwarz inequality 
	$$
	|\left( f, \phi\right)_{\mathcal{H}}| \leq \| f \|_{\mathcal{H}} \| \phi \|_{\mathcal{H}} \leq  \| f \|_{\mathcal{H}} \| \phi \|_{\mathcal{H}^1_-}, \quad \forall \phi \in \mathcal{H}^1_-
	$$
and the Riesz Representation Theorem, there exists a unique solution $u \in \mathcal{H}^1_-$ of (\ref{res-M-minus}) such that 
$\| u \|_{\mathcal{H}^1_-} \leq \| f \|_{\mathcal{H}}$. This yields 
$$
\langle u', \phi' \rangle_{L^2} = \int_{\R} {\rm sech}^2(\cdot) (\bar{f}-\bar{u}) \phi dx, \quad \forall \phi \in \mathcal{H}^1_-
$$
with ${\rm sech}^2(\cdot) (f-u) \in L^2(\R)$. Therefore, if we take $\phi \in C^{\infty}_c(\R) \subset \mathcal{H}^1_-$, then $u'$ is weakly differentiable and $-u'' = {\rm sech}^2(\cdot) (f-u) \in L^2(\R)$. 
Hence, $-\cosh(\cdot) u'' = {\rm sech}(\cdot) (f-u) \in L^2(\R)$ so that 
$u \in \mathcal{H}^2_-$ is a strong solution of 
$\mathcal{M}_- u = f$ for every $f \in \mathcal{H}$. 

The bound on the inverse operator $\mathcal{M}_-^{-1}$ from $\mathcal{H}$ to $\mathcal{H}^2_-$ follows from $\| u \|_{\mathcal{H}^1_-} \leq \| f \|_{\mathcal{H}}$ and
$$
\| \cosh(\cdot) u'' \|_{L^2} \leq \| f \|_{\mathcal{H}} + \| u \|_{\mathcal{H}} \leq 2 \| f \|_{\mathcal{H}}, 
$$
so that the definition (\ref{operator-norm-minus}) implies that $\| \mathcal{M}_-^{-1} f \|_{\mathcal{H}^2_-} \leq \sqrt{5} \| f \|_{\mathcal{H}}$ holds for every $f \in \mathcal{H}$.

Finally, the integration by parts 
formula (Lemma \ref{lem-B}) ensures for any $u,v \in \mathcal{H}^2_-$ that 
$$
\left( \mathcal{M}_- u, v \right)_{\mathcal{H}} = 
\int_{\R} (-\bar{u}'' + {\rm sech}^2(\cdot) \bar{u}) v dx = 
\int_{\R} \bar{u} (-v'' + {\rm sech}^2(\cdot) v)  dx = 
\left( u, \mathcal{M}_- v \right)_{\mathcal{H}},
$$
so that the linear operator (\ref{M-minus}) is symmetric.
\end{proof}

The next goal is to prove compactness of the embedding of $\mathcal{H}^2_-$ into $\mathcal{H}$. This will ensure that the spectrum of the operator $\mathcal{M}_-$ in $\mathcal{H}$ is purely discrete.

\begin{lemma}
	\label{lem-compact-minus}
	The embedding $\mathcal{H}^2_- \hookrightarrow \mathcal{H}$ is compact. 
\end{lemma}

\begin{proof}
	Let $\{ u_n \}$ denotes a bounded sequence in $\mathcal{H}^2_-$ so that $\{ \cosh(\cdot) u_n''\}$, $\{ u_n'\}$, and 
	$\{ {\rm sech}(\cdot) u_n \}$ are bounded in $L^2(\R)$. In particular, 
	since $\cosh(\cdot) \geq 1$, $\{ u_n'\}$ is bounded in $H^1(\R)$. 
	Since the mapping $H^1(\R) \ni v \mapsto {\rm sech}(\cdot) v \in L^2(\R)$ is compact, there exists a subsequence $\{ u_{n_k}\}$ such that $\{ {\rm sech}(\cdot) u_{n_k}' \}$ is convergent in $L^2(\R)$. 
	
	Furthermore, $\left( {\rm sech}(\cdot) u_n \right)' = {\rm sech}(\cdot) u_n' - \tanh(\cdot) {\rm sech}(\cdot) u_n$ is bounded in $L^2(\R)$, hence 
	$\{ {\rm sech}(\cdot) u_n\}$ is bounded in $H^1(\R)$. By Sobolev's embedding 
	of $H^1(\R)$ into $C^0(\R) \cap L^{\infty}(\R)$, $\{ u_n(0) \}$ is bounded in $\mathbb{C}$ so that without loss of generality, 
	$\{ u_{n_k}(0) \}$ converges in $\mathbb{C}$. Using 
	$$
	u_n(x) = u_n(0) + \int_0^x \cosh(t) {\rm sech}(t) u_n'(t) dt,
	$$
	we obtain 
\begin{align*}
	{\rm sech}(x) |u_{n_k}(x) - u_{n_j}(x)| & \leq 
	{\rm sech}(x) |u_{n_k}(0) - u_{n_j}(0)| \\
	& + 
	{\rm sech}(x) \int_0^x \cosh(t) {\rm sech}(t) |u'_{n_k}(t) - u'_{n_j}(t)| dt.
\end{align*}
The first term in the right-hand side is bounded in $L^2(\R)$ by 
$$
\sqrt{2} |u_{n_k}(0) - u_{n_j}(0)|, 
$$
which vanishes in $L^2(\R)$ as $k,j \to \infty$ due to convergence of $\{ u_{n_k}(0)\}$. Using the bound (\ref{tech-3}) (Lemma \ref{lem-A}), the second term in the right-hand side is bounded in $L^2(\R)$ by 
$$
2 \| {\rm sech}(\cdot) (u_{n_k}' - u_{n_j}') \|_{L^2}, 
$$
which vanishes as $k,j \to \infty$ due to convergence of $\{ {\rm sech}(\cdot) u_{n_k}' \}$. Hence, $\{ {\rm sech}(\cdot) u_{n_k}\}$ converges in $L^2(\R)$, and thus 
$\{ u_{n_k}\}$ converges in $\mathcal{H}$ which verifies that 
the embedding $\mathcal{H}^2_- \hookrightarrow \mathcal{H}$ is compact.
\end{proof}

Combining the inverse operator $\mathcal{M}_-^{-1} : \mathcal{H} \to \mathcal{H}^2_-$  in Lemma \ref{lem-M-minus} and 
the compact embedding $\mathcal{H}^2_- \hookrightarrow  \mathcal{H}$ 
in Lemma \ref{lem-compact-minus}, we can define a compact 
operator 
\begin{equation}
\label{K-minus}
\mathcal{K}_- = \mathcal{M}_-^{-1} : \mathcal{H} \to \mathcal{H}^2_- \hookrightarrow  \mathcal{H}. 
\end{equation}
The following corollary 
gives the desired result for the proof of Theorem \ref{theorem-lin}.

\begin{corollary}
	\label{cor-minus}
The spectrum of $\mathcal{L}_-$ in $\mathcal{H}$ consists of isolated eigenvalues. 
\end{corollary}

\begin{proof}
The compact operator $\mathcal{K}_-$ has a purely discrete spectrum of 
eigenvalues accumulating at $0$. Hence the linear operator 
$\mathcal{M}_- : \mathcal{H}^2_- \subset \mathcal{H} \to \mathcal{H}$ has a  purely discrete spectrum of eigenvalues. Since $\mathcal{L}_- = \mathcal{M}_- - 3$, the spectrum of $\mathcal{L}_-$ in $\mathcal{H}$ consists of isolated eigenvalues. 
\end{proof}

\begin{remark}
	\label{rem-minus}
	Isolated eigenvalues of $\mathcal{L}_-$ in $\mathcal{H}$ given by Corollary \ref{cor-minus} satisfy the ordering (\ref{spectrum-minus}) because each eigenvalue is simple and $\nu = 0$ is the second eigenvalue of the spectral problem (\ref{spectrum-L-minus}) (Remark \ref{rem-Sturm}).
\end{remark}

\subsection{Spectrum of $\mathcal{L}_+$ in $\mathcal{H}$}

It follows from (\ref{L-plus}) that $L_+ = -\partial_x^2 + 4 - 6\, {\rm sech}^2(x)$. The spectral problem for $\mathcal{L}_+$ in $\mathcal{H}$ take the form 
\begin{equation}
\label{spectrum-L-plus}
- u''(x) + 4u(x) - 6 \; {\rm sech}^2(x) \;u(x) = \mu \; {\rm sech}^2(x) \; u(x), \quad x \in \mathbb{R}.
\end{equation}
The first eigenvalue of the spectral problem (\ref{spectrum-L-plus}) is available explicitly:
\begin{align*}
\mu = 0 : & \quad u(x) = {\rm sech}^2(x).
\end{align*}
The differential equation (\ref{spectrum-L-plus}) has exponentially growing and exponentially decaying solutions at infinity. By Levinston's theorem (Proposition 8.1 in \cite{CL}), the growth and decay rates are $\pm 2$, hence 
the exponentially growing solutions do not belong to $\mathcal{H}$. 
Thus, we conclude that eigenvalues of the spectral problem 
(\ref{spectrum-L-plus}) coincide with the admissible values of $\mu$ 
for which the Schr\"{o}dinger operator 
$$
\mathcal{R}_{\mu} := -\partial_x^2 - (6+\mu) {\rm sech}^2(x) : H^2(\R) \subset L^2(\R) \to L^2(\R)
$$ 
admits an eigenvalue at the energy level $-4$ with a bounded and exponentially decaying eigenfunction $u(x)$. By using the exact solution from the hypergeometric equation (see \cite[Chapter 6, pp.768--769]{LL}), the bounded and exponentially decaying eigenfunction exists at the energy level $-4$ if and only if 
$$
\mu = n(n+5), \quad \mbox{\rm where} \;\; n \in \mathbb{N}_0 := \{0,1,2,\dots\}.
$$ 
This gives the exact location of the simple eigenvalues (Remark \ref{remark-L-minus}).

\begin{remark}
	\label{rem-Sturm-pos}
Each eigenvalue is simple because if one solution of the second-order differential equation 
(\ref{spectrum-L-plus}) is exponentially decaying, then the second solution is exponentially growing in $x$ at infinity. Moreover, Sturm's oscillation theorem states that the bounded eigenfunctions for the $n$-th eigenvalue has $(n-1)$ zeros on $\mathbb{R}$. Hence $0$ is the first eigenvalue of $\mathcal{L}_-$. 
\end{remark}

We will complete the picture above with the proof that 
the linear operator $\mathcal{L}_+$ in the Hilbert space $\mathcal{H}$ has a purely discrete spectrum consisting of simple isolated eigenvalues. In order to do so, we define a positive linear operator 
$$
M_+ := -\partial_x^2 + 4
$$ 
or equivalently, 
\begin{equation}
\label{M-plus}
\mathcal{M}_+ : \mathcal{H}^2_+ \to \mathcal{H}, \quad 
\mathcal{M}_+ := \cosh^2(x) (-\partial_x^2 + 4),
\end{equation}
where the operator domain $\mathcal{H}_+^2$ given by (\ref{operator-domain}) can be rewritten  in the equivalent form
\begin{equation}
\label{operator-domain-plus}
\mathcal{H}^2_+ = \left\{ f \in H^1(\R) \;\; 
\cosh(\cdot) (-\partial_x^2 + 4) f \in L^2(\R) \right\},
\end{equation}
equipped with the norm
\begin{equation}
\label{operator-norm-plus}
\| f \|_{\mathcal{H}^2_+} := \sqrt{ \| \cosh(\cdot) (-f'' + 4f) \|_{L^2}^2 + \| f' \|^2_{L^2} + 4 \| f \|^2_{L^2}}.
\end{equation}
We also recall that $\mathcal{H}^1_+ \equiv H^1(\R)$.

The next two lemmas give equivalent results to those in Lemma \ref{lem-M-minus} and \ref{lem-compact-minus}. 

\begin{lemma}
	\label{lem-M-plus}
	The linear operator $\mathcal{M}_+$ given by (\ref{M-plus}) is bijective 
	and symmetric, whereas $\mathcal{M}^{-1}_+$ is a bounded operator from $\mathcal{H}$ to $\mathcal{H}^2_+$.
\end{lemma}

\begin{proof}
		For injectivity, we again use the integration by parts formula (Lemma \ref{lem-C}) which yields for every $u \in \mathcal{H}^2_+$:
	$$
\left( \mathcal{M}_+ u,u \right)_{\mathcal{H}} = \int_{\mathbb{R}} \left( |u'(x)|^2 + 4 |u(x)|^2 \right) dx.
	$$
If $u \in \mathcal{H}^2_+$ is a solution of $\mathcal{M}_+ u = 0$, then	 $u = 0$ in $\mathcal{H}$ and ${\rm Ker}(\mathcal{M}_+)$ is trivial in $\mathcal{H}$. 
	
For surjectivity, let $f \in \mathcal{H}$ and consider the resolvent equation for $\mathcal{M}_+$ in the weak form:
	\begin{equation}
	\label{res-M-plus}
	\left( u, \phi\right)_{\mathcal{H}^1_+} = \left( f, \phi\right)_{\mathcal{H}}, \quad  \forall \phi \in \mathcal{H}^1_+,
	\end{equation}
	where we recall that $\mathcal{H}^1_+ \equiv H^1(\R)$ and use the  inner product
$$
\left( u, \phi \right)_{\mathcal{H}^1_+} := 	4 \langle u, v \rangle_{L^2} + \langle u', \phi' \rangle_{L^2}, \quad \forall u, \phi \in \mathcal{H}^1_+, 
$$
together with the induced norm $\| \cdot \|_{\mathcal{H}^1_+}$. By the Riesz Representation Theorem and the Cauchy--Schwarz inequality, there exists a unique solution $u \in \mathcal{H}^1_+$ of (\ref{res-M-plus}) such that $\| u \|_{\mathcal{H}^1_+} \leq \| f \|_{\mathcal{H}}$. This yields 
	$$
	\langle u', \phi' \rangle = \int_{\R} ({\rm sech}^2(\cdot) \bar{f} - 4\bar{u}) \phi dx, \quad  \forall \phi \in \mathcal{H}^1_+.
	$$
In particular, this holds for $\phi \in C^{\infty}_c(\R)$ 
implying $-u'' = {\rm sech}^2(\cdot) f - 4u \in L^2(\R)$ 
and  $\cosh(\cdot) (-u'' + 4u) = {\rm sech}(\cdot) f \in L^2(\R)$ so that 
$u \in \mathcal{H}^2_+$ is a strong solution of the resolvent equation 
$\mathcal{M}_+ u = f$ and this holds for every $f \in \mathcal{H}$.
	
The bound on the inverse operator $\mathcal{M}_+^{-1}$ from $\mathcal{H}$ to $\mathcal{H}^2_+$ follows from 
$$
	\| \cosh(\cdot) (-u'' + 4 u) \|_{L^2} = \| f \|_{\mathcal{H}}
$$
so that the definition (\ref{operator-norm-plus}) implies that 
$\| \mathcal{M}_+^{-1} f \|_{\mathcal{H}^2_-} \leq \sqrt{2} \| f \|_{\mathcal{H}}$ holds for every $f \in \mathcal{H}$.
	
The symmetry of $\mathcal{M}_+ : \mathcal{H}^2_+ \subset \mathcal{H}\to \mathcal{H}$ follows from the integration by parts 
	formula (Lemma \ref{lem-C}).
\end{proof}

\begin{lemma}
	\label{lem-compact-plus}
	The embedding $\mathcal{H}^2_+ \hookrightarrow \mathcal{H}$ is compact. 
\end{lemma}

\begin{proof}
		Since the mapping $H^1(\R) \ni v \mapsto {\rm sech}(\cdot) v \in L^2(\R)$ is compact and $\mathcal{H}^1_+ \equiv H^1(\R)$, 
		the embedding $\mathcal{H}^1_+ \hookrightarrow \mathcal{H}$ is compact and so is the embedding $\mathcal{H}^2_+ \hookrightarrow \mathcal{H}$  
		since $\mathcal{H}^2_+ \subset \mathcal{H}^1_+$ with continuous embedding.
\end{proof}

Combining the inverse operator $\mathcal{M}_+^{-1} : \mathcal{H} \to \mathcal{H}^2_+$  in Lemma \ref{lem-M-plus} and 
the compact embedding $\mathcal{H}^2_+ \hookrightarrow  \mathcal{H}$ 
in Lemma \ref{lem-compact-plus}, we can define a compact 
operator 
\begin{equation}
\label{K-plus}
\mathcal{K}_+ = \mathcal{M}_+^{-1} : \mathcal{H} \to \mathcal{H}^2_+ \hookrightarrow  \mathcal{H}. 
\end{equation}
The following corollary 
gives the desired result for the proof of Theorem \ref{theorem-lin}.

\begin{corollary}
	\label{cor-plus}
	The spectrum of $\mathcal{L}_+$ in $\mathcal{H}$ consists of isolated eigenvalues. 
\end{corollary}

\begin{proof}
The compact operator $\mathcal{K}_+$ has a purely discrete spectrum of 
eigenvalues accumulating to $0$. Hence the linear operator 
$\mathcal{M}_+ : \mathcal{H}^2_+ \subset \mathcal{H} \to \mathcal{H}$ has a purely discrete spectrum of eigenvalues. Since $\mathcal{L}_+ = \mathcal{M}_+ - 6$, the spectrum of $\mathcal{L}_+$ in $\mathcal{H}$ consists of isolated eigenvalues.
\end{proof}

\begin{remark}
	\label{rem-plus}
Isolated eigenvalues of $\mathcal{L}_+$ in $\mathcal{H}$ given by Corollary \ref{cor-plus} satisfy the ordering (\ref{spectrum-plus}) because each eigenvalue is simple and $\mu = 0$ is the first eigenvalue of the spectral problem (\ref{spectrum-L-plus}) (Remark \ref{rem-Sturm-pos}).
\end{remark}

\subsection{Spectrum of $\mathcal{L}$ in $\mathcal{H} \times \mathcal{H}$}

The following lemma characterizes the quadruple zero eigenvalue of the stability problem (\ref{lin-stab-H}).

\begin{lemma}
	\label{lem-quadruple}
	The zero eigenvalue $\lambda = 0$ of the stability problem (\ref{lin-stab-H}) has double geometric multiplicity and quadruple algebraic multiplicity.
\end{lemma}

\begin{proof}
Recall the two eigenvectors (\ref{eigenvectors}) of the stability problem (\ref{lin-stab-H}) for $\lambda = 0$. 
Since the linearized operator $\mathcal{L}$ in (\ref{operator-L}) is anti-diagonal, eigenvectors for $\lambda = 0$ 
are given by solutions of $\mathcal{L}_+ u = 0$ and $\mathcal{L}_- v = 0$ in $\mathcal{H}$.  Hence the two eigenvectors (\ref{eigenvectors}) are the only solutions 
of these equations in $\mathcal{H}$ (Remarks \ref{rem-minus} and \ref{rem-plus}) so that $\lambda = 0$ is an eigenvalue of geometric multiplicity {\em two}. 

Generalized eigenvectors of the stability problem (\ref{lin-stab-H}) 
for $\lambda = 0$ can be written in the form 
\begin{equation}
\label{eigenvectors-generalized}
\vec{w}_{g1} = \left[ \begin{array}{c} 0 \\ v_{\varphi} \end{array} \right], \quad
\vec{w}_{g2} = \left[ \begin{array}{c} u_{\varphi} \\ 0 \end{array} \right], \qquad 
\begin{array}{c} \mathcal{L}_- v_{\varphi} = \varphi', \\ -\mathcal{L}_+ u_{\varphi} = \varphi. \end{array}
\end{equation}
Since $\varphi(x) = \tanh(x)$, the inhomogeneous equations 
(\ref{eigenvectors-generalized}) can be solved in $\mathcal{H}$ exactly as 
\begin{equation}
\label{eigen-gener-exact}
v_{\varphi}(x) = \frac{1}{4} {\rm sech}^2(x) - \frac{1}{2}, \quad 
u_{\varphi}(x) = -\frac{1}{4} x \; {\rm sech}^2(x), 
\end{equation}
where the homogeneous solutions are set to zero without loss of generality. In order to prove that $\lambda = 0$ is an eigenvalue of algebraic multiplicity {\em four}, we need to show that the inhomogeneous linear equations 
\begin{equation}
\label{second-level}
\begin{array}{c} 
	-\mathcal{L}_+ \tilde{u} = v_{\varphi}, \\
	\mathcal{L}_- \tilde{v} =  u_{\varphi},
\end{array}
\end{equation} 
do not admit solutions in $\mathcal{H}$. Since $(\mathcal{L}_+ \tilde{u}, \varphi' )_{\mathcal{H}} = 0$ for every $\tilde{u} \in \mathcal{H}^2_+$ due to integration by parts (Lemma \ref{lem-C}) and $\mathcal{L}_+ \varphi' = 0$, it follows that 
$\tilde{u} \in \mathcal{H}^2_+ \subset \mathcal{H}$ in (\ref{second-level}) exists if and only if $(v_{\varphi},\varphi')_{\mathcal{H}} = 0$.
However, this is a contradiction since $(v_{\varphi},\varphi')_{\mathcal{H}} < 0$ 
because $v_{\varphi}(x) < 0$ and $\varphi'(x) > 0$ for every $x \in \mathbb{R}$. One can compute exactly that $(v_{\varphi},\varphi')_{\mathcal{H}} = -\frac{2}{5}$.

Since $(\mathcal{L}_- \tilde{v}, \varphi )_{\mathcal{H}} = 0$ for every $\tilde{v} \in \mathcal{H}^2_-$ due to integration by parts (Lemma \ref{lem-B}) and $\mathcal{L}_- \varphi = 0$, it follows that $\tilde{v} \in \mathcal{H}^2_- \subset \mathcal{H}$ in (\ref{second-level}) exists if and only if 
$(u_{\varphi},\varphi)_{\mathcal{H}} = 0$. 
However, this is again a contradiction since $(u_{\varphi},\varphi)_{\mathcal{H}} < 0$ 
because $u_{\varphi}(x) < 0$ and $\varphi(x) > 0$ for every $x > 0$ and $u_{\varphi}$, $\varphi$ are odd functions. One can compute exactly that $(u_{\varphi},\varphi)_{\mathcal{H}} = -\frac{1}{12}$. 

Thus, the Jordan chain of generalized eigenvectors of the stability problem (\ref{lin-stab-H}) is truncated at the two generalized eigenvectors (\ref{eigenvectors-generalized}) so that $\lambda = 0$ is eigenvalue of algebraic multiplicity {\em four}.
\end{proof}

We will further prove that the spectrum of the stability problem (\ref{lin-stab}) is purely discrete and consists of eigenvalues 
with the eigenfunctions $(u,v) \in \mathcal{H}^2_+ \times \mathcal{H}^2_-$. To do so, we rewrite the resolvent equation in the form 
\begin{equation}
\label{resolvent-eq}
\begin{array}{c} \mathcal{L}_- v = \lambda u + f, \\ -\mathcal{L}_+ u = \lambda 
v + g,\end{array}
\end{equation}
for a given $(f,g) \in \mathcal{H}\times \mathcal{H}$. Since the quadruple eigenvalue $\lambda = 0$ is prescribed by Lemma \ref{lem-quadruple}, the resolvent equation (\ref{resolvent-eq}) is considered for $\lambda\neq 0$.

By using the compact operators  $\mathcal{K}_{\pm}$ given by (\ref{K-minus}) and (\ref{K-plus}), we rewrite the resolvent equation (\ref{resolvent-eq}) in the equivalent form 
\begin{equation}
\label{resolvent-equiv}
\begin{array}{c} (I - 3 \mathcal{K}_-) v = \lambda \mathcal{K}_- u + \mathcal{K}_- f, \\ -(I - 6 \mathcal{K}_+) u = \lambda 
\mathcal{K}_+ v + \mathcal{K}_+ g,\end{array},
\end{equation}
where $I$ is the identity in $\mathcal{H}$. Before progressing further, we recall the generalized kernel of the stability problem in Lemma \ref{lem-quadruple} from solutions of the linear equations
\begin{align*}
\mathcal{L}_+ \varphi' = 0, \quad \mathcal{L}_- \varphi = 0, \quad 
\mathcal{L}_+ u_{\varphi} = - \varphi, \quad \mathcal{L}_- v_{\varphi} = \varphi',
\end{align*}
which we can reproduce in the equivalent form 
\begin{align*}
(I - 6 \mathcal{K}_+) \varphi' = 0, \quad  (I - 3 \mathcal{K}_-) \varphi = 0, \quad 
(I - 6 \mathcal{K}_+) u_{\varphi} = - \mathcal{K}_+ \varphi, \quad  (I - 3 \mathcal{K}_-) v_{\varphi} = \mathcal{K}_- \varphi'.
\end{align*}
The following lemma will ensure that projections to the eigenfunctions of the generalized kernel are non-degenerate.

\begin{lemma}
	\label{lem-non-degeneracy}
	The following two matrices are invertible:
	\begin{equation}
	\label{A-plus}
	A_+ := \left[\begin{array}{cc} (\varphi', \mathcal{K}_+ \varphi)_{\mathcal{H}} & (\varphi', \mathcal{K}_+ v_{\varphi})_{\mathcal{H}} \\
	(u_{\varphi}, \mathcal{K}_+ \varphi)_{\mathcal{H}} & (u_{\varphi}, \mathcal{K}_+ v_{\varphi})_{\mathcal{H}} \end{array} \right]
	\end{equation}
	and
\begin{equation}
	\label{A-minus}
	A_- := \left[\begin{array}{cc} (\varphi, \mathcal{K}_- \varphi')_{\mathcal{H}} & (\varphi, \mathcal{K}_-  u_{\varphi})_{\mathcal{H}} \\
	(v_{\varphi}, \mathcal{K}_- \varphi')_{\mathcal{H}} & (v_{\varphi}, \mathcal{K}_- u_{\varphi})_{\mathcal{H}} \end{array} \right]
	\end{equation}
\end{lemma}

\begin{proof}
	The first diagonal entries of $A_{\pm}$ are zero because 
	$$
	(\varphi', \mathcal{K}_+ \varphi)_{\mathcal{H}} = 
	(\mathcal{K}_+ \varphi',  \varphi )_{\mathcal{H}} = 
\frac{1}{6} 	(\varphi',  \varphi )_{\mathcal{H}} = 0
	$$
	and 
$$
(\varphi, \mathcal{K}_- \varphi')_{\mathcal{H}} = 
(\mathcal{K}_-  \varphi, \varphi')_{\mathcal{H}} = 
\frac{1}{3}	 (\varphi, \varphi')_{\mathcal{H}} = 0
$$
We will show that the off-diagonal entries of $A_{\pm}$ are all non-zero. 
This will ensure invertibility of $A_{\pm}$ irrespectively 
of the last diagonal entries of $A_{\pm}$. For two off-diagonal entries, 
we obtain 
$$
(\varphi', \mathcal{K}_+ v_{\varphi})_{\mathcal{H}} = 
(\mathcal{K}_+ \varphi',  v_{\varphi} )_{\mathcal{H}} = 
\frac{1}{6} 	(\varphi',  v_{\varphi})_{\mathcal{H}} < 0
$$
and 
$$
(\varphi, \mathcal{K}_- u_{\varphi})_{\mathcal{H}} = 
(\mathcal{K}_-  \varphi, u_{\varphi})_{\mathcal{H}} = 
\frac{1}{3}	 (\varphi, u_{\varphi})_{\mathcal{H}} < 0,
$$
where the signs of the last terms are computed in the proof 
of Lemma \ref{lem-quadruple}. For the other two off-diagonal entries, 
we claim the following explicit expressions
	\begin{align}
(\mathcal{K}_- \varphi')(x) &= \frac{4}{7} + \frac{1}{7} {\rm sech}^2(x), 
\label{K-minus-exp} \\
(\mathcal{K}_+ \varphi)(x) &= x \cosh(2x) - \sinh(2x) \log(2 \cosh(x)) 
+ \frac{1}{2} \tanh(x).
\label{K-plus-exp}
\end{align}
The expression (\ref{K-minus-exp}) is obtained by solving the inhomogeneous equation 
		$$
	(-\partial_x^2 + {\rm sech}^2(x)) (\mathcal{K}_- \varphi') = 
	(-\partial_x^2 - 2 {\rm sech}^2(x)) v_{\varphi} = {\rm sech}^4(x).
	$$
from which we obtain $(v_{\varphi}, \mathcal{K}_- \varphi')_{\mathcal{H}} < 0$ since $v_{\varphi}(x) < 0$ and $(\mathcal{K}_- \varphi')(x) > 0$ for every $x \in \R$.

The expression (\ref{K-plus-exp}) is obtained by solving the inhomogeneous equation 
$$
(-\partial_x^2 + 4) (\mathcal{K}_+ \varphi) = 
-(-\partial_x^2 + 4 - 6 {\rm sech}^2(x)) u_{\varphi} = \tanh(x) {\rm sech}^2(x)
$$
with the homogeneous solution chosen to satisfy the requirement
that $\mathcal{K}_+ \varphi \in \mathcal{H}$. To show this, we can rewrite the right-hand-side of (\ref{K-plus-exp}) for $x > 0$ as 
\begin{equation}
\label{expl-expr}
x e^{-2x} 
	- \frac{1}{2} g(x) + \frac{1}{2} (g(x) + 1) e^{-4x} - \frac{e^{-2x}}{1+e^{-2x}},
\end{equation}
	where 
	$$
	g(x) := e^{2x} \left[ \log(2\cosh(x)) - x\right] - 1 = 
	e^{2x} \log(1 + e^{-2x}) - 1.
	$$
Taylor expansion shows that $g(x) = -\frac{1}{2} e^{-2x} + \mathcal{O}(e^{-4x})$ as $x \to +\infty$ so that the explicit expression (\ref{expl-expr}) decays to $0$ as $x \to +\infty$ exponentially fast. Since it is odd in $x$,
it belongs to $H^1(\R) \subset\mathcal{H}$ 
so that it coincides with $(\mathcal{K}_+ \phi)$. Moreover, since $(\mathcal{K}_+ \varphi)(x) \to 0$ as $x \to +\infty$ 
and 
$$
\frac{d}{dx} \frac{(\mathcal{K}_+ \varphi)(x)}{\sinh(2x)} = \frac{\tanh(x) - x}{2 \sinh^2(x) \cosh^2(x)} < 0, \quad x > 0,
$$
we have $(\mathcal{K}_+ \varphi)(x) > 0$ for $x > 0$, from which we obtain 
$(u_{\varphi}, \mathcal{K}_+ \varphi)_{\mathcal{H}} < 0$ 
since $u_{\varphi}(x) < 0$ for $x > 0$ and both functions are odd. Thus, all four off-diagonal terms in $A_{\pm}$ are nonzero so that $A_{\pm}$ are invertible.
\end{proof}

Let us now define the following two subspaces of $\mathcal{H}$:
$$
\mathcal{U} := {\rm span}(\varphi',u_{\varphi})^{\perp}, \quad 
\mathcal{V} := {\rm span}(\varphi,v_{\varphi})^{\perp},
$$
and denote the orthogonal projection operators $\Pi_{\mathcal{U}} : \mathcal{H}\mapsto \mathcal{U}$ and $\Pi_{\mathcal{V}} : \mathcal{H}\mapsto \mathcal{V}$.

The following lemma describes decompositions of $\mathcal{H}$ into two direct sums, which are generally non-orthogonal.

\begin{lemma}
	\label{lem-projection}
	$\mathcal{H}$ can be decomposed into the following direct sums 
	$$
	\mathcal{H} = \mathcal{U} \oplus {\rm span}(\mathcal{K}_+ \varphi, \mathcal{K}_+ v_{\varphi}) = \mathcal{V} \oplus {\rm span}(\mathcal{K}_- \varphi', \mathcal{K}_- u_{\varphi}).
	$$
\end{lemma}

\begin{proof}
	First we show that
	$$
	\xi_1 := \mathcal{K}_+ \varphi - \Pi_{\mathcal{U}}(\mathcal{K}_+ \varphi), \qquad 
	\xi_2 := \mathcal{K}_+ v_{\varphi} - \Pi_{\mathcal{U}}(\mathcal{K}_+ v_{\varphi})
	$$
are linearly independent. Letting $(c_1,c_2) \in \CC^2$ such that $c_1 \xi_1 + c_2 \xi_2 = 0$ implies 
$$
c_1 \mathcal{K}_+ \varphi  + c_2 \mathcal{K}_+  v_{\varphi} \in \mathcal{U}.
$$
Taking inner products in $\mathcal{H}$ with $\varphi'$ and $u_{\varphi}$ yields 
the linear homogeneous system 
$$
A_+ \left[ \begin{array}{c} c_1 \\ c_2 \end{array} \right] = \left[ \begin{array}{c} 0 \\ 0 \end{array} \right],
$$
which has no nontrivial solutions by Lemma \ref{lem-non-degeneracy}. Hence, $c_1 = c_2 = 0$. Thus, $\{ \xi_1,\xi_2 \}$ is a basis of 
$\mathcal{U}^{\perp} := {\rm span}(\varphi',u_{\varphi})$ so that every $f \in \mathcal{H}$ can be written as 
$$
f = u + \alpha \xi_1 + \beta \xi_2 = \tilde{u} + \alpha \mathcal{K}_+ \varphi + \beta \mathcal{K}_+ v_{\varphi},
$$
with uniquely determined $u \in \mathcal{U}$ and $(\alpha,\beta) \in \CC^2$, 
and with $\tilde{u}:= u - \Pi_{\mathcal{U}}(\alpha \mathcal{K}_+ \varphi + \beta \mathcal{K}_+ v_{\varphi})$. 

The decomposition of $\mathcal{H}$ 
into $\mathcal{V} \oplus {\rm span}(\mathcal{K}_- \varphi', \mathcal{K}_- u_{\varphi})$ follows analogously but with using invertibility of $A_-$ in Lemma \ref{lem-non-degeneracy}.
\end{proof}

With the help of Lemma \ref{lem-projection}, we define the projection operator $P_{\mathcal{U}}$ from $\mathcal{H}$ to $\mathcal{U}$ along ${\rm span}(\mathcal{K}_+ \varphi, \mathcal{K}_+ v_{\varphi})$ as
$$
P_{\mathcal{U}} ( u + \alpha \mathcal{K}_+ \varphi + \beta \mathcal{K}_+ v_{\varphi}) = u, \quad \forall u \in \mathcal{U}, \quad \forall (\alpha,\beta) \in \R^2.
$$
Similarly, the projection operator $P_{\mathcal{V}}$ from $\mathcal{H}$ to $\mathcal{V}$ along ${\rm span}(\mathcal{K}_- \varphi', \mathcal{K}_- u_{\varphi})$ is defined as
$$
P_{\mathcal{V}} ( v + \alpha \mathcal{K}_- \varphi' + \beta \mathcal{K}_- u_{\varphi}) = v, \quad \forall v \in \mathcal{V}, \quad \forall (\alpha,\beta) \in \R^2.
$$
The following lemma gives invertibility 
of the linear operators in the resolvent equation (\ref{resolvent-equiv}) 
after appropriate projection operators. 

\begin{lemma}
	\label{lem-invertibility}
	The following two operators 
	$$
	\mathcal{S}_+ := P_{\mathcal{U}} ( I - 6 \mathcal{K}_+) |_{\mathcal{U}} : \mathcal{U} \mapsto \mathcal{U}
	$$
	and
		$$
	\mathcal{S}_- := P_{\mathcal{V}} ( I - 3 \mathcal{K}_-) |_{\mathcal{V}} : \mathcal{V} \mapsto \mathcal{V}
	$$
	are invertible with bounded inverse.
\end{lemma}

\begin{proof}
	Since $\mathcal{K}_+$ and $\mathcal{K}_-$ are compact, 
	the operators $\mathcal{S}_+$ and $\mathcal{S}_-$ are compact perturbations of the identity operators in $\mathcal{U}$ and $\mathcal{V}$. Therefore, the assertion follows from Fredholm's alternative theorem, once we have proven injectivity of $\mathcal{S}_+$ and $\mathcal{S}_-$. 
	
	To prove injectivity of $\mathcal{S}_+$, we let $u \in \mathcal{U}$ satisfy $\mathcal{S}_+ u = 0$ which implies 
	$$
	(I - 6 \mathcal{K}_+) u = \alpha \mathcal{K}_+ \varphi + \beta \mathcal{K}_+ v_{\varphi},
	$$
	for some $(\alpha,\beta) \in \CC^2$. Taking the inner product in $\mathcal{H}$ with $\varphi'$ yields 
	$$
	0 = \alpha (\varphi',\mathcal{K}_+ \varphi)_{\mathcal{H}} + \beta (\varphi',\mathcal{K}_+ v_{\varphi})_{\mathcal{H}}.
	$$ 
	Since $(\varphi',\mathcal{K}_+ \varphi)_{\mathcal{H}} = 0$ 
	and $(\varphi',\mathcal{K}_+ v_{\varphi})_{\mathcal{H}} \neq 0$ by Lemma \ref{lem-non-degeneracy}, we have $\beta = 0$. With the help of $(I - 6 \mathcal{K}_+) u_{\varphi} = - \mathcal{K}_+ \varphi$, we obtain 
	$$
	u = -\alpha u_{\varphi} + \gamma \varphi',
	$$
	for some $\gamma \in \R$, but since $u \in \mathcal{U}$, we obtain $\alpha = \gamma = 0$ so that $u = 0$ and $\mathcal{S}_+$ is injective.
	
	Injectivity of $\mathcal{S}_-$ follows analogously.
\end{proof}

With the preliminary results in Lemmas \ref{lem-non-degeneracy}, \ref{lem-projection}, and \ref{lem-invertibility}, we now obtain 
the desired result in the following theorem. 

\begin{theorem}
	\label{th-discrete}
	The spectrum of the linear stability problem (\ref{lin-stab}) consists only of isolated eigenvalues of finite multiplicity 
	and admits no finite accumulation points. 
\end{theorem}

\begin{proof}
	By Lemma \ref{lem-non-degeneracy},  
	$\{ \mathcal{K}_+ \varphi, \mathcal{K}_+ v_{\varphi}\}$ is linearly independent of $\{ \varphi', u_{\varphi} \}$ so that 
	$$
{\rm Null}(P_{\mathcal{U}}) \cap {\rm Null}(I - \Pi_{\mathcal{U}}) = {\rm Null}(P_{\mathcal{V}}) \cap {\rm Null}(I - \Pi_{\mathcal{V}}) = \{ 0 \}.
	$$
As a result, the resolvent problem (\ref{resolvent-equiv}) is identical to the following two systems:
	\begin{equation}
	\label{resolvent-1}
	\begin{array}{l}
	P_{\mathcal{V}}( I - 3 \mathcal{K}_-) v = P_{\mathcal{V}}( \lambda \mathcal{K}_- u + \mathcal{K}_- f), \\
	-P_{\mathcal{U}} ( I - 6 \mathcal{K}_+) u = P_{\mathcal{U}}( \lambda \mathcal{K}_+ v + \mathcal{K}_+ g) \end{array}
	\end{equation}
	and
		\begin{equation}
	\label{resolvent-2}
	\begin{array}{l}
(I - \Pi_{\mathcal{V}}) ( I - 3 \mathcal{K}_-) v = (I - \Pi_{\mathcal{V}})( \lambda \mathcal{K}_- u + \mathcal{K}_- f), \\
	-(I - \Pi_{\mathcal{U}}) ( I - 6 \mathcal{K}_+) u = (I - \Pi_{\mathcal{U}}) ( \lambda \mathcal{K}_+ v + \mathcal{K}_+ g) \end{array}
	\end{equation}
Since $\mathcal{H}$ is given by the orthogonal sums $\mathcal{U}\oplus \mathcal{U}^{\perp}$ and $\mathcal{V}\oplus \mathcal{V}^{\perp}$, we decompose 
		\begin{equation}
\label{decomposition-res}
u = \tilde{u} + \alpha \varphi' + \beta u_{\varphi}, \quad 
v = \tilde{v} + \gamma \varphi + \delta v_{\varphi},
	\end{equation}
where $\tilde{u} \in \mathcal{U}$, $\tilde{v}\in \mathcal{V}$, and 
$\alpha,\beta,\gamma,\delta \in \CC$. Since 
$$
\begin{array}{l}
P_{\mathcal{U}} ( I - 6 \mathcal{K}_+) \varphi' = 0, \quad P_{\mathcal{U}} (I - 6 \mathcal{K}_+) u_{\varphi} = 0, \quad P_{\mathcal{U}} ( \mathcal{K}_+ \varphi)  = 0, \quad P_{\mathcal{U}} (\mathcal{K}_+ v_{\varphi}) = 0, \\
P_{\mathcal{V}} ( I - 3 \mathcal{K}_-) \varphi = 0, \quad P_{\mathcal{V}} (I - 3 \mathcal{K}_-) v_{\varphi} = 0, \quad P_{\mathcal{V}} ( \mathcal{K}_- \varphi')  = 0, \quad P_{\mathcal{V}} (\mathcal{K}_- u_{\varphi}) = 0, 
\end{array}
$$
the first system (\ref{resolvent-1}) with the decomposition (\ref{decomposition-res})  is equivalent to 
	\begin{equation*}
\begin{array}{l}
P_{\mathcal{V}}( I - 3 \mathcal{K}_-) \tilde{v} = P_{\mathcal{V}}( \lambda \mathcal{K}_- \tilde{u} + \mathcal{K}_- f), \\
-P_{\mathcal{U}} ( I - 6 \mathcal{K}_+) \tilde{u} = P_{\mathcal{U}}( \lambda \mathcal{K}_+ \tilde{v} + \mathcal{K}_+ g). \end{array}
\end{equation*}
By Lemma \ref{lem-invertibility}, the linear operators in the left-hand-side can be inverted so that the system can be rewritten in the equivalent form:
	\begin{equation*}
\left[ \begin{array}{l} \tilde{u} \\ \tilde{v} \end{array} \right] 
= \lambda \left[ \begin{array}{cc} 0 & -\mathcal{S}_+^{-1} P_{\mathcal{U}}  \mathcal{K}_+ |_{\mathcal{V}} \\ \mathcal{S}_-^{-1} P_{\mathcal{V}} \mathcal{K}_- |_{\mathcal{U}} & 0 \end{array} \right] \left[ \begin{array}{l} \tilde{u} \\ \tilde{v} \end{array} \right] + \left[ \begin{array}{c} -\mathcal{S}_+^{-1} P_{\mathcal{U}} \mathcal{K}_+ g \\ \mathcal{S}_-^{-1} P_{\mathcal{V}} \mathcal{K}_- f \end{array} \right].
\end{equation*}
Since the matrix operator in the right-hand side is a compact operator 
from $\mathcal{U} \times \mathcal{V}$ to itself, Fredholm's alternative theorem 
tells us that either the inhomogeneous system has a unique solution $(\tilde{u},\tilde{v}) \in \mathcal{U} \times \mathcal{V}$ for every $(f,g) \in \mathcal{H}\times \mathcal{H}$ or $\lambda$ is an isolated eigenvalue of finite multiplicity of the homogeneous system with $f = g = 0$. Moreover, isolated 
eigenvalues have no finite accumulation points. 

Let us show that the Fredholm's alternative carries over to the original system (\ref{resolvent-equiv}). To do so, we analyze the second system (\ref{resolvent-2}) with the decomposition (\ref{decomposition-res}). Since 
$$
(\varphi',(I - 6 \mathcal{K}_+) u)_{\mathcal{H}} = 0, \quad 
(\varphi,(I - 3 \mathcal{K}_-) v)_{\mathcal{H}} = 0,
$$
and
$$
\begin{array}{l}
(u_{\varphi},(I - 6 \mathcal{K}_+) u)_{\mathcal{H}} = ((I - 6 \mathcal{K}_+) u_{\varphi},u)_{\mathcal{H}} = -(\mathcal{K}_+\varphi,u)_{\mathcal{H}}, \\
(v_{\varphi},(I - 3 \mathcal{K}_-) v)_{\mathcal{H}} = ((I - 3 \mathcal{K}_-)  v_{\varphi}, v)_{\mathcal{H}} = (\mathcal{K}_- \varphi', v)_{\mathcal{H}}, \end{array}
$$
we obtain the following system of four equations 
\begin{align*}
0 &= \lambda \left[ (\varphi, \mathcal{K}_- \tilde{u})_{\mathcal{H}} + 
\alpha (\varphi, \mathcal{K}_- \varphi')_{\mathcal{H}} + \beta 
(\varphi, \mathcal{K}_- u_{\varphi})_{\mathcal{H}} \right] + (\varphi, \mathcal{K}_- f)_{\mathcal{H}}, \\
(\mathcal{K}_- \varphi', \tilde{v})_{\mathcal{H}} 
+ \delta (\mathcal{K}_- \varphi', v_{\varphi})_{\mathcal{H}} &= \lambda \left[ (v_{\varphi}, \mathcal{K}_- \tilde{u})_{\mathcal{H}} + 
\alpha (v_{\varphi}, \mathcal{K}_- \varphi')_{\mathcal{H}} + \beta 
(v_{\varphi}, \mathcal{K}_- u_{\varphi})_{\mathcal{H}} \right] + (v_{\varphi}, \mathcal{K}_- f)_{\mathcal{H}}, \\
0 &= \lambda \left[ (\varphi', \mathcal{K}_+ \tilde{v})_{\mathcal{H}} + 
\gamma (\varphi', \mathcal{K}_+ \varphi)_{\mathcal{H}} + \delta 
(\varphi', \mathcal{K}_+ v_{\varphi})_{\mathcal{H}} \right] + (\varphi', \mathcal{K}_+ g)_{\mathcal{H}}, \\
-(\mathcal{K}_+ \varphi, \tilde{u})_{\mathcal{H}} 
- \beta (\mathcal{K}_+  \varphi, u_{\varphi})_{\mathcal{H}} &= \lambda \left[ (u_{\varphi}, \mathcal{K}_+ \tilde{v})_{\mathcal{H}} + 
\gamma (u_{\varphi}, \mathcal{K}_+ \varphi)_{\mathcal{H}} + \delta 
(u_{\varphi}, \mathcal{K}_+ v_{\varphi})_{\mathcal{H}} \right] + (u_{\varphi}, \mathcal{K}_+ g)_{\mathcal{H}},
\end{align*}
where we have used 
$$
(\mathcal{K}_- \varphi',\varphi)_{\mathcal{H}} = (\mathcal{K}_+ \varphi, \varphi')_{\mathcal{H}} = 0
$$
and our definition (\ref{inner-product}) with the complex conjugation applied to the first element of the inner product in $\mathcal{H}$.

The system can be written in the matrix-vector form by using matrices $A_+$ and $A_-$ in (\ref{A-plus}) and (\ref{A-minus}):
\begin{align*}
& \lambda A_- \left[ \begin{array}{c} \alpha \\ \beta \end{array} \right] 
- A_- \left[ \begin{array}{c} \delta \\ 0 \end{array} \right]  = 
\left[ \begin{array}{c} 0 \\ (\mathcal{K}_- \varphi', \tilde{v})_{\mathcal{H}}  \end{array} \right] - \lambda \left[ \begin{array}{c} (\varphi, \mathcal{K}_- \tilde{u})_{\mathcal{H}} \\ (v_{\varphi}, \mathcal{K}_- \tilde{u})_{\mathcal{H}} \end{array} \right] - \left[ \begin{array}{c} (\varphi, \mathcal{K}_- f)_{\mathcal{H}} \\ (v_{\varphi}, \mathcal{K}_- f)_{\mathcal{H}} \end{array} \right], \\
& \lambda A_+ \left[ \begin{array}{c} \gamma \\ \delta \end{array} \right] 
+ A_+ \left[ \begin{array}{c} \beta \\ 0 \end{array} \right] = 
-\left[ \begin{array}{c} 0 \\ (\mathcal{K}_+ \varphi, \tilde{u})_{\mathcal{H}}  \end{array} \right] - \lambda \left[ \begin{array}{c} (\varphi', \mathcal{K}_+ \tilde{v})_{\mathcal{H}} \\ (u_{\varphi}, \mathcal{K}_+ \tilde{v})_{\mathcal{H}} \end{array} \right] - \left[ \begin{array}{c} (\varphi', \mathcal{K}_+ g)_{\mathcal{H}} \\ (u_{\varphi}, \mathcal{K}_+ g)_{\mathcal{H}} \end{array} \right], 
\end{align*}
where we have used 
$$
A_-\left[ \begin{array}{c} 1 \\ 0 \end{array} \right]  = \left[ \begin{array}{c} 0 \\ (v_{\varphi},\mathcal{K}_- \varphi')_{\mathcal{H}} \end{array} \right], \quad A_+ \left[ \begin{array}{c} 1 \\ 0 \end{array} \right]  = \left[ \begin{array}{c} 0 \\ (u_{\varphi}, \mathcal{K}_+ \varphi)_{\mathcal{H}} \end{array} \right]
$$
and all entries of $A_+$ and $A_-$ are real-valued. By Lemma \ref{lem-non-degeneracy}, 
matrices $A_+$ and $A_-$ are invertible and so is the matrix 
\begin{equation*}
\left[ \begin{array}{cc} A_- & 0 \\ 0 & A_+ \end{array} \right] \left[ \begin{array}{cccc} \lambda & 0 & 0 & -1 \\
0 & \lambda & 0 & 0 \\ 0 & 1 & \lambda & 0 \\ 0 & 0 & 0 & \lambda \end{array} \right], \quad \lambda \neq 0.
\end{equation*}
Therefore, for every $\lambda \neq 0$, there exists a unique solution $(\alpha,\beta,\gamma,\delta) \in \CC^4$ for every $(\tilde{u},\tilde{v}) \in \mathcal{U}\times \mathcal{V}$ and $(f,g) \in \mathcal{H}\times \mathcal{H}$.

We recall that the resolvent equation (\ref{resolvent-equiv}) is considered 
for $\lambda \neq 0$ since the quadruple eigenvalue $\lambda = 0$ is prescribed by Lemma \ref{lem-quadruple}. Thus, by Fredholm's alternative theorem, if the inhomogeneous system has a unique solution $(\tilde{u},\tilde{v}) \in \mathcal{U}\times \mathcal{V}$, 
then there exists a unique solution $(\alpha,\beta,\gamma,\delta) \in \CC^4$ 
and so the unique solution $(u,v) \in \mathcal{H}\times \mathcal{H}$ to the resolvent equation (\ref{resolvent-equiv}) for every 
$(f,g) \in \mathcal{H}\times \mathcal{H}$. Hence, this $\lambda$ belongs to the resolvent set of the linear stability problem (\ref{lin-stab}). Alternatively, 
there are finitely many linearly independent solutions for  $(\tilde{u},\tilde{v}) \in \mathcal{U}\times \mathcal{V}$ to the homogeneous system with $f = g = 0$ 
and together with the unique solutions for $(\alpha,\beta,\gamma,\delta) \in \CC^4$, this gives the same number of linearly independent solutions for 
$(u,v) \in \mathcal{H}\times \mathcal{H}$ to the homogeneous system (\ref{resolvent-equiv}) with $f = g = 0$. Hence, this $\lambda$ is an isolated eigenvalue of the linear stability problem (\ref{lin-stab}).
\end{proof}

\subsection{Proof of Theorem \ref{theorem-lin}}

By Corollaries \ref{cor-minus} and \ref{cor-plus}, 
the spectrum of $\mathcal{L}_-$ and $\mathcal{L}_+$ in $\mathcal{H}$ consists of simple isolated eigenvalues. The orderings (\ref{spectrum-plus}) and (\ref{spectrum-minus})  
follow from the ordering of the zero eigenvalue of $\mathcal{L}_-$ and $\mathcal{L}_+$ in Remarks 
\ref{rem-minus} and \ref{rem-plus} respectively. 

By Theorem \ref{th-discrete}, the spectrum of the stability problem (\ref{lin-stab-H}) consists of isolated eigenvalues of finite multiplicity and admits no finite accumulation points. By Lemma \ref{lem-quadruple}, it includes a quadruple zero eigenvalue associated to the symmetries (\ref{symmetry}). The nonzero isolated eigenvalues occur in pairs 
$\pm \lambda$ since if $\lambda$ is an eigenvalue of the linear stability 
problem (\ref{lin-stab-H}) with the eigenvector $(u,v)^T$, 
then $-\lambda$ is an  eigenvalue with the eigenvector $(u,-v)^T$. 
It remains to prove that all eigenvalues are purely imaginary and hence 
satisfy the ordering (\ref{spectrum-stab}). This follows from the characterization of nonzero eigenvalues $\lambda$ of the stability problem (\ref{lin-stab-H}) for $(u,v) \in \mathcal{V}\times \mathcal{U}$.

To be precise, let us define the following constrained operators 
\begin{align*}
\mathcal{T}_+ &:= \mathcal{L}_+ |_{\{\varphi' \}^{\perp}} \; : \; \mathcal{H}^2_+ |_{\{\varphi' \}^{\perp}} \mapsto \mathcal{H}|_{\{\varphi'\}^{\perp}}, \\
\mathcal{T}_- &:= \mathcal{L}_- |_{\{\varphi \}^{\perp}} \; : \; \mathcal{H}^2_- |_{\{\varphi \}^{\perp}} \mapsto \mathcal{H}|_{\{\varphi\}^{\perp}}
\end{align*}
It is clear that $\mathcal{T}_+$ and $\mathcal{T}$ are invertible operators with bounded inverses $\mathcal{T}_+^{-1}$ and $\mathcal{T}_-^{-1}$.

If $(\varphi',v)_{\mathcal{H}} = 0$ and $(u_{\varphi},v)_{\mathcal{H}} = 0$, then the second equation of system (\ref{lin-stab-H}) is solved by 
$$
u = -\lambda \left( \mathcal{T}_+^{-1} v - \frac{(\mathcal{T}_+^{-1} v, v_{\varphi})_{\mathcal{H}}}{(\varphi', v_{\varphi})_{\mathcal{H}}} \varphi' \right),
$$ 
where $(\varphi',v_{\varphi})_{\mathcal{H}} < 0$. By construction, 
$(v_{\varphi},u)_{\mathcal{H}} = 0$ is satisfied. In addition, 
$(\varphi, u)_{\mathcal{H}} = 0$ since 
$$
(\varphi,u)_{\mathcal{H}} = -(\mathcal{L}_+ u_{\varphi},u)_{\mathcal{H}} 
= -(u_{\varphi},\mathcal{L}_+ u)_{\mathcal{H}} = \lambda(u_{\varphi},v)_{\mathcal{H}} = 0.
$$
Thus, if $v \in \mathcal{U}$ and $\lambda \neq 0$, a unique solution 
of the second equation of system (\ref{lin-stab-H}) exists for $u \in \mathcal{V}$. The first equation of system (\ref{lin-stab-H}) can then be written as the generalized eigenvalue problem 
\begin{equation}
\label{gen-eig}
\mathcal{L}_- v = -\lambda^2 \left( \mathcal{T}_+^{-1} v -  \frac{(\mathcal{T}_+^{-1} v, v_{\varphi})_{\mathcal{H}}}{(\varphi', v_{\varphi})_{\mathcal{H}}} \varphi' \right),
\end{equation}
for which the smallest eigenvalue $-\lambda_0^2$ can be obtained from the Rayleigh quotient 
\begin{equation}
\label{Rayleigh}
-\lambda_0^2 = \inf_{{\tiny \begin{array}{l} v \in \mathcal{H} \backslash \{0\} \\ (\varphi',v)_{\mathcal{H}} = 0 \\
(u_{\varphi},v)_{\mathcal{H}} = 0		\end{array}}} 
	\frac{Q_-(v)}{(\mathcal{T}_+^{-1} v, v)_{\mathcal{H}}},
\end{equation}
where $Q_-(v) := (\mathcal{L}_- v, v)_{\mathcal{H}}$ is given by (\ref{quad-minus}). Since $(\varphi',v)_{\mathcal{H}} = 0$ and the spectrum of $\mathcal{L}_+$ satisfies (\ref{spectrum-plus}), we have 
$(\mathcal{T}_+^{-1} v, v)_{\mathcal{H}} > 0$. On the other hand, it follows from Appendix A in \cite{W85} that 
\begin{itemize}
	\item $Q_-(v) \geq 0$, for $v$ satisfying $(\varphi',v)_{\mathcal{H}} = 0$, 
	\item $Q_-(v) = 0$ is attained only at $v \in {\rm span}(\varphi) \subset \mathcal{H}$,
\end{itemize}
provided that 
$$	
(\mathcal{T}_-^{-1} \varphi', \varphi')_{\mathcal{H}}  = (v_{\varphi},\varphi')_{\mathcal{H}} < 0,
$$
which is satisfied. However, $(u_{\varphi},\varphi)_{\mathcal{H}} \neq 0$ 
so that if $v \in \mathcal{U}$, then $Q_-(v) > 0$. Hence $-\lambda_0^2 > 0$ follows from the Rayleight quotient (\ref{Rayleigh}) so that all nonzero eigenvalues of the stability problem (\ref{lin-stab-H}) satisfy $-\lambda^2 > 0$ and are thus located on $i \R$ away from the quadruple zero eigenvalue.

The proof of Theorem \ref{theorem-lin} is complete.

\section{Existence of traveling dark solitons}
	\label{sec-existence}

Here we study the construction of traveling dark solitons for arbitrary wave speed $c$ and in the limit $c \to 0$ and thus provide the proof of Theorem \ref{theorem-dark}. 

If $U_c$ is a solution 
of the differential equation (\ref{nls-trav}), then $\bar{U}_c$ is 
the solution of the same equation for $-c$. Therefore, it is sufficient to 
prove the existence result for $c > 0$. The following 
lemma yields the existence of traveling dark solitons. 

\begin{lemma}
	\label{lem-dark}
	For every $c \in (0,\infty)$, there exists the dark soliton with the profile $U_c \in \mathcal{F} \cap C^{\infty}(\mathbb{R})$ satisfying (\ref{nls-trav}) and (\ref{bc-dark}) such that the mapping $c \mapsto U_c$ is $C^{\infty}$ on $(0,\infty)$.
\end{lemma}

\begin{proof}
	By using the polar form $U_c = \sqrt{\rho_c} e^{i \phi_c}$, we rewrite the differential equation (\ref{nls-trav}) as the following system:
	\begin{align}
	\label{system-ode}
	\left\{ \begin{array}{l}
	(\rho_c \phi_c')' - c (1-\rho_c) \rho_c' = 0, \\
	2 \rho_c \rho_c'' - (\rho_c')^2 + 8 (1-\rho_c) \rho_c^2 + 8 c (1-\rho_c) \rho_c^2 \phi_c' - 4 \rho_c^2 (\phi_c')^2 = 0. \end{array} \right.
	\end{align}
	
	The first equation in system (\ref{system-ode}) can be integrated under the boundary conditions (\ref{bc-dark}) 
	to the form 
\begin{equation}
\label{phi}
	\phi_c' = -\frac{c}{2\rho_c} (1-\rho_c)^2,
\end{equation}
	so that $\lim\limits_{|\xi| \to \infty} \phi_c'(\xi) = 0$ if $\lim\limits_{|\xi| \to \infty} \rho_c(\xi) = 1$. 
	
	Substituting (\ref{phi}) to the second equation in system (\ref{system-ode}) 
	and using the transformation $\rho_c = q_c^2$, we obtain the second-order differential equation
	$$
	q_c'' + 2(1-q_c^2) q_c - \frac{c^2}{q_c} (1-q_c^2)^3 - \frac{c^2}{4 q_c^3} (1-q_c^2)^4 = 0,
	$$
	which can be integrated under the boundary conditions (\ref{bc-dark}) to the form
	$$
	(q_c')^2 - (1- q_c^2)^2 + \frac{c^2}{4 q_c^2} (1-q_c^2)^4 = 0.
	$$
Unfolding the transformation $\rho_c = q_c^2$ yields 	
\begin{equation}
\label{rho-scalar}
	(\rho_c')^2 = (1-\rho_c)^2 \left[ 4 \rho_c - c^2 (1-\rho_c)^2 \right],
\end{equation}
which is the first-order invariant of the second-order equation 
\begin{equation}
\label{rho-second-order}
\rho_c'' + 2 (1-\rho_c)(3 \rho_c - 1 - c^2(1-\rho_c)^2) = 0,
\end{equation}
with the unique choice of the integration constant due to the boundary conditions (\ref{bc-dark}).

The critical point $(\rho,\rho') = (1,0)$ is a saddle point 
of the second-order equation (\ref{rho-second-order}) 
with two homoclinic orbits located in $[\rho_-(c),1] \times \mathbb{R}$ and $[1,\rho_+(c)] \times \mathbb{R}$, where $\rho_{\pm}(c)$ are turning points of the homoclinic orbits. The turning points correspond to the values of $\xi_0$ for which $\rho'(\xi_0) = 0$ and hence are obtained from the roots of the quadratic equation $4 \rho = c^2 (1-\rho)^2$ that follows from (\ref{rho-scalar}). The turning points are given explicitly by 
	$$
	\rho_{\pm}(c) = \frac{c^2 + 2 \pm 2 \sqrt{1+c^2}}{c^2}
	$$
	and they exist for every $c > 0$. Moreover, 
	$$
	\rho_-(c) = \frac{1}{4} c^2 + \mathcal{O}(c^4), \quad \rho_+(c) = \frac{4}{c^2} + \mathcal{O}(1), \quad \mbox{\rm as} \;\; c \to 0.
	$$
	There exists exactly one (up to translation) solution $\rho_c \in [\rho_-(c),1]$ of the first-order invariant (\ref{rho-scalar}) in the set $\mathcal{F}$. For every $c \in (0,\infty)$, the solution satisfies
	$$
	\rho_c(\xi) \geq \rho_-(c) > 0, \quad \forall \xi \in \R
	$$ 
	so that $U_c \in C^{\infty}(\mathbb{R})$ in the polar form. 
	Moreover, $\rho_c$, $\phi_c$, and thus $U_c$ are $C^{\infty}$ 
	with respect to $c$ on $(0,\infty)$.
\end{proof}

\begin{remark}
	Another solution of the first-order invariant (\ref{rho-scalar}) exists for $\rho \in [1,\rho_+(c)]$. This solution is not used for the scopes of this work.
\end{remark}

\begin{remark}
	\label{remark-rate}
	Linearization of the second-order equation (\ref{rho-second-order}) at $(\rho,\rho') = (1,0)$ shows that the saddle point is associated with the same exponential rate $\pm 2$ independently of the speed parameter $c \in \mathbb{R}$. Consequently, up to translation in $\xi$, $\rho_c$ satisfies
	$\rho_c(\xi) = 1 + A_c e^{-2|\xi|} + {\rm o}(e^{-2|\xi|})$ as $|\xi| \to \infty$ with a unique choice for $A_c < 0$.
\end{remark}

\begin{remark}
	The $C^{\infty}$ smoothness of $U_c$ at $c = 0$ cannot be concluded from the polar form. Indeed, we have $\rho_c = \rho_0 + \mathcal{O}_{H^2}(c^2)$, see the proof of Theorem \ref{lem-conv-to-0}, where 
	$\rho_0(\xi) = \tanh^2(\xi)$ is the suitable solution of $(\rho_0')^2 = 4 \rho_0 (1-\rho_0)^2$. However, since $\rho_0(0) = 0$, the
	phase factor $\phi_c(\xi)$ is singular at $\xi = 0$. If we use the formal expansion 
	$\phi_c' = c \phi_1' + \mathcal{O}(c^3)$, then 
	$$
	\phi_1' = -\frac{(1-\rho_0)^2}{2 \rho_0}
	$$
	so that $\phi_1(\xi) = \coth(2\xi)$ is singular at $\xi = 0$. The polar form 
	$U_c = \sqrt{\rho_c} e^{i\phi_c}$ gives formally 
	$$
	U_c(\xi) = \tanh(\xi) + \frac{i c \cosh(2\xi)}{2 \cosh^2(\xi)} + \mathcal{O}(c^2),
	$$
	where the second term coincides with the explicit expression 
	for $-2 i c v_{\varphi}$, see (\ref{eigen-gener-exact}). 
This formal expansion cannot be justified from the polar form because of the singularity at $\xi = 0$.
\end{remark}

The following theorem justifies the expansion of $U_c$ as $c \to 0$
by using the real and imaginary coordinates for $U_c = u_c + i v_c$ in addition 
to the variable $\rho_c = |U_c|^2$.

\begin{theorem}
	\label{lem-conv-to-0} 
	There exists $c_0 > 0$ and $A_0 > 0$ such that for every $c \in (-c_0,c_0)$, we have 
	\begin{equation}
	\label{U-dark-expansion}
	\| U_c - \varphi + 2 i c v_{\varphi} \|_{\mathcal{H}^2_-} \leq A_0 c^2,
	\end{equation}
	where $v_{\varphi}$ is defined by  (\ref{eigen-gener-exact}).
\end{theorem}

\begin{proof}
	Let us substitute $U_c = u_c + i v_c$ in the differential equation (\ref{nls-trav}) and obtain the following system 
	\begin{align}
	\label{system-u-v}
	\left\{ \begin{array}{l}
	u_c'' + 2(1-u_c^2-v_c^2) u_c + 2c (1-u_c^2-v_c^2) v_c' = 0, \\
	v_c'' + 2(1-u_c^2-v_c^2) v_c - 2c (1-u_c^2-v_c^2) u_c' = 0. \end{array} \right.
	\end{align}
	The relevant analysis is obtained in three steps.\\
	
	{\em Step 1.} By Lemma \ref{lem-dark}, $\rho_c = |U_c|^2$ is a solution of the second-order differential equation (\ref{rho-second-order}). We decompose $\rho_c = \rho_0 + \eta_c$, where $\rho_0 := \varphi^2$ is a solution of the limiting equation
\begin{equation}
\label{rho-limiting}
	\rho_0'' + 2(1-\rho_0) (3 \rho_0 - 1) = 0
\end{equation}
	and $\eta_c$ is a suitable solution of the residual equation
\begin{equation}
\label{eta-c-problem}
	L_0 \eta_c + 2 c^2 (1-\varphi^2 - \eta_c)^2 + 6 \eta_c^2 = 0,
\end{equation}
	associated with the Schr\"{o}dinger operator $L_0 : H^2(\R) \subset L^2(\R) \mapsto L^2(\R)$ given by
	$$
	L_0 := -\partial_{\xi}^2 + 4  - 12 {\rm sech}^2(\cdot).
	$$
Since ${\rm Ker}(L_0) = {\rm span} (\varphi \varphi')$ and $\varphi \varphi'$ is odd, the linear operator $L_0$ is injective on the space of even functions in $H^2(\R)$ denoted as $H^2_{\rm even}(\R)$. Since $-12 {\rm sech}^2(\cdot)$ is a relatively compact perturbation to the unbounded positive operator $-\partial_{\xi}^2 + 4$, it is also surjective in $L^2(\R)$. Thus, $L_0 : H^2_{\rm even}(\R) \subset L^2(\R) \mapsto L^2(\R)$ is invertible with a bounded inverse. By the implicit function 
theorem, there exist $c_1 > 0$, $A_1 > 0$, and a unique solution $\eta_c \in H^2_{\rm even}(\R)$ to (\ref{eta-c-problem}) for every $c \in (-c_1,c_1)$ such that 
\begin{equation}
\label{eta-c-bound}
\| \eta_c \|_{H^2} \leq A_1 c^2, \quad c \in (-c_1,c_1).
\end{equation}
We also claim that there exists $B_1 > 0$ such that
\begin{equation}
\label{eta-c-infinity}
| \eta_c(\xi)|  \leq B_1 c^2 (1-\varphi(\xi)^2), \qquad \forall \xi \in \R,
\end{equation}	
which agrees with the $c$-independent rate $\eta_c(\xi) = \mathcal{O}(e^{-2|\xi|})$ as $|\xi| \to \infty$ (Remark \ref{remark-rate}). In order to justify the bound (\ref{eta-c-infinity}), we rewrite the residual equation (\ref{eta-c-problem}) in the integral form, 
\begin{align}
\label{eta-integral}
\eta_c(\xi) = f_1(\xi) \int_0^{\xi} f_2(\xi') F(\xi') d\xi' - f_2(\xi) \int_{-\infty}^{\xi} f_1(\xi') F(\xi') d\xi', 
\end{align}
where $f_{1,2}$ are homogeneous solutions of $L_0 f = 0$ normalized by their Wronskian 
$$
f_1(\xi) f_2'(\xi) - f_1'(\xi) f_2(\xi) = 1, \quad \forall \xi \in \R,
$$ 
and $F := -2 c^2 (1 - \varphi^2 - \eta_c)^2 - 6 \eta_c^2$. We choose 
$$
f_1(\xi) = \varphi(\xi) \varphi'(\xi) = {\rm sech}^2(\xi) \tanh(\xi),
$$
so that $f_2$ is obtained computationally as 
$$
f_2(\xi) = \frac{1}{8} \left[ 2 \cosh^2(\xi)  + 5 
- 15 {\rm sech}^2(\xi) + 15 \xi {\rm sech}^2(\xi) \tanh(\xi) \right] =: \cosh^2(\xi) f(\xi).
$$
Note that the constant of integrations in (\ref{eta-integral}) have been choosen 
to ensure that $\eta_c$ is even if $F$ is even in $\xi$ and 
that $\eta_c(\xi) \to 0$ as $|\xi| \to \infty$. The integral equation (\ref{eta-integral}) can be rewritten in the exponentially weighted form 
for $\gamma_c := \cosh^2(\cdot) \eta_c$ as the fixed-point equation 
$\gamma_c = \mathcal{A}(\gamma_c)$, where
\begin{align*}
\mathcal{A}(\gamma_c) & := f \cosh^4(\cdot)  \int_{-\infty}^{\xi} {\rm sech}^6(\cdot) \varphi  \left[ 2 c^2 (1 - \gamma_c)^2 - 6 \gamma_c^2 \right] d\xi' \\
& \quad - 
\varphi \int_0^{\xi} {\rm sech}^2(\cdot) f \left[ 2 c^2 (1 - \gamma_c)^2 - 6 \gamma_c^2 \right] d\xi'.
\end{align*}
Since the integral operator is closed in the space of even and bounded 
functions and satisfies the estimates 
\begin{align*}
& \| A(\gamma_c) \|_{L^{\infty}} \leq C_1 (c^2 + \| \gamma_c \|^2_{L^{\infty}}), \\
& \| A(\gamma_c) - A(\gamma_c')\|_{L^{\infty}} \leq C_2 (c^2 + \| \gamma_c \|_{L^{\infty}}) \| \gamma_c - \gamma_c' \|_{L^{\infty}}
\end{align*}
for some $c$-independent constants $C_1,C_2 > 0$, the bound (\ref{eta-c-infinity}) follows by the implicit function theorem 
for $\gamma_c \in L^{\infty}(\R)$ satisfying $\| \gamma_c \|_{L^{\infty}} \leq B_1 c^2$.\\

{\em Step 2.} We are now ready to analyze system (\ref{system-u-v}) for $u_c$ and $v_c$. Using the decomposition $u_c = \varphi + \tilde{u}_c$, $v_c = \tilde{v}_c$, we obtain the following system of equations:
	\begin{equation}
	\label{system-u-v-c}
	\left\{ \begin{array}{l}
	L_+ \tilde{u}_c = 2c (1-\varphi^2 - \eta_c) \tilde{v}_c' - 2 \eta_c \tilde{u}_c - 
	2 \varphi (\tilde{u}_c^2+\tilde{v}_c^2), \\
	L_- \tilde{v}_c =  - 2c (1-\varphi^2 - \eta_c) (\varphi' + \tilde{u}_c') -2 \eta_c \tilde{v}_c, \end{array} \right.
	\end{equation}
where $L_+$ and $L_-$ are given by (\ref{L-plus}) and (\ref{L-minus}) 
and $\eta_c := \rho_c - \rho_0 = 2 \varphi \tilde{u}_c + \tilde{u}_c^2+\tilde{v}_c^2$. 

We are looking for solutions $\tilde{u}_c \in \mathcal{H}^2_-$ being odd in $\xi$ and $\tilde{v}_c \in \mathcal{H}^2_-$ being  even in $\xi$. For such functions, the right-hand side of the first equation in system (\ref{system-u-v-c}) is odd in $\xi$ and the right-hand side of the second equation in system (\ref{system-u-v-c}) is even in $\xi$.

Since ${\rm Ker}(\mathcal{L}_-) = {\rm span}(\varphi) \subset \mathcal{H}$ is odd in $\xi$, the operator $\mathcal{L}_-$ is invertible on  even functions in $\mathcal{H}$. Due to the exponential decay (\ref{eta-c-infinity}) and 
smallness of $c$, $2 \eta_c$ is a small perturbation to $L_-$ 
so that $\mathcal{L}_- + 2 {\rm cosh}^2(\cdot) \eta_c$ is also 
invertible on even functions in $\mathcal{H}$. By the implicit function theorem, 
there exists $c_2 > 0$, $A_2 > 0$ and a unique even solution 
$\tilde{v}_c \in \mathcal{H}^2_-$ to the second equation in system (\ref{system-u-v-c}) 
for every $c \in (-c_2,c_2)$ and every odd $\tilde{u}_c \in \mathcal{H}^2_-$
such that 
\begin{equation}
\label{bound-v-c}
\| \tilde{v}_c \|_{\mathcal{H}^2_-} \leq A_2 |c|, \quad c \in (-c_2,c_2).
\end{equation}
We claim that there exists the limit
\begin{equation}
\label{beta-c-limit}
\beta_c := \lim_{|\xi| \to \infty} \tilde{v}_c(\xi)
\end{equation}
and there exists $B_2 > 0$ such that
\begin{equation}
\label{bound-v-c-infinity}
| \tilde{v}_c(\xi) - \beta_c |  \leq B_2 |c| |\xi|  (1-\varphi(\xi)^2), \qquad \forall \xi \in \R.
\end{equation}	
In order to justify (\ref{beta-c-limit}) and (\ref{bound-v-c-infinity}), we again write the second equation in system (\ref{system-u-v-c}) as the integral equation
\begin{align}
\label{v-integral}
\tilde{v}_c(\xi) = \tanh(\xi) \int_0^{\xi} \left[ \xi' \tanh(\xi') - 1 \right] G(\xi') d\xi' - \left[ \xi \tanh(\xi) - 1 \right] \int_{-\infty}^{\xi} \tanh(\xi') G(\xi') d\xi', 
\end{align}
where we have used two solutions of $L_- g = 0$ given by 
$g_1 = \varphi$ and $g_2 = \xi \varphi - 1$ normalized by the Wronskian identity 
$$
g_1(\xi) g_2'(\xi) - g_1'(\xi) g_2(\xi) = 1, \quad \forall \xi \in \R
$$
and $G := - 2c (1-\varphi^2 - \eta_c) (\varphi' + \tilde{u}_c') -2 \eta_c \tilde{v}_c$. The limit of integration in the second term in (\ref{v-integral}) ensures that the second term converges to $0$ exponentially fast as $|\xi| \to \infty$ since $G$ is even and has the exponential weight if $\tilde{u}_c,\tilde{v}_c \in \mathcal{H}^2_-$ due to the exponential decay (\ref{eta-c-infinity}). The first term in (\ref{v-integral}) is even and converges to a finite limit 
as $|\xi| \to \infty$, which defines $\beta_c$ uniquely since 
$\tilde{v}_c \in \mathcal{H}^2_-$ is already uniquely defined for a given $\tilde{u}_c \in \mathcal{H}^2_-$. Thus, 
the limit (\ref{beta-c-limit}) has been confirmed. 
The bound (\ref{bound-v-c-infinity}) is obtained due to the linearly growing terms in the integral equation (\ref{v-integral}) and the exponentially decaying weights $1-\varphi^2$ under the integration signs.

A simple computation shows that 
$$
\tilde{v}_c = -2c v_{\varphi} + \mathcal{O}_{\mathcal{H}^2_-}(c^3),
$$
where $v_{\varphi}$ is given by (\ref{eigen-gener-exact}) and where the $\mathcal{O}_{\mathcal{H}^2_-}(c^3)$ order is justified if 
$\tilde{u}_c = \mathcal{O}_{\mathcal{H}^2_-}(c^2)$ as $c \to 0$. Using (\ref{beta-c-limit}), we also obtain $\beta_c = c + \mathcal{O}(c^3)$ 
so that $\beta_c = \mathcal{O}(c)$. \\

{\em Step 3.}  Finally, we use the decomposition $\tilde{u}_c = \alpha_c \varphi + \hat{u}_c$  with 
\begin{equation*}
\alpha_c := \sqrt{1 - \beta_c^2} - 1 = \mathcal{O}(c^2)
\end{equation*}
obtained from the appropriate root of the quadratic equation $2 \alpha + \alpha^2 + \beta^2 = 0$. With this decomposition, we  rewrite the first equation in system (\ref{system-u-v-c}) in the form
\begin{align*}
L_+ \hat{u}_c = \; 2c (1-\varphi^2 - \eta_c) \tilde{v}_c' - 2 \eta_c 
(\alpha_c \varphi + \hat{u}_c) - 2 \varphi (\alpha_c^2 \varphi^2 + 2 \alpha_c \varphi \hat{u}_c + \hat{u}_c^2 +\tilde{v}_c^2) - 4 \alpha_c \varphi^3,
\end{align*}
where we have used $L_+ \varphi = 4 \varphi^3$. Since $2 \alpha_c + \alpha_c^2 + \beta_c^2 = 0$, the right-hand side is coverges to $0$ exponentially fast due to the exponential bounds (\ref{eta-c-infinity}) and (\ref{bound-v-c-infinity}). 
Therefore, we can consider the linear operator $L_+ : H^2(\R) \subset L^2(\R) \mapsto L^2(\R)$ without the exponential weights. Since ${\rm Ker}(L_+) = {\rm span}(\varphi') \subset L^2(\R)$ is even in $\xi$, the operator $L_+$ is invertible on odd functions in $L^2(\R)$. By the implicit function theorem, 
there exists $c_3 > 0$, $A_3 > 0$ and a unique odd solution 
$\hat{u}_c \in H^2(\R)$ 
for every $c \in (-c_3,c_3)$ and every even $\tilde{v}_c \in \mathcal{H}^2_-$
satisfying (\ref{bound-v-c}), (\ref{beta-c-limit}), and (\ref{bound-v-c-infinity}) such that 
\begin{equation}
\label{bound-u-c}
\| \hat{u}_c \|_{H^2} \leq A_3 c^2, \quad c \in (-c_3,c_3).
\end{equation}
The decompositions $u_c = \varphi + \alpha_c \varphi + \mathcal{O}_{H^2}(c^2)$ and $v_c = -2c v_{\varphi} + \mathcal{O}_{\mathcal{H}^2_-}(c^3)$ 
with $\alpha_c = \mathcal{O}(c^2)$ justify the bound (\ref{U-dark-expansion}) for some $A_0$ and $c_0 := {\rm min}(c_1,c_2,c_3)$.
\end{proof}

\begin{remark}
	\label{remark-phases}
	The decomposition of $U_c$ in the proof of Theorem \ref{lem-conv-to-0} can be written in the form $U_c = (1+\alpha_c) \varphi + \hat{u}_c + i \tilde{v}_c$. Since $\varphi(\xi) \to \pm 1$, $\hat{u}_c(\xi) \to 0$, and $\tilde{v}_c(\xi) \to \beta_c$ as $|\xi| \to \infty$, the phases $\theta_{\pm}(c)$ in the boundary conditions (\ref{bc-dark}) are computed explicitly as 
	$$
	\theta_+(c) = \arctan\frac{\beta_c}{1 + \alpha_c}, \quad 
	\theta_-(c) = \pi - \arctan \frac{\beta_c}{1 + \alpha_c}.
	$$
	Substituting $1 + \alpha_c = \sqrt{1 - \beta_c^2}$ yields 
	$$
	\theta_+(c) = \arcsin \beta_c, \quad 
	\theta_-(c) = \pi - \arcsin \beta_c,
	$$
	where $\beta_c = \mathcal{O}(c)$ as $c \to 0$ 
\end{remark}

\begin{remark}
	\label{remark-proximity}
It follows from the bound (\ref{eta-c-infinity}) that there is $A > 0$ such that  for every $c \in (-c_0,c_0)$ and every $f \in \mathcal{H}$, it holds that 
	\begin{equation}
	\label{proximity}
	\left| \int_{\R} |U_c|^2 f^2 dx -  \int_{\R} \varphi^2 f^2 dx \right| 
	\leq A c^2 \| f \|^2_{\mathcal{H}}.
	\end{equation}	
\end{remark}

In order to obtain the energetic stability of the black soliton, we 
need to compute the asymptotic expansion of the mass $M(u)$ and the momentum $P(u)$  at $u = U_c$ as $c \to 0$. This asymptotic expansion is given by the following lemma.

\begin{lemma}
	\label{lem-momentum}
	Let $U_c$ be the dark soliton of Lemma \ref{lem-dark} and Theorem \ref{lem-conv-to-0}. Then, 
\begin{equation}
\label{expansion-M-c}
M(U_c) = \frac{4}{3} + \mathcal{O}(c^2) \quad \mbox{\rm as} \;\; c \to 0
\end{equation}
and	
\begin{equation}
\label{expansion-P-c}
	P(U_c) = -\pi + \frac{16}{5} c + \mathcal{O}(c^3) \quad \mbox{\rm as} \;\; c \to 0,
\end{equation}
	where $M(u)$ and $P(u)$ are given by (\ref{cons}) and (\ref{momentum}).
\end{lemma}

\begin{proof}
	Since $M(U_c) = \int_{\R} (1-\rho_c)^2 d\xi$, the expansion (\ref{expansion-M-c}) is justified due to the bound (\ref{eta-c-bound}) 
	on $\rho_c = \varphi^2 + \eta_c$ and the explicit computation 
	$$
	\int_{\R} (1-\varphi^2)^2 d\xi = \frac{4}{3}.
	$$
	
	For the expansion (\ref{expansion-P-c}), we use the representation (\ref{momentum-new}), 
	the polar form $U_c = \sqrt{\rho}_c e^{i \phi_c}$ with $\phi_c'$ given by (\ref{phi}), and the boundary conditions (\ref{bc-dark}). This yields
	\begin{align*}
P(U_c) &= \theta_+(c) - \theta_-(c) + \frac{c}{2} \int_{\R} (2 - \rho_c) (1-\rho_c)^2 d\xi \\
&= -\pi + 2 \arcsin \beta_c + \frac{c}{2} \int_{\R} (2 - \rho_c) (1-\rho_c)^2 d\xi,
\end{align*}
where we have used Remark \ref{remark-phases} in the second equality.
We recall that $\rho_c = \varphi^2 + \mathcal{O}_{H^2}(c^2)$ and 
$U_c = \varphi - 2 ic v_{\varphi} + \mathcal{O}_{\mathcal{H}^2_-}(c^2)$ with $v_{\varphi}$ given by (\ref{eigen-gener-exact}) so that $\beta_c = c + \mathcal{O}(c^3)$. Hence, $P(U_c) = -\pi + c P_1 + \mathcal{O}(c^3)$ with 
\begin{align*}
P_1 = 2 + \frac{1}{2} \int_{\R} (2 - \varphi^2) (1-\varphi^2)^2 d\xi = \frac{16}{5},
\end{align*}
which justifies the assertion.
\end{proof}

Theorem \ref{theorem-dark} is proven with Lemma \ref{lem-dark} and Theorem \ref{lem-conv-to-0}. Lemma \ref{lem-momentum} is used for the stability 
analysis of the black solitons in the proof of Theorem \ref{theorem-nonlinear}.

\section{Energetic stability of the black soliton}
\label{sec-stability}

Here we study the energetic stability of the black soliton in $\Sigma \cap \mathcal{H}$ from the Lyapunov 
functional constructed from conserved quantities of the NLS model 
(\ref{nls-idd}). This yields the proof of Theorem \ref{theorem-nonlinear}. The energetic stability is equivalent 
to orbital stability provided that the initial-value problem is locally well-posed in $\Sigma \cap \mathcal{H}$ and the energy, mass, and momentum are conserved in the time evolution of the NLS model (\ref{nls-idd}).

We start with coercivitity of the Lyapunov functional $\Lambda$ defined by 
(\ref{Lyapunov}) at the black soliton with the profile $\varphi$. This follows from Theorem \ref{theorem-lin} and Lemma \ref{lem-quadruple}. 

\begin{lemma}
	\label{lem-coercivity}
	There exists $C > 0$ such that for every $\psi = \varphi + u + i v \in \Sigma \cap \mathcal{H}$ satisfying 
	\begin{equation}
	\label{constraints}
	(\varphi, u)_{\mathcal{H}} = 0, \quad 
	(\varphi',u)_{\mathcal{H}} = 0, \quad 
	(\varphi, v)_{\mathcal{H}} = 0, \quad 
	(\varphi',v)_{\mathcal{H}} = 0,
	\end{equation}
we have 
\begin{equation}
\label{coercivity}
\Lambda(\varphi + u + iv) - \Lambda(\varphi) \geq C (\| u \|_{\mathcal{H}^1_-}^2 + 
\| v \|_{\mathcal{H}^1_-}^2 + \| \eta \|_{L^2}^2),
\end{equation}	
where $\eta$ is given by (\ref{eta-intro}).
\end{lemma}

\begin{proof}
	The expansion (\ref{expansion-Lambda}) can be rewritten equivalently as 
\begin{equation*}
\Lambda(\varphi + u + iv) - \Lambda(\varphi) = Q_-(u)
+ Q_-(v) + \| \eta \|_{L^2}^2.
\end{equation*}	
Hence it suffices to show coercivity of the quadratic form $Q_-(v)$ under the two constraints:
\begin{equation}
\label{coercivity-H}
Q_-(v) \geq C_0 \| v \|^2_{\mathcal{H}^1_-}, \quad 
	(\varphi, v)_{\mathcal{H}} = 0, \quad 
(\varphi',v)_{\mathcal{H}} = 0.
\end{equation}
Since the spectrum of $\mathcal{L}_-$ in $\mathcal{H}$ is purely discrete and 
$v$ satisfies the constraints  $(\varphi,v)_{\mathcal{H}} = 0$ and $(\varphi',v)_{\mathcal{H}} = 0$, it follows from the two items 
in the end of the proof of Theorem \ref{theorem-lin} that there exists $C_0 > 0$ such that 
\begin{equation}
\label{coercivity-L2}
Q_-(v)  \geq C_0 \| v \|^2_{\mathcal{H}}, \quad 
(\varphi,v)_{\mathcal{H}} = 0, \quad  (\varphi',v)_{\mathcal{H}} = 0.
\end{equation}
The coercivity bound (\ref{coercivity-H}) in $\mathcal{H}^1_-$ 
follows from the bound (\ref{coercivity-L2}) and the standard G{\aa}rding's inequality by adjusting the constant $C_0 > 0$.
\end{proof}

The four constraints (\ref{constraints}) are not generally 
preserved in the time evolution of the NLS model (\ref{nls-idd}). 
In order to ensure their preservation, we need to introduce four parameters in the family of solutions near the black soliton with the profile $\varphi$. 
Two parameters are given by translations along the symmetries (\ref{symmetry}). 
One parameter is the wave speed $c$ in the family of dark solitons $U_c$ given by (\ref{soliton-parameters}). One more parameter is the scaling parameter $\omega$ in the scaling transformation (\ref{scaling-transform}). 

The existence of $U_c$ for every $c \in \R$ is given by Theorem \ref{theorem-dark}. In the limit $c \to 0$, the bound (\ref{U-dark-expansion}) can be written as
\begin{equation}
\label{dark-soliton-exp}
U_c = \varphi - 2 i c v_{\varphi} + \mathcal{O}_{\mathcal{H}^2_-}(c^2) \quad \mbox{\rm as} \;\; c \to 0\,,
\end{equation}
where $\mathcal{O}_{\mathcal{H}^2_-}(c^2)$ is small in the $\mathcal{H}^2_-$ norm and $v_{\varphi}$ is defined in (\ref{eigen-gener-exact}).

To incorporate $c$ and $\omega$, we consider an extended Lyapunov functional in the form 
\begin{equation}
\label{action}
\Lambda_{c,\omega}(\psi) := E(\psi) + \omega^2 M(\psi) + c \omega P(\psi). 
\end{equation}
Since $E'(\psi) = -\psi''$, $M'(\psi) = -2 (1-|\psi|^2) \psi$, and 
$P'(\psi) = 2i (1-|\psi|^2 ) \psi'$ (see Lemma \ref{lem-expansion} below), 
the Euler--Lagrange equation for $\Lambda_{c,\omega}$ is given by  
\begin{align}
\label{EL}
U'' + 2 \omega^2 (1-|U|^2) U - 2i c \omega (1-|U|^2) U' = 0.
\end{align}
If $U_c$ solves the differential equation (\ref{nls-trav}), 
then $U_{c,\omega}(x) := U_{c}(\omega x)$ solves the Euler--Lagrange equation (\ref{EL}) and hence $U_{c,\omega}$ is a critical point 
of the Lyapunov functional $\Lambda_{c,\omega}$ in (\ref{action}).

\begin{remark}
	\label{remark-norms}
Since $U_{c,\omega}$ extends $\varphi$ for $(c,\omega)$ near $(0,1)$, we can define a modified Hilbert space $\mathcal{H}_{c,\omega}$ with the modified inner product 
	\begin{equation*}
	(f,g)_{\mathcal{H}_{c,\omega}} := \int_{\R} (1-|U_{c,\omega}|^2) \bar{f} g dx, \quad f,g \in\mathcal{H}_{c,\omega} 
	\end{equation*}
	and the induced norm $\| \cdot \|_{\mathcal{H}_{c,\omega}}$. 
	Due to the proximity result (\ref{proximity}) in Remark \ref{remark-proximity} 
	and the closeness of $(c,\omega)$ to $(0,1)$, the norm $\| \cdot \|_{\mathcal{H}_{c,\omega}}$ is equivalent to the norm $\| \cdot \|_{\mathcal{H}}$.
\end{remark}

The following lemma gives the decomposition of a point in a local neighborhood 
of the black soliton $\varphi$ with four modulation parameters defined 
near $U_{c,\omega}$.

\begin{lemma}
	\label{lem-parameters}
There exists $\epsilon_0 > 0$ and $C_0 > 0$ such that 
for every $\psi \in \Sigma \cap \mathcal{H}$ satisfying  
	\begin{equation}
\label{eps-initial}
\epsilon := \inf_{\theta, \zeta  \in\R} \| \psi - e^{i \theta} \varphi(\cdot + \zeta) \|_{\mathcal{H}^1_-} \leq \epsilon_0,
	\end{equation}
there exists unique $\theta, \zeta, c, \omega \in \mathbb{R}$ such that 
	\begin{equation}
	\label{decomposition}
\psi = e^{i \theta} \left[ U_{c,\omega}(\cdot + \zeta) + u(\cdot+ \zeta) + i v(\cdot + \zeta) \right], 
	\end{equation}
	and
	\begin{equation}
\label{eps-final}
|c| + |\omega - 1| + \| \psi - e^{i \theta} U_{c,\omega}(\cdot + \zeta) \|_{\mathcal{H}^1_-}  \leq C_0 \epsilon, 
\end{equation}		
where $u,v \in \mathcal{H}^1_-$ satisfy the four orthogonality conditions 
	\begin{equation}
\label{constraints-with-c}
(U_{c,\omega}, u)_{\mathcal{H}_{c,\omega}} = 0, \quad
(U_{c,\omega}', u)_{\mathcal{H}_{c,\omega}}  = 0, \quad 
(U_{c,\omega}, v)_{\mathcal{H}_{c,\omega}}  = 0, \quad 
(U_{c,\omega}', v)_{\mathcal{H}_{c,\omega}}  = 0,
\end{equation}
whereas $U_{c,\omega} := U_{c}(\omega \cdot)$ is the dark soliton of Theorem \ref{theorem-dark}.
\end{lemma}

\begin{proof}
	Let us define the vector function $\vec{F}(c,\omega,\theta,\zeta;\psi) : \mathbb{R}^4 \times \Sigma \cap \mathcal{H} \mapsto \mathbb{C}^4$ representing the constraints (\ref{constraints-with-c}):
	$$
\vec{F}(c,\omega,\theta,\zeta;\psi) := \left[ \begin{array}{l}
(U_{c,\omega}, {\rm Re} (e^{-i\theta} \psi(\cdot - \zeta) - U_{c,\omega}))_{\mathcal{H}_{c,\omega}}  \\
(U_{c,\omega}', {\rm Re} (e^{-i\theta} \psi(\cdot - \zeta) - U_{c,\omega}))_{\mathcal{H}_{c,\omega}}   \\ 
(U_{c,\omega}, {\rm Im} (e^{-i\theta} \psi(\cdot - \zeta) - U_{c,\omega}))_{\mathcal{H}_{c,\omega}}   \\
(U_{c,\omega}', {\rm Im} (e^{-i\theta} \psi(\cdot - \zeta) - U_{c,\omega}))_{\mathcal{H}_{c,\omega}}  
\end{array}	\right].
$$
For a given $\psi \in \Sigma \cap \mathcal{H}$, $\theta$ in the infimum  (\ref{eps-initial}) is defined on the compact interval $[0,2\pi]$ due to periodicity of $e^{i\theta}$ and there exists $C_{\infty} > 0$ such that 
$$
\lim_{\zeta \to \infty} \| \psi - e^{i \theta} \varphi(\cdot + \zeta) \|_{\mathcal{H}^1_-} \geq C_{\infty}.
$$
Therefore, the infimum in (\ref{eps-initial}) is attained if  
$\epsilon_0 \in (0,C_{\infty})$ is properly chosen. Let $(\theta_0,\zeta_0)$ be arguments of the infimum in (\ref{eps-initial}). Since $\varphi, \varphi' \in \mathcal{H}$, the Cauchy--Schwarz inequality implies that there exists $C > 0$ such that 
$$
| \vec{F}(0,1,\theta_0,\zeta_0;\psi)| \leq C \epsilon, 
$$
where $|\cdot|$ 
denotes the standard Euclidean norm for vectors in $\mathbb{C}^4$. 
We can write 
$$
\vec{F}(0,1,\theta_0,\zeta_0;\psi) = \mathcal{O}(\epsilon)
$$ 
to indicate the remainder term for $\psi = e^{i \theta_0} \varphi(\cdot + \zeta_0) + \mathcal{O}_{\mathcal{H}^1_-}(\epsilon)$.

Recall that the mapping $(c,\omega) \mapsto U_{c,\omega}$ is $C^1$ near $(c,\omega) = (0,1)$. Hence the function $\vec{F}$ is $C^1$ with respect to its arguments and we can compute the Jacobian of $\vec{F}$ with respect to $(c,\omega,\theta,\zeta)$ at $(0,1,\theta_0,\zeta_0)$ and for fixed $\psi \in \Sigma \cap \mathcal{H}$ satisfying (\ref{eps-initial}). By using (\ref{dark-soliton-exp}) and $U_{c,\omega} := U_{c}(\omega \cdot)$, 
we compute 
\begin{align*}
(U_{c,\omega},(e^{-i\theta} \psi(\cdot - \zeta) - U_{c,\omega}))_{\mathcal{H}_{c,\omega}} &=  (\varphi(\omega \cdot),(e^{-i\theta} \psi(\cdot - \zeta) - \varphi(\omega \cdot)))_{\mathcal{H}_{c,\omega}} \\
& \quad + 2ic (v_{\varphi}(\omega \cdot),(e^{-i\theta} \psi(\cdot - \zeta) - \varphi(\omega \cdot)))_{\mathcal{H}_{c,\omega}} \\
& \quad + 2ic (\varphi(\omega \cdot),v_{\varphi}(\omega \cdot)))_{\mathcal{H}_{c,\omega}} + \mathcal{O}(c^2)
\end{align*}
and similarly for $(U_{c,\omega}',(e^{-i\theta} \psi(\cdot - \zeta) - U_{c,\omega}))_{\mathcal{H}_{c,\omega}}$. The Jacobian evaluated at $(c,\omega,\theta,\zeta) = (0,1,\theta_0,\zeta_0)$ is computed 
by using the proximity bound (\ref{eps-initial}) and is given 
by the following matrix
$$
-\left[ \begin{array}{cccc} 0 & (\varphi, x \varphi' )_{\mathcal{H}}
& 0 & 0  \\ 0 & 0 & 0 & (\varphi', \varphi' )_{\mathcal{H}} \\ 
0 & 0 & (\varphi, \varphi )_{\mathcal{H}} & 0   \\
-2 (\varphi', v_{\varphi} )_{\mathcal{H}} & 0 & 0 & 0 
\end{array}	\right] + \mathcal{O}(\epsilon)
$$
where we have used that $\varphi$ is odd and $v_{\varphi}$ is even. 
We recall from the proof of Lemma \ref{lem-quadruple} that 
$ (\varphi', v_{\varphi} )_{\mathcal{H}} \neq 0$ and $(\varphi, x \varphi' )_{\mathcal{H}} \neq 0$, where $u_{\varphi}(x) = -\frac{1}{4} x \varphi'(x)$. 
By choosing $\epsilon_0$ small enough in (\ref{eps-initial}), 
the Jacobian for $\vec{F}$ is invertible. By the local inverse mapping theorem,  for any $\psi \in \Sigma \cap \mathcal{H}$ satisfying (\ref{eps-initial}) 
there exists a unique solution $(c,\omega,\theta,\zeta) \in \R^4$ of $\vec{F}(c,\omega,\theta,\zeta;\psi) = 0$ satisfying 
$$
|c| + |\omega - 1| + |\theta - \theta_0| + |\zeta - \zeta_0| \leq C \epsilon,
$$
for some $\epsilon$-independent $C > 0$. Thus, the decomposition 
(\ref{decomposition}) is justified. The last bound in (\ref{eps-final}) follows 
by the triangle inequality from the $C^1$ property of $U_{c,\omega}$ in $x$ and $(c,\omega)$.
\end{proof}

We will use the following expansion of the mass and momentum functionals.

\begin{lemma}
	\label{lem-expansion}
For every $c, \omega \in \mathbb{R}$ and $\psi \in \Sigma \cap \mathcal{H}$ satisfying 
(\ref{eps-initial}), (\ref{decomposition}), (\ref{eps-final}), and (\ref{constraints-with-c}) with $\epsilon \in (0,\epsilon_0)$, where $\epsilon_0 > 0$ is defined in Lemma \ref{lem-parameters}, we have the expansions 
\begin{equation}
\label{expansion-M}
M(U_{c,\omega} + u + iv) = M(U_{c,\omega}) - 2 \| u  \|^2_{\mathcal{H}_{c,\omega}} - 2 \| v \|^2_{\mathcal{H}_{c,\omega}} + \| \eta \|^2_{L^2}
\end{equation}
and 
\begin{equation}
\label{expansion-P}
P(U_{c,\omega} + u + iv) = P(U_{c,\omega}) + \hat{P}(u,v), 
\end{equation}
where 
$$
\eta := | U_{c,\omega} + u + iv|^2 - |U_{c,\omega}|^2 = 2 u {\rm Re}(U_{c,\omega}) + 2 v {\rm Im}(U_{c,\omega}) + u^2 + v^2
$$ 
and there is $C > 0$ such that 
\begin{equation}
\label{bound-on-R}
|\hat{P}(u,v)| \leq C \left( \| u' \|^2_{L^2} + \|v'\|^2_{L^2} + \| u \|^2_{\mathcal{H}} + \| v \|^2_{\mathcal{H}} + \| \eta \|^2_{L^2} \right).
\end{equation}
\end{lemma}

\begin{proof}
	Since $\eta = |U_{c,\omega} + u + iv|^2 - |U_{c,\omega}|^2$, we have $\eta \in L^2(\R)$ if $\psi, U_{c,\omega} \in \Sigma$. The expansion (\ref{expansion-M}) follows from direct computations: 
	\begin{align*}
	M(U_{c,\omega} + u + iv) &= M(U_{c,\omega}) - 2 \int_{\R} (1 - |U_{c,\omega}|^2) \eta dx + \| \eta \|^2_{L^2} \\
	&= M(U_{c,\omega}) - 2 \int_{\R} (1 - |U_{c,\omega}|^2) (u^2+v^2) dx + \| \eta \|^2_{L^2},
	\end{align*}
where the second line is obtained due to the first and third orthogonality conditions in (\ref{constraints-with-c}). 
The expansion (\ref{expansion-P}) is proven in Appendix \ref{app-b}, where the second and fourth orthogonality conditions in (\ref{constraints-with-c}) 
are used to remove the linear term of the expansion.
\end{proof}

The following lemma generalizes the coercivity result of Lemma \ref{lem-coercivity} by using the four-parameter decomposition of 
Lemma \ref{lem-parameters}.

\begin{lemma}
	\label{lem-coercivity-c}
For every $c, \omega \in \mathbb{R}$ and $\psi \in \Sigma \cap \mathcal{H}$ satisfying 
(\ref{eps-initial}), (\ref{decomposition}), (\ref{eps-final}), and (\ref{constraints-with-c}) with $\epsilon \in (0,\epsilon_0)$, where $\epsilon_0 > 0$ is defined in Lemma \ref{lem-parameters}, there is $C > 0$ such that
\begin{equation}
	\label{coercivity-with-c}
\hat{\Lambda}_{c,\omega}(u,v) \geq C \left( \| u \|_{\mathcal{H}^1_-}^2 + 
	\| v \|_{\mathcal{H}^1_-}^2 + \| \eta \|_{L^2}^2 - c^2 - (\omega-1)^2 \right),
	\end{equation}	
where 
\begin{align*}
\hat{\Lambda}_{c,\omega}(u,v) & := E(U_{c,\omega} + u + iv) - E(\varphi) + \omega^2 \left[ M(U_{c,\omega} + u + iv) - M(\varphi) \right] \\
& \qquad + c \omega \left[ P(U_{c,\omega} + u + iv) + \pi \right]
\end{align*}
and
$\eta := |U_{c,\omega} + u + iv|^2 - |U_{c,\omega}|^2$.
\end{lemma}

\begin{proof}
Since $U_{c,\omega}$ is a critical point for $\Lambda_{\omega,c}$ in $\Sigma$, 
we obtain with the help of the expansions (\ref{expansion-M}) and (\ref{expansion-P}) in Lemma \ref{lem-expansion} that 
\begin{align*}
\hat{\Lambda}_{c,\omega}(u,v) = \Delta(c,\omega) + \| u' \|^2_{L^2} 
+ \| v'\|^2_{L^2} + \omega^2 \left( - 2 \| u  \|^2_{\mathcal{H}_{c,\omega}} - 2 \| v \|^2_{\mathcal{H}_{c,\omega}} + \| \eta \|^2_{L^2} \right) + c \omega \hat{P}(u,v),
\end{align*}
where
\begin{align*}
\Delta(c,\omega) & := \hat{\Lambda}_{c,\omega}(0,0) \\
& = E(U_{c,\omega}) - E(\varphi) + \omega^2 \left[ M(U_{c,\omega}) - M(\varphi) \right] + c \omega \left[ P(U_{c,\omega}) + \pi \right].
\end{align*}
Since 
\begin{align*}
\frac{\partial \Delta}{\partial c} = \omega \left[ P(U_{c,\omega}) + \pi \right], \qquad 
\frac{\partial \Delta}{\partial \omega} = 2 \omega \left[ M(U_{c,\omega}) - M(\varphi)  \right] + c \left[ P(U_{c,\omega}) + \pi \right], 
\end{align*}
it follows from (\ref{expansion-M-c}) and (\ref{expansion-P-c}) 
that $(c,\omega) = (0,1)$ is a critical point of $\Delta(c,\omega)$. 
Since $\Delta(0,1) = 0$, there is $A > 0$ such that 
$$
\Delta(c,\omega) \geq - A \left[ c^2 + (\omega - 1)^2 \right],
$$
for every $(c,\omega)$ near $(0,1)$. Due to the bound (\ref{bound-on-R}) and Young's inequality, there is $A > 0$ such that  
$$
c  \omega \hat{P}(u,v) \geq - A \left[ c^2 + \left( \| u \|^2_{\mathcal{H}^1_-} 
+ \| v \|^2_{\mathcal{H}^1_-} + \| \eta \|^2_{L^2} \right)^2 \right]
$$
for every $(c,\omega)$ near $(0,1)$. A similar lower bound is obtained for 
$$
(\omega^2 - 1)  (- 2 \| u  \|^2_{\mathcal{H}_{c,\omega}} - 2 \| v \|^2_{\mathcal{H}_{c,\omega}} + \| \eta \|^2_{L^2}) 
\geq - A \left[ (\omega - 1)^2 + c^4 + \left( \| u \|^2_{\mathcal{H}} 
+ \| v \|^2_{\mathcal{H}} + \| \eta \|^2_{L^2} \right)^2 \right],
$$
due to the proximity of norms in Remark \ref{remark-proximity}.
The remaining terms in $\hat{\Lambda}_{c,\omega}(u,v)$ are 
$$
\| u' \|^2_{L^2} 
+ \| v'\|^2_{L^2} - 2 \| u  \|^2_{\mathcal{H}_{c,\omega}} - 2 \| v \|^2_{\mathcal{H}_{c,\omega}} + \| \eta \|^2_{L^2} = Q_{c,\omega}(u) + Q_{c,\omega}(v) + \| \eta \|^2_{L^2}, 
$$
where $Q_{c,\omega}(u) := \| u' \|^2_{L^2} - 2 \| u \|^2_{\mathcal{H}_{c,\omega}}$. The coercivity $Q_{c,\omega}(u)$ under the two constraints is obtained similarly to the proof of Lemma \ref{lem-coercivity}:
\begin{equation}
Q_{c,\omega}(u) \geq C \left( \| u' \|^2_{L^2} + \| u \|^2_{\mathcal{H}_{c,\omega}} \right), \quad 
(U_{c,\omega},u)_{\mathcal{H}_{c,\omega}} = 0, \quad 
(U_{c,\omega}',u)_{\mathcal{H}_{c,\omega}} = 0.
\end{equation}
Combining all lower bounds and using the proximity of norms in 
Remark \ref{remark-norms}  yields 
\begin{align*}
\hat{\Lambda}_{c,\omega}(u,v) & \geq 
C \left( \| u \|_{\mathcal{H}^1_-}^2 + 
\| v \|_{\mathcal{H}^1_-}^2 + \| \eta \|_{L^2}^2 \right) \\
& - A \left[ c^2 + (\omega-1)^2 + \left( \| u \|^2_{\mathcal{H}^1_-} 
+ \| v \|^2_{\mathcal{H}^1_-} + \| \eta \|^2_{L^2} \right)^2 \right].
\end{align*}
This justifies the bound (\ref{coercivity-with-c}) due to smallness 
of $\| u \|^2_{\mathcal{H}^1_-} 
+ \| v \|^2_{\mathcal{H}^1_-} + \| \eta \|^2_{L^2}$ by adjusting the constant $C > 0$.
\end{proof}

We are now in position to complete the proof of Theorem \ref{theorem-nonlinear}.

Let $\psi_0 \in \Sigma \cap \mathcal{H}$ satisfy
\begin{equation}
\label{delta}
\mathcal{D}_{\Sigma \cap \mathcal{H}}(\psi_0,\varphi) < \delta
\end{equation} 
for some small $\delta \in (0,\epsilon_0)$, where $\epsilon_0$ is defined in Lemma \ref{lem-parameters}. 
Since we are assuming that the initial-value problem for the NLS model (\ref{nls-idd}) 
is locally well-posed in $\Sigma \cap \mathcal{H}$, there exists a 
unique solution $\psi \in C^0([-\tau_0,\tau_0],\Sigma \cap \mathcal{H})$ 
of the NLS model (\ref{nls-idd}) for some small $\tau_0> 0$ such that 
$\psi(0,\cdot) = \psi_0$. Therefore, at least for small $\tau_0 > 0$, 
the bound (\ref{eps-initial}) holds true so that 
the orthogonal decomposition (\ref{decomposition}) and (\ref{constraints-with-c}) can be used for $t \in [-\tau_0,\tau_0]$.
 
Assuming conservation of energy $E(\psi)$, mass $M(\psi)$, and momentum $P(\psi)$, there is $C > 0$ such that 
\begin{align}
\label{bound-delta}
\left\{ \begin{array}{l} 
|E(\psi) - E(\varphi)| = |E(\psi_0) - E(\varphi)| \leq C \delta, \\
|M(\psi) - M(\varphi)| = |M(\psi_0) - M(\varphi)| \leq C \delta, \\
|P(\psi) + \pi| = |P(\psi_0) + \pi| \leq C \delta,
\end{array}
\right.
\end{align}
where the upper bounds are due to (\ref{delta}) and the expansion of 
$E(\psi_0)$, $M(\psi_0)$, and $P(\psi_0)$ near $\varphi$. Compared to Lemma \ref{lem-expansion}, 
no orthogonality conditions are imposed on the perturbation term $\psi_0 - \varphi$ so that the linear terms of the expansion of the $\mathcal{O}(\delta)$ order are generally nonzero.

By using the orthogonal decomposition (\ref{decomposition}) and (\ref{constraints-with-c}) for $t \in [-\tau_0,\tau_0]$ with small $\tau_0> 0$, 
it follows from the symmetry (\ref{symmetry}) that 
\begin{align}
\label{expansion-delta}
\begin{array}{l} M(\psi_0) - M(\varphi) = M(U_{c,\omega} + u + iv) - M(\varphi), \\
P(\psi_0) + \pi = P(U_{c,\omega} + u + iv) + \pi, \end{array}
\end{align}
where $(c,\omega)$ depend on time $t \in [-\tau_0,\tau_0]$. Using now expansions (\ref{expansion-M}) and (\ref{expansion-P}), we claim that there is $C > 0$ 
such that 
\begin{align}
\label{c-omega}
|c| + |\omega - 1| \leq C (\delta + \| u \|^2_{\mathcal{H}^1_-} + 
 \| v \|^2_{\mathcal{H}^1_-} + \| \eta \|^2_{L^2})	
\end{align}
This follows from the implicit function theorem if the Jacobian of the transformation 
\begin{equation}
\label{mapping}
(c,\omega) \mapsto (M(U_{c,\omega}) - M(\varphi), P(U_{c,\omega}) + \pi)
\end{equation}
is invertible at $(c,\omega) = (0,1)$. Due to the scaling transformation  
with $U_{c,\omega} := U_c(\omega \cdot)$, we have 
$$
M(U_{c,\omega}) = \omega^{-1} \hat{M}(c), \qquad 
P(U_{c,\omega}) = \hat{P}(c),
$$
where $\hat{M}(c) := M(U_{c})$ and $\hat{P}(c) := P(U_{c})$, so that the determinant of the Jacobian is 
$$
\left| \begin{array}{cc} 
\omega^{-1} \hat{M}'(c) & - \omega^{-2} \hat{M}(c) \\
\hat{P}'(c) & 0
\end{array} \right| = \omega^{-2} \hat{M}(c\omega) \hat{P}'(c\omega).
$$
By using expansions (\ref{expansion-M-c}) and (\ref{expansion-P-c}) in Lemma \ref{lem-momentum}, the determinant at $(c,\omega) = (0,1)$ 
is equal to $\frac{64}{15} \neq 0$ so that 
the transformation is invertible. Hence, expansions (\ref{expansion-M}) and (\ref{expansion-P}) with the bounds (\ref{bound-on-R}) and (\ref{bound-delta}) yield (\ref{c-omega}) from (\ref{expansion-delta}).

Finally, we substitute (\ref{c-omega}) into (\ref{coercivity-with-c})  
and use conservation of $E(\psi)$, $M(\psi)$, and $P(\psi)$ with the bound (\ref{bound-delta}). This yields the bound 
\begin{equation}
\label{u-v}
\| u \|_{\mathcal{H}^1_-}^2 + 
\| v \|_{\mathcal{H}^1_-}^2 + \| \eta \|_{L^2}^2 \leq C \delta,
\end{equation}
which shows that the perturbations of the orthogonal decomposition (\ref{decomposition}) and (\ref{constraints-with-c}) for $t \in [-\tau_0,\tau_0]$ with small $\tau_0> 0$ are controlled uniformly in time. 
Moreover, using the triangle inequality, the bounds (\ref{c-omega}) and (\ref{u-v}), and the expansion (\ref{dark-soliton-exp}), we obtain
\begin{align}
\mathcal{D}_{\Sigma \cap \mathcal{H}}\left(\psi,e^{i\theta}  \varphi(\cdot + \zeta)\right)
&\leq \mathcal{D}_{\Sigma \cap \mathcal{H}}\left(\psi,e^{i\theta} U_{c,\omega}(\cdot + \zeta)\right) + \mathcal{D}_{\Sigma \cap \mathcal{H}}\left(U_{c,\omega},\varphi\right)
\nonumber \\
&\leq C \left( \| u \|_{\mathcal{H}^1_-} + \| v \|_{\mathcal{H}^1_-} 
+ \| \eta \|_{L^2} + |c| + |\omega - 1| \right) 
\nonumber \\
&\leq C \delta^{1/2}.
\label{bound-eps}
\end{align}
Thus for every $\epsilon \in (0,\epsilon_0)$ there is $\delta < {\rm min}(\epsilon^2/C^2,\epsilon_0)$ such that 
the bound (\ref{eps-initial}) is satisfied for every $t \in [-\tau_0,\tau_0]$ 
so that the orthogonal decomposition (\ref{decomposition}) and (\ref{constraints-with-c}) can be extended beyond the times $t = \pm \tau_0$ 
with the same estimate (\ref{u-v}) and the same bound (\ref{bound-eps}). 
Hence, the local solution $\psi \in C^0([-\tau_0,\tau_0],\Sigma \cap \mathcal{H})$ near $e^{i \theta} \varphi(\cdot + \zeta)$ 
is extended globally in time with the bound (\ref{bound-eps}) for every $t \in \R$.

The proof of Theorem \ref{theorem-nonlinear} is complete.

\section{Persistence and stability of black solitons in potentials}
\label{sec-potential}

Here we consider a small and decaying potential in the framework of the perturbed NLS model (\ref{nls-idd-potential}) and prove persistence and stability of the black soliton. This yields the proof of Theorem \ref{theorem-potential}. The following two lemmas give the persistence and stability results separately.

The standing wave solution $\psi(t,x) = e^{-2it} \phi_{\varepsilon}(x)$ of the perturbed NLS model (\ref{nls-idd-potential}) are found from the second-order differential equation 
\begin{equation}
\label{ode-perturbed}
\phi_{\varepsilon}'' + 2 (1-|\phi_{\varepsilon}|^2) \phi_{\varepsilon} = 
\varepsilon V(x) \phi_{\varepsilon}.
\end{equation}
Since $V(x) : \R \mapsto \R$, we consider real solutions $\phi_{\varepsilon}(x) : \R \mapsto \R$ which converge to $\varphi$ pointwise in $x$ as $|\varepsilon| \to 0$, where $\varphi(x) := \tanh(x)$. The following lemma uses the Lyapunov--Schmidt reduction method to get the existence result.

\begin{lemma}
	\label{lem-potential-existence}
		Assume that $V \in W^{2,\infty}(\R) \cap L^2(\R)$ and that $s \in \mathbb{R}$ is a simple root of $\mathcal{V}'$, where $\mathcal{V}$ is given by (\ref{potential-effective}). There exists $\varepsilon_0 > 0$ such that for every $\varepsilon \in (-\varepsilon_0,\varepsilon_0)$, there exists a real solution of equation (\ref{ode-perturbed}) in the form $\phi_{\varepsilon} = \varphi_s + \tilde{\varphi}_{\varepsilon}$,
	where $\varphi_s(x) := \varphi(x-s)$ and $\tilde{\varphi}_{\varepsilon} \in H^2(\R)$ satisfies $\| \tilde{\varphi}_{\varepsilon} \|_{H^2} \leq C |\varepsilon|$ for some $\varepsilon$-independent positive constant $C$. 
\end{lemma}

\begin{proof}
We use the decomposition 
		\begin{equation}
	\label{decomp-vaphi}
\phi_{\varepsilon} = \varphi_{s+a} + \psi,
	\end{equation}	
where $s \in \R$ is a simple root of $\mathcal{V}'(s)$, $a \in \R$ is a parameter to be determined, and $\psi \in H^2(\R)$ is a correction term to be determined. The decomposition (\ref{decomp-vaphi}) allows us to rewrite (\ref{ode-perturbed}) in the equivalent form 
	\begin{equation}
	\label{ode-persistence}
L_+ \psi + \varepsilon V \left[ \varphi_{s+a} + \psi \right] + N(\psi) = 0,
	\end{equation}
where $L_+ := -\partial_x^2 + 6 \varphi^2_{s+a} - 2$ is the translated version of the linearized operator $L_+$ in (\ref{L-plus}) and $N(\psi) := 6 \varphi_{s+a} \psi^2 + 2 \psi^3$ is the nonlinear part of the cubic nonlinearity. 
	
The linearized operator $L_+$ can be considered in $L^2(\R)$, where $0$ is a simple isolated eigenvalue of $L_+$ at the bottom of $\sigma(L_+)$ with ${\rm Null}(L_+) = {\rm span}(\varphi'_{s+a})$. The Lyapunov--Schmidt reduction method relies on the orthogonal decomposition by using the orthogonal projection
$$
\Pi_a \psi := \psi - \frac{\langle \varphi'_{s+a}, \psi \rangle}{\| \varphi'_{s+a} \|^2_{L^2}} \varphi'_{s+a},
$$
where $\langle \cdot, \cdot \rangle$ is the inner product in $L^2(\R)$. 
Assuming $\Pi_a \psi = \psi$ (that is, $\langle \varphi'_{s+a}, \psi \rangle = 0$) for uniqueness of definitions of $a \in \R$ and $\psi \in H^2(\R)$ allows us to split (\ref{ode-persistence}) into two parts:
	\begin{align}
	\label{implicit-F}
	F(\varepsilon,a,\psi) & := L_+ \psi + \Pi_a \left( \varepsilon V\left[ \varphi_{s+a} + \psi \right] + Z(\psi) \right) = 0,	\\
	\label{implicit-f}
	f(\varepsilon,a,\psi) & := \langle \varphi'_{s+a}, \varepsilon V\left[ \varphi_{s+a} + \psi \right] + Z(\psi) \rangle = 0.
	\end{align}
Since 
\begin{itemize}
	\item $F(\varepsilon, a,\psi) : \R \times \R \times H^2(\R)|_{\{\varphi'_{s+a}\}^{\perp}} \to L^2(\R) |_{\{\varphi'_{s+a}\}^{\perp}}$ is $C^{\infty}$ for every $V \in L^2(\R)$,\\
	\item $F(0,0,0) = 0$, \\
	\item $D_{\psi} F(0,0,0) = L_+ : H^2(\R)|_{\{\varphi'_{s+a}\}^{\perp}}  \mapsto L^2(\R)|_{\{\varphi'_{s+a}\}^{\perp}} $ is an invertible operator with a bounded inverse from $L^2(\R)|_{\{\varphi'_{s+a}\}^{\perp}}$ to $H^2(\R)|_{\{\varphi'_{s+a}\}^{\perp}}$, 
\end{itemize}
the implicit function theorem gives the existence of a unique $C^{\infty}$ mapping $(\varepsilon,a) \mapsto \psi_{\varepsilon,a} \in H^2(\R) |_{\{\varphi'_{s+a}\}^{\perp}}$ which solves 
(\ref{implicit-F}) for every small $(\varepsilon,a) \in \R \times \R$ with the estimate 
\begin{equation}
	\label{estimate-1}
\| \psi_{\varepsilon,a} \|_{H^2} \leq C |\varepsilon| (1 + |a|),
\end{equation}
for some $(\varepsilon,a)$-independent constant $C > 0$.

The first term of $f(\varepsilon,a,\psi)$ can be simplified after integration by parts:
\begin{align*}
\int_{\R} V(x) \varphi(x-s-a) \varphi'(x - s - a) dx 
&= \frac{1}{2} \int_{\R} V'(x) \left[ 1 - \varphi^2(x-s-a) \right] dx \\
&= \frac{1}{2} \int_{\R} V'(x) {\rm sech}^2(x-s-a) dx \\
&= \frac{1}{2} \mathcal{V}'(s+a).
\end{align*}
Since $V \in W^{2,\infty}(\R)$, $\mathcal{V}$ is a $C^2$ function of its argument. If $s \in \R$ is a simple root of $\mathcal{V}'$, then 
$$
\lim_{a \to 0} \frac{\mathcal{V}'(s+a) - \mathcal{V}''(s) a}{a} = 0.
$$
Substituting the $C^{\infty}$ mapping 
$(\varepsilon,a) \mapsto \psi_{\varepsilon,a} \in H^2(\R) |_{\{\varphi'_{s+a}\}^{\perp}}$ 
from equation (\ref{implicit-F}) into (\ref{implicit-f}) 
gives the implicit equation 
\begin{equation}
\label{implicit-f-hat}
\hat{f}(\varepsilon,a) := \varepsilon^{-1} f(\varepsilon,a,\psi_{\varepsilon,a}) = 0,
\end{equation}
where $\hat{f}(\varepsilon,a) : \R\times \R\to \R$ is $C^1$ in its arguments, $\hat{f}(0,0) = 0$, and 
$$
\partial_a \hat{f}(0,0) = \frac{1}{2} \mathcal{V}''(s) \neq 0.
$$ 
The implicit function theorem gives the existence of a unique $C^1$ mapping $\varepsilon \mapsto a_{\varepsilon} \in \R$ which solves (\ref{implicit-f-hat}) for every small $\varepsilon \in \R$ with the estimate 
\begin{equation}
\label{estimate-2}
|a_{\varepsilon}| \leq C |\varepsilon|,
\end{equation}
for some $\varepsilon$-independent constant $C > 0$. Combining the two estimates (\ref{estimate-1}) and (\ref{estimate-2}) yields $\phi_{\varepsilon} = \varphi_s + \tilde{\varphi}_{\varepsilon}$ with $\| \tilde{\varphi}_{\varepsilon} \|_{H^2} \leq C |\varepsilon|$ after the triangle inequality.
\end{proof}

\begin{remark}
We cannot generally state that $\phi_{\varepsilon} \in \mathcal{F}$, where $\mathcal{F}$ is given by (\ref{function-set}).
\end{remark}

Using the black soliton $\phi_{\varepsilon}$ from Lemma \ref{lem-potential-existence}, 
we can define the spectral stability problem in the form 
\begin{equation}
\label{lin-stab-pot}
\begin{array}{c} L_-(\varepsilon) v = \lambda (1-\phi_{\varepsilon}^2) u, \\ -L_+(\varepsilon)  u = \lambda 
(1-\phi^2_{\varepsilon}) v,\end{array}
\end{equation}
where
\begin{align*}
L_+(\varepsilon) &= -\partial_x^2 + 6 \phi_{\varepsilon}^2 - 2 + \varepsilon V, \\
\label{L-minus-eps}
L_-(\varepsilon) &= -\partial_x^2 + 2 \phi_{\varepsilon}^2 - 2 + \varepsilon V.
\end{align*}
The Hilbert space (\ref{Hilbert-space}) is replaced by 
\begin{equation*}
\mathcal{H}_{\varepsilon} := \left\{ f \in L^2_{\rm loc}(\R) : \quad 
\sqrt{1 - \phi_{\varepsilon}^2} f \in L^2(\R) \right\}.
\end{equation*}
The following lemma ensures that under the additional condition $V \in L^1(\R)$, $L_{\pm}(\varepsilon)$ have the same properties 
as $L_{\pm}(0)$ except for one eigenvalue which bifurcates from the zero eigenvalue of $L_+(0)$ as $\varepsilon \to 0$. Consequently, the location of this eigenvalue gives a definite result on the spectral stability or instability within the linear stability problem (\ref{lin-stab-pot}). 

\begin{lemma}
	\label{lem-potential-stability}
Assume that $V \in W^{2,\infty}(\R) \cap L^1(\R)$ and that $s \in \mathbb{R}$ is a simple root of $\mathcal{V}'(s)$, where $\mathcal{V}(s)$ is given by (\ref{potential-effective}). There exists $\varepsilon_0 > 0$ such that for every $\varepsilon \in (0,\varepsilon_0)$
	\begin{itemize}
		\item If $\mathcal{V}''(s) > 0$, then the spectrum of the linear stability problem (\ref{lin-stab-pot}) in $\mathcal{H}_{\varepsilon}$ 
		consists of pairs of isolated eigenvalues (\ref{spectrum-stab}) 
		and a double zero eigenvalue. 
		\item If $\mathcal{V}''(s) < 0$, then the spectrum of the linear stability problem (\ref{lin-stab-pot}) in $\mathcal{H}_{\varepsilon}$ 
		consists of pairs of isolated eigenvalues (\ref{spectrum-stab}), 
		a double zero eigenvalue, and a pair of simple real eigenvalues $\{ \pm \lambda_0 \}$ with $\lambda_0 > 0$. 
	\end{itemize}
\end{lemma}

\begin{proof}
The profile $\phi_{\varepsilon}$ is obtained from the second-order differential equation (\ref{ode-perturbed}) with the boundary conditions $\phi_{\varepsilon}(x) \to \pm 1$ as $x \to \pm \infty$.
By Levinson's theorem (Theorem 8.1 on p.92 in \cite{CL}), 
if $V \in L^1(\R)$, then $\phi_{\varepsilon} \to \pm 1$ exponentially fast with the same exponential rate $2$ as $\varphi$. Consequently, the weighted norm in $\mathcal{H}_{\varepsilon}$ is equivalent to the one in $\mathcal{H}$ 
and the proof of Theorem \ref{th-discrete} extends to the stability problem (\ref{lin-stab-pot}). 
	
Since $L_-(\varepsilon) \phi_{\varepsilon} = 0$, the zero eigenvalue 
in $\mathcal{H}_{\varepsilon}$ persists in $\varepsilon \in (-\varepsilon_0,\varepsilon_0)$. Since $\varphi$ has a simple zero at $x = 0$, the decomposition (\ref{decomp-vaphi}) with small $\| \tilde{\varphi}_{\varepsilon} \|_{H^2}$ implies that $\phi_{\varepsilon}$ has only one zero on $\R$. Consequently, by Sturm's oscillation theorem, $L_-(\varepsilon)$ has only one simple negative eigenvalue for every $\varepsilon \in (-\varepsilon_0,\varepsilon_0)$. 

On the other hand, $L_+(\varepsilon)$ may not have the zero eigenvalue 
for $\varepsilon \neq 0$ because the translational symmetry is broken by the potential $V$. To study the sign of this small eigenvalue of $L_+(\varepsilon)$, we expand 
\begin{equation}
\label{expansion-1}
L_+(\varepsilon) = L_+(0) + \varepsilon V + 12 \varphi_{s} \tilde{\varphi}_{\varepsilon} + 6 \tilde{\varphi}_{\varepsilon}^2,
\end{equation}
where $L_+(0) := -\partial_x^2 + 6 \varphi^2_{s} - 2$. Since $\| \tilde{\varphi}_{\varepsilon} \|_{H^2} = \mathcal{O}(\varepsilon)$ as $\varepsilon \to 0$, we can write 
\begin{equation}
\label{expansion-2}
\tilde{\varphi}_{\varepsilon} = \varepsilon \tilde{\varphi}^{(1)} + \mathcal{O}_{H^2}(\varepsilon^2),
\end{equation}
where $\tilde{\varphi}^{(1)}$ is uniquely determined by 
Lemma \ref{lem-potential-existence}. 
Perturbation theory for isolated eigenvalues of a self-adjoint operator gives the small eigenvalue of $L_+(\varepsilon)$ in the form 
\begin{equation}
\label{mu-expansion}
\mu(\varepsilon) = \frac{1}{\| \varphi'_{s} \|^2_{L^2}} \left[ 
\varepsilon \langle \varphi'_{s}, (V + 12 \varphi_{s} \tilde{\varphi}^{(1)})  \varphi'_{s} \rangle + \mathcal{O}(\varepsilon^2) \right].
\end{equation}
Derivative of equation (\ref{ode-perturbed}) in $x$ yields 
$$
L_+(\varepsilon) \phi_{\varepsilon}' = -\varepsilon V' \phi_{\varepsilon},
$$
from which we derive at the $\mathcal{O}(\varepsilon)$ order:
$$
L_+(0) (\tilde{\varphi}^{(1)})' + (V + 12 \varphi_{s} \tilde{\varphi}^{(1)}) \varphi'_{s} = -V' \varphi_{s}.
$$
Therefore, we obtain by integration by parts that 
\begin{align*}
\langle \varphi'_{s}, (V + 12 \varphi_{s} \tilde{\varphi}^{(1)})  \varphi'_{s} \rangle &= -\int_{\R} V'(x) \varphi(x-s) \varphi'(x - s) dx \\
&= -\frac{1}{2} \int_{\R} V''(x) \left[ 1 - \varphi^2(x-s) \right] dx \\
&= -\frac{1}{2} \int_{\R} V''(x) {\rm sech}^2(x-s) dx \\
&= -\frac{1}{2} \mathcal{V}''(s),
\end{align*}
where we have used the condition $V \in W^{2,\infty}(\R)$. 
Thus, it follows from (\ref{mu-expansion}) that 
$$
\mu(\varepsilon) = -\frac{1}{2 \| \varphi'_{s} \|^2_{L^2}} \left[ 
\varepsilon \mathcal{V}''(s) + \mathcal{O}(\varepsilon^2) \right]. 
$$
Since $\mathcal{V}''(s) \neq 0$ by assumption, we conclude 
that for $\varepsilon \in (0,\varepsilon_0)$ with sufficiently small 
$\varepsilon_0 > 0$, $L_+(\varepsilon)$ admits no negative eigenvalues if 
$\mathcal{V}''(s) < 0$ and a simple negative eigenvalue if $\mathcal{V}''(s) > 0$, since all other eigenvalues of $L_+$ in $\mathcal{H}$ are strictly positive.

It remains to compute splitting of the double zero eigenvalue 
for $\varepsilon \neq 0$ in the spectral stability problem (\ref{lin-stab-pot}) due to the broken translational symmetry. The double zero eigenvalue 
due to the rotational symmetry persists in $\varepsilon$, whereas 
all eigenvalues on $i \R$ also persist in $\varepsilon$ since they are associated with positive $Q_+(u)$ and $Q_-(v)$ \cite{CP10}. 
To compute the splitting of the double zero eigenvalue, we use the method of Puiseux expansions \cite{Walters} and expand the eigenvector $(u,v) \in \mathcal{H}_{\varepsilon} \times \mathcal{H}_{\varepsilon}$ and the eigenvalue $\lambda$ of the stability problem (\ref{lin-stab-pot}) for $\varepsilon \in (0,\varepsilon_0)$:
\begin{align*}
\left\{ \begin{array}{l}
u = \varphi'_{s} + \varepsilon u_1 + \mathcal{O}(\varepsilon^2), \\
v = \varepsilon^{1/2} v_1 + \mathcal{O}(\varepsilon^{3/2}), \\
\lambda = \varepsilon^{1/2} \lambda_1 + \mathcal{O}(\varepsilon^{3/2}).
\end{array} \right.
\end{align*}
As $\varepsilon = 0$, the Puiseux expansions recover the double isolated eigenvalue $\lambda = 0$ of the unperturbed stability problem (\ref{lin-stab}) with the eigenvector $(u,v) = (\varphi'_{s},0)^T$ and the generalized eigenvector $(u,v) = (0, v_{\varphi}(\cdot - s))^T$, where $v_{\varphi}$ is given by (\ref{eigen-gener-exact}). By using (\ref{decomp-vaphi}), (\ref{expansion-1}), and (\ref{expansion-2}), we obtain from (\ref{lin-stab-pot}) that 
$v_1 = \lambda_1 v_{\varphi}(\cdot - s)$ and $u_1 \in \mathcal{H}_{\varepsilon}$ is found from the linear inhomogeneous equation 
$$
L_+ u_1 + (V + 12 \varphi_{s} \tilde{\varphi}^{(1)}) \varphi'_{s} = -\lambda_1^2 (1 - \varphi_{s}^2) v_{\varphi}(\cdot - s).
$$
A solution $u_1 \in \mathcal{H}_{\varepsilon}$ exists if and only if the Fredholm condition is satisfied:
$$
\lambda_1^2 (\varphi',v_{\varphi})_{\mathcal{H}} + 
\langle \varphi'_{s}, (V + 12 \varphi_{s} \tilde{\varphi}^{(1)}) \varphi'_{s} \rangle = 0,
$$
which yields 
$$
\lambda_1^2 = \frac{1}{2 (\varphi',v_{\varphi})_{\mathcal{H}}} \mathcal{V}''(s).
$$
Since it follows from the proof of Lemma \ref{lem-quadruple} that $(\varphi',v_{\varphi})_{\mathcal{H}} < 0$,  we conclude 
that for $\varepsilon \in (0,\varepsilon_0)$ with sufficiently small 
$\varepsilon_0 > 0$, the spectral stability problem (\ref{lin-stab-pot}) admits a pair of real eigenvalues if $\mathcal{V}''(s) < 0$ and a pair of purely imaginary eigenvalues if $\mathcal{V}''(s) > 0$. In both cases, the pair of eigenvalues is given by given by the Puiseux expansion $\lambda = \varepsilon^{1/2} \lambda_1 + \mathcal{O}(\varepsilon^{3/2})$ with either $\lambda_1^2 > 0$ or $\lambda_1^2 < 0$ respectively. This complete the proof of lemma in view of the other eigenvalues of the spectral eigenvalue problem (\ref{lin-stab-pot}) on $i \R$.
\end{proof}

\begin{remark}
An alternative proof of instability in Lemma \ref{lem-potential-stability} can be developed based on the Rayleigh quotient 
\begin{equation}
\label{Rayleigh-2}
-\lambda_0^2 = \inf_{v \in \mathcal{H}_{\varepsilon} \backslash \{0\}} 
\frac{Q_-(v) + \varepsilon \int_{\R} V v^2 dx}{(\mathcal{L}_+(\varepsilon)^{-1} v, v)_{\mathcal{H}_{\varepsilon}}},
\end{equation}
where $Q_-(v)$ is exactly the same as in (\ref{quad-minus}) and $\mathcal{L}_+(\varepsilon) := (1-\phi_{\varepsilon}^2)^{-1} L_+(\varepsilon)$ is invertible on $\mathcal{H}_{\varepsilon}$ for every $\varepsilon \in (0,\varepsilon_0)$. No constraint on $v \in \mathcal{H}_{\varepsilon}$ is used in (\ref{Rayleigh-2}) compared to (\ref{Rayleigh}). If $\mathcal{V}''(s) < 0$, then $\mathcal{L}_+(\varepsilon)$ is strictly positive in $\mathcal{H}_{\varepsilon}$ so that the denominator in (\ref{Rayleigh-2}) is positive. However, the numerator in (\ref{Rayleigh-2}) attains negative values since $\mathcal{L}_-(\varepsilon) := (1-\phi_{\varepsilon}^2)^{-1} L_-(\varepsilon)$ has a simple negative eigenvalue in $\mathcal{H}_{\varepsilon}$ and $v \in \mathcal{H}_{\varepsilon}$ does not satisfy any orthogonality condition. Therefore, the smallest eigenvalue $-\lambda_0^2$ is negative so that there exists a pair of real eigenvalues $\{ \lambda_0,-\lambda_0\}$.
\end{remark}

\begin{remark}
	The method based on the Rayleigh quotient (\ref{Rayleigh-2}) is inconclusive if $\mathcal{V}''(s) > 0$ because both the numerator and the denominator attain negative values for some $v \in \mathcal{H}_{\varepsilon}$.
\end{remark}

\section{Conclusion}
\label{sec-conclusion}

This work can be summarized as follows. We have considered a new NLS model with intensity-dependent dispersion (\ref{nls-idd}) as an alternative of the cubic defocusing NLS equation (\ref{nls-cubic}). We have shown that the same black soliton has better stability properties in the NLS model (\ref{nls-idd}) in several aspects. The NLS model provides a natural Hilbert $L^2$ space with exponential weights, where the linearized stability problem admits only isolated eigenvalues of finite multiplicity and no finite accumulation points. The energetic stability argument is obtained with four rather than three orthogonality conditions, where the additional constraint is due to the new scaling symmetry (\ref{scaling-transform}). The black soliton is continued to the family of traveling dark solitons for every speed compared to the finite speed cutoff in the cubic NLS equation (\ref{nls-cubic}). 
Finally, the black soliton remains orbitally stable in the presence of a small decaying potential if it is pinned to the minimum point of the effective potential compared to either oscillatory or monotone instability of the black soliton of the cubic NLS equation with a small decaying potential. \\

We end this paper with a list of further open problems. \\

{\bf 1.} We have not attempted to address local well-posedness 
of the new NLS model (\ref{nls-idd}) in $\Sigma \cap \mathcal{H}$, 
where $\Sigma$ is the energy space (\ref{energy-space}) and $\mathcal{H}$ is the exponentially weighted $L^2(\R)$ space (\ref{Hilbert-space}). 
We conjecture that the initial-value problem is locally well-posed
in the set of functions $\mathcal{F}$ in (\ref{function-set}) for which $|u(t,x)| \leq 1$ for $(t,x) \in \R\times \R$, but this problem is a subject on its own. 
Recent results on local well-posedness of such NLS models 
with zero boundary conditions can be found in \cite{Marzuola}. \\

{\bf 2.} Another related problem is the asymptotic stability of 
black solitons. By using ideas of \cite{GS}, the asymptotic stability 
analysis can be developed provided that the local well-posedness problem 
for the NLS model (\ref{nls-idd}) is solved first. \\

{\bf 3.} Orbital stability of the continuous family of dark solitons 
with any speed $c \in \R$ can be considered by using the 
exponentially weighted Hilbert space $\mathcal{H}_{c,\omega = 1}$ in Remark \ref{remark-norms}. According to the standard orbital 
stability criterion \cite{Chiron,Lin}, the dark soliton with profile $U_c$ 
is orbitally stable if the mapping $c \mapsto P(U_c)$ is monotonically increasing. It follows from the expansion (\ref{expansion-P-c}) that the stability is satisfied for small $c \in (-c_0,c_0)$ with some $c_0 > 0$. 
However, since $U_c$ is only available implicitly, the precise stability 
criterion can only be checked numerically for larger values of $c \in \R$ 
or verified with computer-assisted proofs \cite{M-book}. \\

{\bf 4.} It would be interesting to justify the new NLS model (\ref{nls-idd}) 
as a reduction of the Maxwell equations in the framework of either temporal or spatial nonlinear optics. Justification of the extended NLS models of the type 
$$
i (1 - \partial_x^2) \psi_t + \psi_{xx} +  2 (1-|\psi|^2) \psi = 0
$$
was considered in \cite[Section 4.1.4]{Lannes} in the framework of the NLS models with a full dispersion relation. However, the relevant analysis did not incorporate the intensity-dependent dispersion. Justification of the NLS model (\ref{nls-idd}) is open for further studies. \\

{\bf 5.} One can think of the regularization of the NLS model (\ref{nls-idd}) 
with bounded intensity-dependent dispersion, similar to the one used in 
\cite{RKP} for the NLS model (\ref{nls-idd-bright}). Local well-posedness of the regularized model in the space of functions with nonzero boundary conditions at infinity can be shown by using the general analysis of \cite{Gallo}. It is expected that the stability properties 
of the black solitons analyzed in our work will persist in the time evolution of the regularized version of the NLS model (\ref{nls-idd}). 

\vspace{0.25cm}

{\bf Acknowledgement.} The work of D. E. Pelinovsky is partially supported by AvHumboldt Foundation as Humboldt Reseach Award. 
The work of M. Plum is supported by Deutsche Forschungsgemeinschaft (German Research Foundation) - Project-ID 258734477 - SFB 1173.

\appendix
\section{Integration by parts in $\mathcal{H}$}
\label{app-a}

We start with the following technical estimates. 

\begin{lemma}
	\label{lem-A}
	For every $w \in L^2(\R)$, it is true that 
\begin{eqnarray}
\label{tech-1}
&& \left\| \cosh(\cdot) \int_{\cdot}^{\infty} {\rm sech}(t) w(t) dt \right\|_{L^2(0,\infty)} \leq 2 \| w \|_{L^2(0,\infty)}, \\
\label{tech-2}
&& \left\| \cosh(\cdot) \int_{-\infty}^{\cdot} {\rm sech}(t) w(t) dt \right\|_{L^2(-\infty,0)} \leq 2 \| w \|_{L^2(-\infty,0)}, 
\end{eqnarray}
and	
\begin{eqnarray}
\label{tech-3}
\left\| {\rm sech}(\cdot) \int_{0}^{\cdot} \cosh(t) w(t) dt \right\|_{L^2(\R)} \leq 2 \| w \|_{L^2(\R)}.
\end{eqnarray}
\end{lemma}

\begin{proof}
	In order to prove (\ref{tech-1}), we integrate by parts and obtain for each $y > 0$:
	\begin{align*}
&	\int_0^y \cosh^2(x) \left| \int_{x}^{\infty} {\rm sech}(t) w(t) dt \right|^2 dx \\
\qquad &= \left[ \frac{1}{2} (x + \cosh(x) \sinh(x)) \left| \int_{x}^{\infty} {\rm sech}(t) w(t) dt \right|^2 \right] \biggr|_{x = 0}^{x = y} \\
\qquad 	&+\int_0^y (x {\rm sech}(x) + \sinh(x)) {\rm Re} \left( \bar{w}(x) \int_{x}^{\infty} {\rm sech}(t) w(t) dt \right) dx 
	\end{align*}
The first and second terms are estimated separately with Cauchy--Schwarz inequality as 
\begin{align*}
& \frac{1}{2} (y + \cosh(y) \sinh(y)) \left| \int_{y}^{\infty} {\rm sech}(t) w(t) dt \right|^2 \\
\qquad & \leq \frac{1}{2} (y + \cosh(y) \sinh(y)) (1-\tanh(y))  \int_{y}^{\infty} |w(t)|^2 dt  \\
\qquad & \leq \int_{y}^{\infty} |w(t)|^2 dt
\end{align*}
and 
\begin{align*}
& \int_0^y (x {\rm sech}(x) + \sinh(x)) |\bar{w}(x)| \left| \int_{x}^{\infty} {\rm sech}(t) w(t) dt \right| dx  \\
\qquad & \leq 2 \int_0^y \cosh(x) |\bar{w}(x)| \left| \int_{x}^{\infty} {\rm sech}(t) w(t) dt \right| dx  \\
\qquad & \leq 2 \left( \int_0^y \cosh^2(x) \left| \int_{x}^{\infty} {\rm sech}(t) w(t) dt \right|^2 dx \right)^{1/2} \left( \int_0^y |w(x)|^2 dx \right)^{1/2} \\
\qquad & \leq \frac{1}{2} \int_0^y \cosh^2(x) \left| \int_{x}^{\infty} {\rm sech}(t) w(t) dt \right|^2 dx + 2 \int_0^y |w(x)|^2 dx,  
\end{align*}
where we have used $2 a b \leq \frac{1}{2} a^2 + 2 b^2$ in the last line. 
Combining the estimates, we obtain 
\begin{align*}
\int_0^y \cosh^2(x) \left| \int_{x}^{\infty} {\rm sech}(t) w(t) dt \right|^2 dx \leq 2 \int_y^{\infty} |w(x)|^2 dx + 4 \int_0^y |w(x)|^2 dx,
\end{align*}
which yields the bound (\ref{tech-1}) in the limit $y \to \infty$. A similar bound (\ref{tech-2}) follows analogously for $y < 0$. 	

In order to prove (\ref{tech-3}), we again integrate by parts and obtain for each $y > 0$:
\begin{align*}
&	\int_0^y {\rm sech}^2(x) \left| \int_{0}^{x} \cosh(t) w(t) dt \right|^2 dx \\
\qquad &= \left[ (\tanh(x) - 1) \left| \int_{0}^{x} \cosh(t) w(t) dt \right|^2 \right] \biggr|_{x = 0}^{x = y} \\
\qquad 	& + 2 \int_0^y (\cosh(x) - \sinh(x)) {\rm Re} \left( \bar{w}(x) \int_{0}^{x} \cosh(t) w(t) dt \right) dx 
\end{align*}
The first term is negative while the second term is again estimated with Cauchy--Schwarz inequality as 
\begin{align*}
& 2 \int_0^y {\rm sech}(x)|w(x)| \left| \int_{0}^{x} \cosh(t) w(t) dt \right| dx   \\
\qquad & \leq 2 \left( \int_0^y {\rm sech}^2(x) \left| \int_{0}^{x} \cosh(t) w(t) dt \right|^2 dx \right)^{1/2} \left( \int_0^y |w(x)|^2 dx \right)^{1/2}.
\end{align*}
In the limit $y \to \infty$, this gives the bound 
\begin{align*}
\int_0^{\infty} {\rm sech}^2(x) \left| \int_{0}^{x} \cosh(t) w(t) dt \right|^2 dx 
\leq 4 \int_0^{\infty} |w(t)|^2 dt.
\end{align*}
A similar is obtained analogously for $y < 0$, which yields together the bound (\ref{tech-3}).
\end{proof}

Using the estimates in Lemma \ref{lem-A}, we prove the integration by parts formula for both function spaces. Here $\langle \cdot, \cdot \rangle$ denotes the standard inner product in $L^2(\R)$.

\begin{lemma}
	\label{lem-B} For every $f \in \mathcal{H}^2_-$ and $g \in \mathcal{H}^1_-$, it is true that $\langle f', g' \rangle = -\langle f'', g \rangle$.
\end{lemma}

\begin{proof}
	First, we check that if $f \in \mathcal{H}^2_-$, then $\cosh(\cdot) f' \in L^2(\R)$ and so $\sinh(\cdot) f' \in L^2(\R)$. Indeed, if $f \in \mathcal{H}^2_-$, then $f' \in L^2(\R)$ and $\cosh(\cdot) f'' \in L^2(\R)$. It is then clear that $f' \in H^1(\R)$ so that $f' \in C(\R)$ and $f'(x) \to 0$ as $|x| \to \infty$ by Sobolev's embedding. Therefore, we write 
	$$
	f'(x) = -\int_{x}^{\infty} f''(t) dt = - \int_x^{\infty} {\rm sech}(t) \left[ \cosh(t) f''(t)\right] dt, 
	$$ 
	so that $\cosh(\cdot) f' \in L^2(0,\infty)$ by the bound (\ref{tech-1}) with $w := \cosh(\cdot) f'' \in L^2(\R)$. Similarly, we obtain 
$\cosh(\cdot) f' \in L^2(-\infty,0)$ by the bound (\ref{tech-2}).
	
	Next, since 
\begin{equation}
\label{der-cosh}
	\frac{d}{dx} \left[ \cosh(x) f'(x) \right] = \cosh(x) f''(x) + \sinh(x) f'(x),
\end{equation}
we infer that if $f \in \mathcal{H}^2_-$, then $\cosh(\cdot) f' \in H^1(\R)$, which implies that $\cosh(x) f'(x) \to 0$ as $|x| \to \infty$. Similarly, if $g \in \mathcal{H}^1_-$, we have ${\rm sech}(\cdot) g \in H^1(\R)$ since ${\rm sech}(\cdot) g \in L^2(\R)$,
${\rm sech}(\cdot) g' \in L^2(\R)$, and 
\begin{equation}
\label{der-sech}
\frac{d}{dx} \left[ {\rm sech}(x) g(x) \right] = {\rm sech}(x) g'(x) - {\rm sech}(x) \tanh(x) g(x).
\end{equation}
This implies that ${\rm sech}(x) g(x) \to 0$ as $|x| \to \infty$  so that
$$
f'(x) g(x) = \left[ \cosh(x) f'(x) \right] \left[ {\rm sech}(x) g(x) \right] \to 0 \quad \mbox{\rm as} \;\; |x| \to \infty
$$
and integration by parts yields $\langle f', g' \rangle = -\langle f'', g \rangle$.
\end{proof}

\begin{lemma}
	\label{lem-C} For every $f \in \mathcal{H}^2_+$ and $g \in \mathcal{H}^1_+$, it is true that $\langle f', g' \rangle = -\langle f'', g \rangle$.
\end{lemma}

\begin{proof}
Here we recall that $\mathcal{H}^1_+ \equiv H^1(\R)$ and that if $f \in \mathcal{H}^2_+$, then $f \in H^1(\R)$ and $\cosh(\cdot) (-f''+ 4f) \in L^2(\R)$. This implies that $(-f'' + 4f) \in L^2(\R)$ and hence $f \in H^2(\R)$ 
so that $f(x), f'(x) \to 0$ as $|x| \to \infty$ and the integration by parts holds.
\end{proof}

\section{Proof of the expansion (\ref{expansion-P}) with the bound (\ref{bound-on-R})}
\label{app-b}

We follow the ideas in the proof of Propositions 4 and 5 in \cite{GS} and the proof of Lemma 4.2 in \cite{GPII} but we give a self-contained presentation. 

It is trivial to see that if $u \in \mathcal{H}^1_-$, then 
\begin{equation}
\label{Sob-emb}
\| \sqrt{1-\varphi^2} u \|_{L^{\infty}} \leq \| u \|_{\mathcal{H}^1_-}.
\end{equation}
Indeed, since $\|  \sqrt{1-\varphi^2} u' \|_{L^2} \leq \| u' \|_{L^2}$, 
the product rule (\ref{der-sech}) implies that 
$$
\| \sqrt{1-\varphi^2} u\|_{H^1} \leq \|u \|_{\mathcal{H}^1_-}
$$ 
which yields (\ref{Sob-emb}) by Sobolev's embedding of $H^1(\R)$ into $L^{\infty}(\R)$. On the other hand, it is rather nontrivial that the decomposition
\begin{equation}
\label{dec-app}
\psi = e^{i \theta} \left[ U_{c,\omega}(\cdot + \zeta) + u(\cdot + \zeta) + i v(\cdot + \zeta) \right],
\end{equation}
with 
\begin{equation}
\label{eps-app}
|c| + |\omega - 1| + \| u \|_{\mathcal{H}^1_-} + \| v \|_{\mathcal{H}^1_-} + \| \eta \|_{L^2} \leq C_0 \epsilon, 
\end{equation}
where 
\begin{equation}
\label{eta}
\eta := |\psi|^2 - |U_{c,\omega}|^2 = 2 u {\rm Re}(U_{c,\omega}) + 2 v {\rm Im}(U_{c,\omega}) + u^2 + v^2, 
\end{equation}
controls the $L^{\infty}$ norm of the perturbation. Nevertheless, this result is given in the following lemma used for the proof of the bound (\ref{bound-on-R}). In what follows, 
the constant $C > 0$ may change from one line to another line.

\begin{lemma}
	\label{lem-supremum}
	For every $c, \omega \in \mathbb{R}$ and $\psi \in \Sigma \cap \mathcal{H}$ satisfying 
	(\ref{dec-app}) and (\ref{eps-app}), there exists $C > 0$ such that 
	\begin{equation}
	\label{Linfty}
	\| u \|_{L^{\infty}} + \| v \|_{L^{\infty}} \leq C, \quad \| \eta \|_{L^{\infty}} \leq C \epsilon.
	\end{equation}
\end{lemma}

\begin{proof}
	Since 
	$$
	(u + {\rm Re}(U_{c,\omega}))^2 + (v + {\rm Im}(U_{c,\omega}))^2 = |U_{c,\omega}|^2 + \eta, 
	$$
there exists $C > 0$ such that 
\begin{equation}
\label{inter-bound}
\| u \|_{L^{\infty}} + \| v \|_{L^{\infty}} \leq C (1 + \| \eta \|_{L^{\infty}})
\end{equation}
Using the bound 
$$
\| \eta \|^2_{L^{\infty}} \leq \| \eta \|_{L^2} \| \eta' \|_{L^2},
$$
the product rule for $\eta$ in (\ref{eta}), and the triangle inequality yields
\begin{align*}
\| \eta \|^2_{L^{\infty}} & \leq 2 \| \eta \|_{L^2} \left( 
\| {\rm Re}(\bar{U}_{c,\omega} (u' + i v')) \|_{L^2} + \|  {\rm Re}(\bar{U}'_{c,\omega} (u + i v)) \|_{L^2} 
+ \| u u' \|_{L^2} + \| v v' \|_{L^2} \right) \\
&\leq 
C \| \eta \|_{L^2} \left( 
\| u \|_{\mathcal{H}^1_-} + \| v \|_{\mathcal{H}^1_-} 
+ \| u \|_{L^{\infty}} \| u' \|_{L^2} + \| v \|_{L^{\infty}} \| v' \|_{L^2} \right) \\
&\leq 
C \left( 1 + \| u \|_{L^{\infty}} + \| v \|_{L^{\infty}} \right) \left( 
\| u \|^2_{\mathcal{H}^1_-} + \| v \|^2_{\mathcal{H}^1_-} + \| \eta \|^2_{L^2}\right),
\end{align*}
where we have used the properties of $U_{c,\omega} := U_c(\omega \cdot)$ from 
Theorem \ref{lem-conv-to-0} and the proximity of the norm in $\mathcal{H}_{c,\omega}$ and $\mathcal{H}$ due to Remark \ref{remark-proximity} for $(c,\omega)$ near $(0,1)$. Using the bounds (\ref{eps-app}) and (\ref{inter-bound}) yields 
\begin{align*}
\| \eta \|^2_{L^{\infty}} \leq C (1 + \| \eta \|_{L^{\infty}}) \epsilon^2
\end{align*}
Since $\epsilon > 0$ is small, this is equivalent to $\| \eta \|_{L^{\infty}} \leq C \epsilon$, which yields (\ref{Linfty}) due to (\ref{inter-bound}).
\end{proof}

We are now ready to give the proof of the expansion (\ref{expansion-P}) with the bound (\ref{bound-on-R}). Since $P(\psi)$ is invariant under the two symmetries (\ref{symmetry}), we translate $\psi \to e^{-i \theta} \psi(\cdot - \zeta)$. 
After translation, we fix $R > 0$ and split the integral in $P(\psi)$ on $(-\infty,-R] \cup [-R,R] \cup [R,\infty)$. The outer and inner integrals are treated differently based on the two representations 
(\ref{momentum}) and (\ref{momentum-new}).

\subsection{Outer integrals on $(-\infty,-R] \cup [R,\infty)$}

We use the following expression for this part of $P(\psi)$ denoted by $P_R(\psi)$:
\begin{equation*}
P_R(\psi) := \frac{1}{2i} \int_{(-\infty,-R] \cup [R,\infty)} \frac{(1-|\psi|^2)^2}{|\psi|^2} (\bar{\psi} \psi_x - \bar{\psi}_x \psi) dx.
\end{equation*}
The integrand is nonsingular because $|\psi|^2 \geq |U_{c,\omega}|^2 - |\eta|$ is bounded away from zero on $(-\infty,-R] \cup [R,\infty)$ which follows from the boundedness of $|U_{c,\omega}|^2$ away from zero and smallness of $\| \eta \|_{L^{\infty}}$  according to the bound (\ref{Linfty}). We expand
\begin{align*}
\bar{\psi} \psi_x - \bar{\psi}_x \psi = \bar{U}_{c,\omega} U'_{c,\omega} - \bar{U}'_{c,\omega} U_{c,\omega} + \nu + 2i(uv' - u'v), 
\end{align*}
where
\begin{align*}
\nu := 2 i {\rm Im}(U_{c,\omega}') u - 2i {\rm Re}(U_{c,\omega}') v 
- 2i {\rm Im}(U_{c,\omega}) u' + 2 i {\rm Re}(U_{c,\omega}) v',
\end{align*}
so that 
\begin{align*}
P_R(U_{c,\omega} + u + i v) &= P_R(U_{c,\omega}) + P_1 + P_2,
\end{align*}
where
\begin{align*}
P_1 := \frac{1}{2i} \int \left[ \frac{(1-|U_{c,\omega}|^2)^2}{|U_{c,\omega}|^2} \nu - \frac{(1-|U_{c,\omega}|^4)}{|U_{c,\omega}|^4} 
(\bar{U}_{c,\omega} U'_{c,\omega} - \bar{U}'_{c,\omega} U_{c,\omega}) 
\eta \right] dx
\end{align*}
and
\begin{align*}
P_2 &:= \int \frac{(1-|U_{c,\omega}|^2-\eta)^2}{|U_{c,\omega}|^2+\eta}  (uv'-u'v) dx \\
& \quad  + \frac{1}{2i} \int \left[ \frac{(1-|U_{c,\omega}|^2 - \eta)^2}{|U_{c,\omega}|^2+ \eta} - \frac{(1-|U_{c,\omega}|^2)^2}{|U_{c,\omega}|^2} \right] \nu dx, \\
& \quad  + \frac{1}{2i}\int \left[ \frac{(1-|U_{c,\omega}|^2 - \eta)^2}{|U_{c,\omega}|^2+ \eta} - \frac{(1-|U_{c,\omega}|^2)^2}{|U_{c,\omega}|^2} + \frac{(1-|U_{c,\omega}|^4)}{|U_{c,\omega}|^4}  \eta \right]
(\bar{U}_{c,\omega} U'_{c,\omega} - \bar{U}'_{c,\omega} U_{c,\omega}) 
dx, 
\end{align*}
and the integration is performed over $(-\infty,-R] \cup [R,\infty)$. 
Substituting $\nu$ into $P_1$ and integrating $u'$ and $v'$ by parts yield 
\begin{align*}
P_1 &=  \frac{(1-|U_{c,\omega}|^2)^2}{|U_{c,\omega}|^2} 
\left( -{\rm Im}(U_{c,\omega}) u + {\rm Re}(U_{c,\omega}) v \right) 
\left( \biggr|^{x \to +\infty}_{x = R} + \biggr|^{x=-R}_{x \to -\infty} \right) \\
& \quad + 4 \int (1-|U_{c,\omega}|^2) (-{\rm Im}(U'_{c,\omega}) u + {\rm Re}(U'_{c,\omega}) v ) dx \\
& \quad + \frac{i}{2} \int \frac{(1-|U_{c,\omega}|^4)}{|U_{c,\omega}|^4} 
(\bar{U}_{c,\omega} U'_{c,\omega} - \bar{U}'_{c,\omega} U_{c,\omega}) 
(u^2 + v^2) dx,
\end{align*}
where the second term in $P_1$ is obtained as follows:
\begin{align*}
& \qquad 2 \int \frac{(1-|U_{c,\omega}|^2)^2}{|U_{c,\omega}|^2} \left[ 
{\rm Im}(U_{c,\omega}') u - {\rm Re}(U_{c,\omega}') v \right] dx \\
& \quad - 2 \int \frac{(1-|U_{c,\omega}|^4)}{|U_{c,\omega}|^4} \left[ 
{\rm Im}(U_{c,\omega}) u - {\rm Re}(U_{c,\omega}) v \right] 
\left[ {\rm Re}(U_{c,\omega}){\rm Re}(U'_{c,\omega}) + 
{\rm Im}(U_{c,\omega}) {\rm Im}(U'_{c,\omega}) \right]dx \\
& \quad + 2 \int \frac{(1-|U_{c,\omega}|^4)}{|U_{c,\omega}|^4} 
\left[ {\rm Re}(U_{c,\omega}) u + {\rm Im}(U_{c,\omega}) v \right] 
\left[ {\rm Im}(U_{c,\omega}){\rm Re}(U'_{c,\omega}) - 
{\rm Re}(U_{c,\omega}) {\rm Im}(U'_{c,\omega}) \right] dx \\
& = 4 \int (1-|U_{c,\omega}|^2) (-{\rm Im}(U'_{c,\omega}) u + {\rm Re}(U'_{c,\omega}) v ) dx
\end{align*}

The first term in $P_1$ is zero in the limits $x \to \pm \infty$ 
due the boundnedness of $u$ and $v$ in (\ref{Linfty}). The limits $x = \pm R$ are combined with the inner intergrals on $[-R,R]$. The second term in $P_1$ combined with a similar term on $[-R,R]$ is zero due to the second and fourth orthogonality conditions (\ref{constraints-with-c}). The third term in $P_1$ 
is bounded by $\| u \|^2_{\mathcal{H}} + \| v \|^2_{\mathcal{H}}$ due to proximity of norms in $\mathcal{H}_{c,\omega}$ and $\mathcal{H}$ given by 
(\ref{proximity}) for $(c,\omega)$ near $(0,1)$ according to (\ref{eps-app}).

The terms in $P_2$ are analyzed by using the facts that $|U_{c,\omega}|^2$ is bounded away from zero on $(-\infty,-R] \cup [R,\infty)$, $\| \eta \|_{L^{\infty}}$ is small, and the perturbation $u + iv$ is bounded according to (\ref{Linfty}). The first term in $P_2$ is estimated by using Cauchy--Schwarz inequality:
\begin{align*}
& \int \frac{(1-|U_{c,\omega}|^2-\eta)^2}{|U_{c,\omega}|^2+\eta}  (uv'-u'v) dx 
\leq C \int (1-|U_{c,\omega}|^2-\eta)^2  (|u| |v'| + |u'| |v|) dx \\
& \quad \leq C \left( 
(1 + \| \eta \|_{L^{\infty}}) \| u + i v \|_{\mathcal{H}_{c,\omega}} \| u' + i v' \|_{L^2} + \| \eta \|_{L^{\infty}} \| u + iv \|_{L^{\infty}} \| u' + i v'\|_{L^2} \| \eta \|_{L^2} \right) \\
& \quad \leq C \left( \| u' + i v' \|^2_{L^2} + \| u + i v \|^2_{\mathcal{H}} + \| \eta \|^2_{L^2} \right),
\end{align*}
where bounds (\ref{Linfty}) 
and the proximity between the norms in $\mathcal{H}_{c,\omega}$ and $\mathcal{H}$ (Remark \ref{remark-norms}) have been used.
The second and third terms in $P_2$ are estimated similarly:
\begin{align*}
& \int \frac{\eta^2 |U_{c,\omega}|^2 - \eta (1-|U_{c,\omega}|^4)}{|U_{c,\omega}|^2 (|U_{c,\omega}|^2+ \eta)} 
\left( {\rm Im}(U_{c,\omega}') u - {\rm Re}(U_{c,\omega}') v 
- {\rm Im}(U_{c,\omega}) u' + {\rm Re}(U_{c,\omega}) v'\right) dx \\
& \quad \leq C \left( \| u' + i v' \|^2_{L^2} + \| u + i v \|^2_{\mathcal{H}} + \| \eta \|^2_{L^2} \right)
\end{align*}
and
\begin{align*}
\frac{1}{2i}\int \frac{\eta^2}{|U_{c,\omega}|^4 (|U_{c,\omega}|^2+ \eta)}  
(\bar{U}_{c,\omega} U'_{c,\omega} - \bar{U}'_{c,\omega} U_{c,\omega}) dx 
 \leq C \| \eta \|^2_{L^2}.
\end{align*}

\subsection{Inner integrals on $[-R,R]$}

We use the following expression for this part of $P(\psi)$ denoted by $\hat{P}_R(\psi)$:
\begin{equation*}
\hat{P}_R(\psi) := \frac{i}{2} \int_{[-R,R]} (2-|\psi|^2) (\bar{\psi} \psi_x - \bar{\psi}_x \psi) dx + \theta(R) - \theta(-R),
\end{equation*}
where $\theta(x) = \arg(\psi(x))$. Expanding $\psi = U_{c,\omega} + u + iv$ 
and $\theta = \arg(U_{c,\omega}) + \theta_1 + \theta_2$ with 
\begin{align*}
\theta_1 &= \frac{{\rm Re}(U_{c,\omega}) v - {\rm Im}(U_{c,\omega}) u}{|U_{c,\omega}|^2}, \\
\theta_2 &= \arg(U_{c,\omega}+u+iv) - \arg(U_{c,\omega}) - \frac{{\rm Re}(U_{c,\omega}) v - {\rm Im}(U_{c,\omega}) u}{|U_{c,\omega}|^2}
\end{align*}
yields the expansion
\begin{align*}
\hat{P}_R(U_{c,\omega} + u + i v) &= \hat{P}_R(U_{c,\omega}) + \hat{P}_1 + \hat{P}_2,
\end{align*}
where
\begin{align*}
\hat{P}_1 := \frac{i}{2} \int \left[ (2-|U_{c,\omega}|^2) \nu - 
(\bar{U}_{c,\omega} U'_{c,\omega} - \bar{U}'_{c,\omega} U_{c,\omega}) 
\eta \right] dx + \theta_1(R) - \theta_1(-R)
\end{align*}
and
\begin{align*}
\hat{P}_2 := -\int (2-|U_{c,\omega}|^2)  (uv'-u'v) dx + \frac{1}{2i} \int \eta \nu dx + \theta_2(R) - \theta_2(-R), 
\end{align*}
and the integration is performed over $[-R,R]$. Substituting $\nu$ into $\hat{P}_1$ and integrating $u'$ and $v'$ by parts yield 
after simplifications:
\begin{align*}
\hat{P}_1 &=  (2-|U_{c,\omega}|^2)
\left( {\rm Im}(U_{c,\omega}) u - {\rm Re}(U_{c,\omega}) v \right) 
\biggr|^{x = R}_{x = -R}  \\
& \quad + 4 \int (1-|U_{c,\omega}|^2) (-{\rm Im}(U'_{c,\omega}) u + {\rm Re}(U'_{c,\omega}) v ) dx \\
& \quad + \frac{1}{2i} \int  (\bar{U}_{c,\omega} U'_{c,\omega} - \bar{U}'_{c,\omega} U_{c,\omega}) 
(u^2 + v^2) dx + \theta_1(R) - \theta_1(-R).
\end{align*}
The quadratic term in $\hat{P}_1$ is bounded by $\| u \|^2_{\mathcal{H}} + \| v \|^2_{\mathcal{H}}$ because there is $C_R > 0$ such that 
$$
\int_{-R}^R  \left| \bar{U}_{c,\omega} U'_{c,\omega} - \bar{U}'_{c,\omega} U_{c,\omega} \right|
(u^2 + v^2) dx  \leq C_R \int_{\R} (1-\varphi^2) \left| \bar{U}_{c,\omega} U'_{c,\omega} - \bar{U}'_{c,\omega} U_{c,\omega} \right|
(u^2 + v^2) dx,
$$
and $U_{c,\omega}, U_{c,\omega}' \in L^{\infty}(\R)$. 
The linear integral term in $\hat{P}_1$ is combined together with 
a similar term in $P_1$ yields 
$$
4 \int_{\R} (1-|U_{c,\omega}|^2) (-{\rm Im}(U'_{c,\omega}) u + {\rm Re}(U'_{c,\omega}) v ) dx,
$$
which vanish due to the second and fourth orthogonality conditions (\ref{constraints-with-c}). Finally, the terms at $x = \pm R$ 
are combined together in $P_1$ and $\hat{P}_1$ to yield
\begin{align*}
\frac{{\rm Im}(U_{c,\omega}) u - {\rm Re}(U_{c,\omega}) v}{|U_{c,\omega}|^2}
\biggr|^{x = R}_{x = -R} + \theta_1(R) - \theta_1(-R) = 0.
\end{align*}
All integral terms in $\hat{P}_2$ are bounded by 
$\| u' + i v' \|^2_{L^2} + \| u + i v \|^2_{\mathcal{H}} + \| \eta \|^2_{L^2}$ because integration on $[-R,R]$ is bounded by integration on $\R$ 
with the weight $(1-\varphi^2)$. Similarly, by using the Sobolev embedding 
of $H^1([-R,R])$ to $L^{\infty}([-R,R])$ and the Taylor expansion with $|U_{c,\omega}(\pm R)|^2$ being bounded away from zero, we obtain 
$$
|\theta_2(\pm R)| \leq C_R (|u(\pm R)|^2 + |v(\pm R)|^2) 
\leq C_R \| u + i v \|^2_{H^1([-R,R])},
$$
which is bounded by 
$\| u' + i v' \|^2_{L^2} + \| u + i v \|^2_{\mathcal{H}}$ since $[-R,R]$ is compact and $1 - \varphi^2$ is bounded away from zero on $[-R,R]$.

\end{document}